\UseRawInputEncoding
\documentclass[11pt,reqno]{amsart}
\usepackage{amsmath, amssymb, amsthm, esint, verbatim, hyperref, accents}
\usepackage{times}
\usepackage{dsfont}

\usepackage{wrapfig}
\usepackage{tikz}
\usetikzlibrary{decorations.fractals}

\numberwithin{equation}{section}

\hypersetup{
  pdftitle={Solving the Dirichlet problem with prescribed density},
  pdfauthor={Ling Xiao},
  pdfsubject={},
  pdfkeywords={},
  pdfpagelayout=SinglePage,
  pdfpagemode=UseOutlines,
  colorlinks,
  bookmarksopen,
  linkcolor=[rgb]{0,0,0.7},
  urlcolor=[rgb]{0,0,0.4},
  citecolor=[rgb]{0.4,0.1,0},
}

\newtheorem{theorem}{Theorem}[section]

\newtheorem{proposition}[theorem]{Proposition}

\newtheorem{lemma}[theorem]{Lemma}
\newtheorem{corollary}[theorem]{Corollary}
\theoremstyle{definition}
\newtheorem{definition}[theorem]{Definition}
\theoremstyle{remark}
\newtheorem{remark}[theorem]{Remark}
\theoremstyle{remark}

\theoremstyle{remark}

\theoremstyle{remark}

\theoremstyle{remark}

\newcommand{\be}{\begin{equation}}
\newcommand{\ee}{\end{equation}}
\newcommand{\ju}[2]{\begin{array}{#1}#2\end{array}}

\newcommand{\ol}{\overline}

\newcommand{\ubar}[1]{\underaccent{\bar}{#1}}
\newcommand*\Laplace{\mathop{}\!\mathbin\bigtriangleup}
\newcommand{\ssubset}{\subset\joinrel\subset}

\newcommand{\lt}{\left}
\newcommand{\rt}{\right}

\newcommand{\goto}{\rightarrow}
\newcommand{\R}{\mathbb{R}}
\newcommand{\td}{\tilde}

\newcommand{\e}{\epsilon}
\newcommand{\s}{\sigma}
\newcommand{\p}{\partial}
\newcommand{\al}{\alpha}

\newcommand{\fb}{\mathfrak b}
\newcommand{\la}{\lambda}
\newcommand{\ev}{\epsilon_\varphi}
\newcommand{\ba}{\mathbf a}
\newcommand{\bb}{\mathbf b}
\newcommand{\ure}{u^{R, \epsilon}}
\newcommand{\hure}{\hat{u}^{R, \epsilon}}
\newcommand{\bd}{\boldsymbol{\delta}}
\newcommand{\bde}{\boldsymbol{\eta}}

\newcommand{\df}{\dot f}
\newcommand{\ga}{\gamma}

\newcommand{\hx}{\hat x}

\newcommand{\C}{\clubsuit}
\newcommand{\G}{\sigma_{n-1}}

\title{Solving the Dirichlet problem with prescribed density}
\author{Ling Xiao}


\address{Department of Mathematics, University of Connecticut, Storrs, CT 06269}
\email{ling.2.xiao@uconn.edu}

\begin{document}
\begin{abstract}In this paper we prove the following result. Let $\Omega\subset\mathbb R^n, n\geq 3,$ be a bounded, strictly convex, smooth domain and $\varphi: \p\Omega\goto\mathbb R$ be a smooth function. Then for any $z\in\Omega,$
there exists $c_1=c_1(\Omega, \{z\}, n, \varphi)>0,$ such that if $c\geq c_1$ then the problem: $\s_{n-1}(D^2 u)=0$ in $\bar\Omega\setminus\{z\},$ $u|_{\p\Omega}=\varphi,$
$\lim\limits_{r\goto0}\sup\limits_{B_r(z)}\frac{u(x)-u(z)}{|x-z|^{(n-2)/(n-1)}}=c,$ admits a smooth solution in $\bar\Omega\setminus\{z\}.$ Moreover, we obtain the optimal a priori estimates for this solution. In particular, we show for any  integer $m\geq 0$ there exists $C_m=C_m(m, \Omega, \{z\}, n, \varphi, c)>0$ such that $|D^m u(x)|<C_m|x-z|^{\frac{n-2}{n-1}-m}.$ This work provides the first result demonstrating the existence of smooth solutions to the Dirichlet problem for fully nonlinear elliptic equations with prescribed density in Euclidean space. Previously, only continuous solutions were obtained.
\end{abstract}
\maketitle
\tableofcontents

\section{Introduction}
\label{int}
In this paper, we study the degenerate Dirichlet problem with prescribed density (DPPD) at a given point $z\in \Omega$ where $\Omega\subset\R^n, n\geq 3,$ is a bounded, strictly convex, smooth domain.

To present our result, we first introduce some notation. We shall use $D^2u$ to denote the Hessian of $u,$ and the $k$-th elementary symmetric function
$\s_k(A)$ of a symmetric matrix $A$ is defined by
\[\s_k(A)=\s_k(\la(A))=\sum\limits_{1\leq i_1<\cdots<i_k\leq n}\la_{i_1}\cdots\la_{i_k},\]
where $\la(A)=(\la_1, \cdots, \la_n)$ are the eigenvalues of $A.$
We also denote $\Gamma_k$ as the G{\aa}rding's cone
\[\Gamma_k=\{ \la\in \R^{n-1} | \s_m(\la) > 0, m = 1,\cdots , k\}.\]
Let $z\in \Omega$ be an arbitrary point. Our goal is to find a smooth function $u: \bar\Omega\setminus\{z\}\subset\R^n\goto\R$ such that
\be\label{int1}
F(D^2u):=\frac{\s_{n-1}}{\s_{n-2}}(D^2u)=0\,\,\mbox{in $\Omega\setminus\{z\}$}.
\ee
We shall need following definitions.

\begin{definition}
\label{def-admissible}
For any open set $U\subset\R^n$, a function $v\in C^{2}(U)$ is said to be {\it\textbf{ admissible}} if $\la(D^2v)\in\bar\Gamma_{n-1}$ and the matrix
$(F^{ij}|v)$ is positive definite for all $x\in U.$ Here, $F^{ij}|_v:=\frac{\partial F(D^2v)}{\partial v_{ij}}.$
\end{definition}

\begin{definition}
\label{def-density}
Let $u$ be an admissible solution of equation \eqref{int1}, the {\it\textbf{density of $u$ at $z$}}, denoted by $\Theta(u, z),$ is defined by
\[\Theta(u, z):=\lim\limits_{r\goto0}\sup\limits_{B_r(z)}\frac{u(x)-u(z)}{|x-z|^{\frac{n-2}{n-1}}}.\]
\end{definition}

Our main results are the following.
\begin{theorem}
\label{thm-main}
Let $\Omega\subset\R^n, n\geq 3,$ be a bounded, strictly convex, smooth domain,  $\varphi :\p\Omega\goto\R$ be a smooth function, and $z\in\Omega.$ Then the problem
\be\label{eq-main}
\left\{\begin{aligned}
\frac{\s_{n-1}}{\s_{n-2}}(D^2u)&=0\,\,&\mbox{in $\bar{\Omega}\setminus\{z\},$}\\
u&=\varphi\,\,&\mbox{on $\p\Omega,$}\\
\Theta(u, z)&=c,
\end{aligned}
\right.
\ee
admits a smooth admissible solution $u: \bar\Omega\setminus\{z\}\goto\mathbb R,$ provided $c\geq c_1=c_1(\Omega, \{z\}, n, \varphi)>0.$ Moreover, $u$ satisfies,
for any integer $m\geq 0,$
\be\label{eq-main-est}
|D^m u(x)|<C_m|x-z|^{\frac{n-2}{n-1}-m} \,\, \mbox{for all}\,\, x\in\bar{\Omega}\setminus\{z\},
\ee
where $C_m=C_m(m, \Omega, \{z\}, \varphi, n, c)>0.$
\end{theorem}

Consequently, we prove
\begin{corollary}
\label{cor-main}
Let $\Omega\subset\R^n, n\geq 3,$ be a bounded, strictly convex, smooth domain, $\varphi :\p\Omega\goto\R$ be a smooth function, and $z\in\Omega.$ Then the problem
\be\label{eq-main-2}
\left\{\begin{aligned}
\s_{n-1}(D^2u)&=0\,\,&\mbox{in $\bar\Omega\setminus\{z\},$}\\
u&=\varphi\,\,&\mbox{on $\p\Omega,$}\\
\Theta(u, z)&=c,
\end{aligned}
\right.
\ee
admits a smooth admissible solution $u: \bar\Omega\setminus\{z\}\goto\mathbb R,$  provided $c\geq c_1=c_1(\Omega, \{z\}, n, \varphi)>0.$  Moreover, $u$ satisfies,
for any integer $m\geq 0,$
\be\label{eq-main2-est}
|D^m u(x)|<C_m|x-z|^{\frac{n-2}{n-1}-m} \,\, \mbox{for all}\,\, x\in\bar{\Omega}\setminus\{z\},
\ee
where $C_m=C_m(m, \Omega, \{z\}, \varphi, n, c)>0.$
\end{corollary}

\subsection{Related work}
Degenerate equations with prescribed singularities have received extensive study through the years. In \cite{Serrin65}, Serrin studied
the local behavior of the solutions of certain second-order, quasi-linear equations $\text{div}\mathcal A(x,u,u_x)=0$ in the perforated domain $\Omega\setminus\{0\}.$
Later, Kichenassamy-Veron \cite{KV86} continued the work of Serrin and studied the $p$-Laplace equation $\text{div}(|\nabla u|^{p-2}\nabla u)=0$ in $\Omega\setminus\{0\}.$
Weak solutions of various fully nonlinear equations with prescribed singularities have also been studied. For example, the $k$-Hessian equation has been investigated by Trudinger-Wang \cite{TW97, TW99, TW02} and Labutin \cite{La02}. A very general family of fully nonlinear equations has been considered by Labutin \cite{La01} and Harvey-Lawson \cite{HL16}.

In closing of this subsection, we remark that Lempert \cite{Lem83} studied the degenerate complex Monge-Amp\'ere equation with a single singularity. More specifically, he proved the following theorem.

{\it Let $\Omega\subset\mathbb C^n$ be a strictly convex, analytically bounded domain and $\varphi: \p\Omega\goto\mathbb R$ be a real analytic function. Then for any $W\in\Omega,$
there exists $C_0>0,$ such that if $C>C_0$ then the problem: $u$ is plurisubharmonic in $\Omega,$ $\det(\p\bar\p u)=0$ in $\Omega\setminus\{W\},$ $u(\omega)=C\log|\omega-W|+O(1)$
as $\omega\goto W,$ $u|_{\p\Omega}=\varphi,$ admits a unique real analytic solution in $\bar\Omega\setminus\{W\}.$}

 We note that Lempert's proof is based on the explicit formula for the solution, which is a Poisson integral. To prove the above theory using a priori estimates is a longstanding and challenging problem in the field. So far, the best result was obtained by B\l oki in \cite{Blo03}, where he showed the solution is $C^{1, 1}.$

 In an upcoming work \cite{LX27}, we prove an analog of Lampert's result in $\mathbb{R}^4.$ More specifically, we prove the existence and uniqueness of smooth solutions to $\sigma_2=0$ with a prescribed density at a singularity in $\mathbb{R}^4.$ In fact, we find smooth solutions to $\sigma_{n-2}=0$ with a prescribed density at a singularity in $\mathbb{R}^n$ for $n \geq 4.$ The main techniques used in \cite{LX27} are completely different from those in this paper. We do not think the methods used in this paper work for the degenerate $\sigma_{n-2}$ equation in general; however, with significant refinements, they may work for $n = 4, 5$.
\subsection{Comments on the proof}
There are several obstacles that need to be overcome in order to establish Theorem \ref{thm-main}.

First, the existence of a $C^{1,1}$ solution -- that is, a type of ``classical'' solution -- to \eqref{int1} is unknown. Since equation \eqref{int1} is a degenerate equation, the standard way to prove the existence of a $C^{1,1}$ solution, roughly speaking, goes as follows.
Instead of  considering \eqref{int1} directly, we consider
\be\label{int1-appr}
\frac{\s_{n-1}}{\s_{n-2}}(D^2u)=f_k\,\,\mbox{in $\bar\Omega\setminus B_{1/R}(z)$},
\ee
where $f_k$ is a smooth positive function defined in $\bar\Omega\setminus B_{1/R}(z)$ that satisfies $f_k\goto 0$ as $k\goto \infty.$ Denoting the solution of \eqref{int1-appr} by $u^{k, R},$
we want to show that $u^{k, R}\goto u$ as $k, R\goto\infty$ and $u$ is a solution of \eqref{int1}. Note that in the last step, for $u\in C^{1,1}(\bar\Omega\setminus\{z\})$ one needs an interior $C^2$ estimate for
$u^{k,R}.$ To the author's best knowledge, the interior $C^2$ estimates for Hessian quotient equations have only been studied in the case when the solution is convex (see \cite{Lu25, LT25} and references therein). Moreover, the estimates obtained in \cite{Lu25, LT25} depend on the minimum value of the right hand side. However, in our case, the solution of \eqref{int1-appr} is not convex and the right hand side of our equation goes to $0$, that is, $f_k\goto 0.$ Therefore, it is not applicable.

Second, and also the most difficult part of this paper, is to prove that the solution of \eqref{eq-main} is smooth. To keep things simple, let us assume that we already obtained a solution
$u\in C^{1,1}(\bar\Omega\setminus\{z\})$ of \eqref{eq-main}, then we want to prove that this $u$ is smooth. By standard PDE theory, we only need to show that the functional $F$ is uniformly elliptic with respect to $u,$ that is, there exists positive constants $\la, \Lambda$ such that $0\leq\la |\xi|^2\leq F^{ij}|_u\xi_i\xi_j\leq\Lambda |\xi|^2$ for all $\xi\in\mathbb R^n.$ Clearly, the claim that $F$ is uniformly elliptic with respect to $u$ is highly non-trivial and one may even say it is not natural. For instance, we examine the case in which $n=3,$ then $F(D^2u)=\frac{\s_2}{\s_1}=0.$ It is totally possible that the eigenvalues of the matrix $(u_{ij})$ are $1, 0, 0$ at some point, then the eigenvalues of $(F^{ij})$ at this point would be $0, 1, 1,$ which is not uniformly elliptic.
Now, let us denote the eigenvalues of $(u_{ij})$ by $\la_1, \la_2, \la_3$ and assume $\la_1\geq\la_2\geq\la_3.$ We also denote $a:=\frac{\la_2}{\la_1},$ then a direct calculation yields the eigenvalues of $(F^{ij})$ are $\frac{a^2}{1+a+a^2}, \frac{1}{1+a+a^2},$ and $\frac{(1+a)^2}{1+a+a^2}.$ Therefore, when $n=3,$ to prove that $F$ is uniformly elliptic with respect to $u$ we need to show that $a$ is bounded away from $0.$ Similarly, for general $n$, to prove that $F$ is uniformly elliptic with respect to $u$ we need to show that  $\la_{n-1}/\la_1$ is bounded away from $0,$ where $\la_1\geq\cdots\geq\la_{n-1}\geq\la_n$ are eigenvalues of $(u_{ij}).$ We want to point out that, the assumption that $\Omega$ is strictly convex plays an important role here. This assumption implies that $\la_{n-1}/\la_1$ is bounded away from $0$ on $\p\Omega.$ Therefore, to prove that $F$ is uniformly elliptic, we only need to show that $\la_{n-1}/\la_1$ is bounded away from $0$ in the interior.
We also want to mention that we believe the assumption that $\Omega$ is strictly convex is necessary for obtaining smooth solutions (one may compare it with the main theorem of \cite{Lem83}).

Finally, since we prescribed the density at the singularity, we need to be able to control the density at the singularity. In this paper, we use a bootstrap argument to construct the desired solution. More precisely, we show the existence of an admissible function $u_1\in C^{\infty}(\bar\Omega\setminus \{z\})$ that satisfies
the differential equation $\frac{\s_{n-1}}{\s_{n-2}}(D^2u_1)=0$ in $\Omega\setminus\{z\}$ and the boundary condition $u_1=\varphi$ on $\p\Omega$ in Section \ref{solv-appr}, \ref{improve}, \ref{sec-deformation}, \ref{sec-step3}, \ref{sec-AB}, and \ref{sec-smooth}. We can compute the density of this $u_1$ at $z,$ saying $\Theta(u_1, \{z\})=c_1.$ We then show in Section \ref{sec-DPPD} that for any given $c>c_1,$ we can find a smooth solution of \eqref{eq-main}. The idea is to start with $u_1$ and then construct a sequence of smooth functions $\{u_n\}_{n=1}^\infty$ by adjusting the value of $u_n(z)$. By carefully choosing the value of $u_n(z),$ we prove that $\{u_n\}_{n=1}^\infty$ converges to the desired solution $u.$

\subsection{Outline}The organization of the paper is as follows. In Section \ref{np}, we introduce several fundamental formulas and notation.
Since equation \eqref{eq-main} is degenerate, we formulate the approximate equation \eqref{eq-appr} in Section \ref{apr}.
In Section \ref{solv-appr},  we establish the existence of a unique solution $\ure$ to this approximate equation. Section \ref{improve} is devoted to deriving sharp $C^1$ estimates
and sharp $C^2$ boundary estimates for the solution $\ure$. These estimates are then employed in Sections \ref{sec-deformation}, \ref{sec-step3}, and \ref{sec-AB} to obtain sharp global $C^2$ estimates. Finally, in Sections \ref{sec-smooth} and \ref{sec-DPPD}, we demonstrate that these sharp estimates yield the existence of the desired smooth solution to equation \eqref{eq-main}.

 \section{Notations and preliminaries}
 \label{np}
 We collect some important background results necessary for the subsequent sections.

Let $f: \mathbb R^n\goto\mathbb R$ be a smooth, symmetric function defined on an open, convex, symmetric cone $\Gamma.$ Denote
\[\mathcal{S}_{\Gamma}:=\{A\in \text{Sym}_n: \la(A)\in\Gamma\}.\]
A result of Glaeser \cite{Gla63} implies that there is a smooth, $GL(n)$
invariant function $F: \mathcal S_\Gamma\goto\mathbb R$ such that $f(\la(A))=F(A),$ where $\la(A)=(\la_1, \cdots, \la_n)$ are the eigenvalues of $A.$
We shall denote
\[F^{ij}(A):=\frac{\p F}{\p a_{ij}}\,\,\mbox{ and}\,\, F^{pq, rs}(A):=\frac{\p^2F}{\p a_{pq}\p a_{rs}},\]
for $A=(a_{ij}).$ We shall also use dots to indicate derivatives of $f$:
\[\df^i(\la):=\frac{\p f}{\p\la_i}\,\,\mbox{and}\,\,\ddot{f}^{pq}(\la):=\frac{\p^2f}{\p\la_p\p\la_q}.\]
The following lemma is well known (see \cite{ALM14} for example).
\begin{lemma}
\label{np-lem1}
Let $f\in C^{\infty}(\Gamma)$ for some connected, open, symmetric cone $\Gamma\in\mathbb R^n.$
Define the function $F: \mathcal S_\Gamma\goto\mathbb R$
by $F(A):=f(\la(A))$ as above. Then if $A\in\mathcal S_\Gamma$ is diagonal and $B\in\text{Sym}_n$ we have
\be\label{np-1}
F^{ij}(A)=\df^i(\la(A))\delta_{ij},
\ee
and
\be\label{np-2}
F^{pq, rs}(A)B_{pq}B_{rs}=\ddot{f}^{pq}(\la(A))B_{pp}B_{qq}+\sum_{p\neq q}\frac{\df^p(\la(A))-\df^q(\la(A))}{\la_p(A)-\la_q(A)}(B_{pq})^2.
\ee
Note that \eqref{np-2} holds (as a limit) even if $A$ has eigenvalues of multiplicity greater than one.
\end{lemma}
In particular, when $F(D^2u)=\s_k(D^2u)$ we have
\begin{lemma}
\label{np-lem2}Let $n\geq2, 1<k\leq n,$ and let $D^2u\in \Gamma_k.$ Suppose $D^2u$ is diagonalized
at $x_0.$ Then at $x_0,$ we have
\[\s_k^{pq}=\s_k^{pp}\delta_{pq}=\s_{k-1}(\la|p)\delta_{pq}\]
and
\[
\s_{k}^{pq, rs}=\left\{\begin{aligned}&\s_k^{pp, rr}=\s_{k-2}(\la|pr)\,\,&p=q, r=s, p\neq r\\
&\s_k^{pq, qp}=\frac{\s_k^{pp}-\s_k^{qq}}{\la_p-\la_q}=-\s_k^{pp,qq}\,\,&p=s, q=r,  p\neq q\\
&0,\,\,&\mbox{Otherwise.}\end{aligned}
\right.
\]
Here, $\la=(\la_1, \cdots, \la_n )$ are the eigenvalues of $D^2u$ and $(\la|i)$ is the vector obtained by deleting the $i$-th component of the vector $\la.$
\end{lemma}
We also want to mention that in this paper, we shall use $C(m_1, m_2, \cdots, m_k)$ to denote a constant that only depends on $m_1, m_2, \cdots, m_k.$ The value of
this constant, however, may vary from line to line.

\section{The approximate problem}
\label{apr}
Equation \eqref{eq-main} is a degenerate Dirichlet problem defined on $\bar\Omega\setminus\{z\}$. It is natural to approach it using a sequence of non-degenerate Dirichlet problems defined on perforated domains. To find a good approximation, we start with the Riesz kernel of equation \eqref{int1}.

\subsection{Riesz kernels}
\label{subsection-rk}
It is well known that the classical Riesz kernel is crucial for the study of isolated singularities of degenerate equations. In our case, it is easy to check that
$$\mu(x)=|x|^{\frac{n-2}{n-1}}$$ is the Riesz kernel of equation \eqref{int1} whose Riesz characteristic $p=\frac{n}{n-1}.$ Moreover, a straightforward calculation yields the eigenvalues of $D^2\mu(x)$ are $\la_1=\cdots=\la_{n-1}=\frac{n-2}{n-1}r^{-\frac{n}{n-1}}$ and $\la_n=-\frac{n-2}{(n-1)^2}r^{-\frac{n}{n-1}}.$ Here and throughout this paper we denote $r:=|x|.$ Note also that, the eigenvalues of $(F^{ij}|_{\mu})$ are $F^{11}=\cdots=F^{n-1n-1}=\frac{2}{n(n-1)}$ and $F^{nn}=\frac{2(n-1)}{n}.$ Therefore, $F$ is uniformly elliptic with respect to $\mu$.

Inspired by the Riesz kernel, we consider
\[\phi(x, \eta)=(|x|+\eta)^{\frac{n-2}{n-1}}\,\,\mbox{for any $\eta>0.$}\]
It is easy to check that the eigenvalues of $D^2\phi$ are
$$\la(D^2\phi)=\lt(\frac{n-2}{n-1}r^{-1}(r+\eta)^{-\frac{1}{n-1}}, \cdots, \frac{n-2}{n-1}r^{-1}(r+\eta)^{-\frac{1}{n-1}}, \frac{-(n-2)}{(n-1)^2}(r+\eta)^{-\frac{n}{n-1}} \rt).$$
Therefore,
\be\label{def-f-eta}F(D^2\phi)=\frac{\s_{n-1}}{\s_{n-2}}(D^2\phi)=\frac{\eta}{\lt(\frac{n-1}{n-2}\rt)r(r+\eta)^{\frac{1}{n-1}}\lt[\frac{n}{2}r+(n-1)\eta\rt]}=:f_\eta.\ee

\subsection{The approximate problem}
\label{subsection-appr}
Without loss of generality, in this paper we shall always assume the singular point $z$ to be the origin of $\R^n$ and $B_1(0)\ssubset\Omega.$ Moreover, for technical reasons, instead of solving \eqref{eq-main}, we shall consider the following Dirichlet problem with a prescribed value at the singularity.
\be\label{eq-main1}
\left\{\begin{aligned}
\frac{\s_{n-1}}{\s_{n-2}}(D^2u)&=0\,\,&\mbox{in $\bar\Omega\setminus\{0\},$}\\
u&=1+\varphi\,\,&\mbox{on $\p\Omega,$}\\
u(0)&=0.
\end{aligned}
\right.
\ee
From now on, we shall assume $\varphi\leq 0$ to be a smooth function defined on $\bar\Omega$
that satisfies $\|\varphi\|_{C^2(\bar\Omega)}\leq\e_\varphi.$
Here, $\e_\varphi=\e_\varphi(\Omega, \{0\}, n)>0$ is a sufficiently small positive number that only depends on $\Omega,$ $\{0\},$ and $n.$ Throughout this paper, for our convenience, we shall always assume $\e_\varphi<1/20.$ Intuitively speaking, one can view $\varphi$ as a small perturbation of the constant function $0.$ Of course $\varphi$ can be $0,$ in which case we are looking for the Green's function for the domain $\Omega.$
\begin{remark}
\label{appr-rmk-1}
If $\varphi$ does not satisfy $\varphi\leq 0$ and $\|\varphi\|_{C^2(\bar\Omega)}\leq\e_\varphi,$ we can always consider $\td\varphi:=\frac{\varphi-\max_{\bar\Omega}\varphi}{M}$ for some $M>0$ sufficiently large instead. When $\tilde u$ solves \eqref{eq-main} with boundary data $\td\varphi$ and density $\td c,$ $M\td u+\max_{\bar\Omega}\varphi$ solves \eqref{eq-main} with boundary data $\varphi$ and density $M\td c.$ We also want to point out that the assumption $\varphi\leq 0$ is not necessary. This assumption is made purely for our convenience; it is especially helpful when writing down the argument in Section \ref{sec-DPPD}.
\end{remark}

\begin{remark}
\label{appr-rmk-2}
Let $u$ be a smooth solution of \eqref{eq-main1}, it is not hard to show that the density of $u$ at $0$ is well defined (see \cite{HL18}), denote
$c_1:=\Theta(u, \{0\})>0.$ Then $u-1$ solves equation $\eqref{eq-main}$ with the prescribed density equals $c_1.$ Theorem \ref{thm-main} asserts that for any prescribed density $c\geq c_1,$ one can find a smooth solution of equation \eqref{eq-main}.
\end{remark}

Since \eqref{eq-main1} is a degenerate equation with an isolated singularity at the origin, we shall approximate it with a sequence of non-degenerate equations defined on perforated domains. In particular, we shall consider the following equation.

\be\label{eq-appr}
\left\{\begin{aligned}
F(D^2u)&=f_{\frac{\e}{R}}\,\,&\mbox{in $\overline{\Omega\setminus B_{1/R}},$}\\
u&=1+\varphi\,\,&\mbox{on $\p\Omega,$}\\
u&=\ba R^{-\frac{n-2}{n-1}}\,\,&\mbox{on $\p B_{1/R}.$}
\end{aligned}
\right.
\ee
Here, $f_{\frac{\e}{R}}$ is defined in \eqref{def-f-eta} with $\eta=\frac{\e}{R}$ and $\ba=\ba(\Omega, \{0\}, n)>0$ is a fixed number throughout this paper.
It satisfies $\ba r^{\frac{n-2}{n-1}}>1$ on $\partial\Omega.$ Clearly, the choice of $\ba$ is very flexible and we shall always assume $\ba>1.$ Here and throughout this paper, all bold letters denote fixed numbers that remain constant from line to line.

For our convenience, we shall rescale \eqref{eq-appr} and consider
\be\label{eq-appr-s}
\left\{\begin{aligned}
F(D^2u)&=f_{\e}\,\,&\mbox{in $\overline{\Omega^R\setminus B_1},$}\\
u&=R^\frac{n-2}{n-1}+R^{\frac{n-2}{n-1}}\varphi\lt(\frac{x}{R}\rt)\,\,&\mbox{on $\p\Omega^R,$}\\
u&=\ba \,\,&\mbox{on $\p B_1,$}
\end{aligned}
\right.
\ee
where $f_\e$ is defined in \eqref{def-f-eta} with $\eta=\e$ and $\Omega^R:=R\Omega.$ It is easy to check that, $\ure$ satisfies \eqref{eq-appr} iff its rescaling, denoted by
$\hure(x):=R^{\frac{n-2}{n-1}}\ure\lt(\frac{x}{R}\rt),$ satisfies equation \eqref{eq-appr-s}.

\section{Solvability of \eqref{eq-appr}}
\label{solv-appr}

Recall that the reason we are interested in solving \eqref{eq-appr} is to solve \eqref{eq-main1}. Therefore, we only need to solve \eqref{eq-appr} for sufficiently small $\e$ and sufficiently large $R$. In Section \ref{solv-appr}, \ref{improve}, and \ref{sec-deformation} we will always assume $\e<\eta_0 R^{-8}$ and $R>R_0$ for $\eta_0=\eta_0(\Omega)>0$ sufficiently small and $R_0=R_0(\Omega)>0$ sufficiently large. In particular, for our convenience, we shall choose $\eta_0$ to be way smaller than $\kappa_{\min}(\partial\Omega)$ (i.e., the smallest principal curvature of $\p\Omega$) and $R_0$ to be way larger than the $\text{diam}(\Omega).$ Note that, when $R>0$ is very large, the choice of $\eta_0$ is not important, since $\e<R^{-8}$ would be sufficiently small any way.

In the following, we will follow the idea of \cite{CNS3} and prove the existence of the solution $\ure$ for \eqref{eq-appr}. Keep in mind that the rescaling of $\ure$, which is denoted by $\hure,$ is the solution of \eqref{eq-appr-s}. For our convenience, when we establish estimates on $\p\Omega,$ we consider equation \eqref{eq-appr}; when we establish estimates on $\p B_{1/R},$ we consider equation \eqref{eq-appr-s}. We prove
\begin{theorem}
\label{solve-appr-thm}
Given $\eta_0=\eta_0(\Omega)>0, \ev=\ev(\Omega, \{0\}, n)>0$ sufficiently small, $R_0=R_0(\Omega)>0$ sufficiently large, then for any $0<\e<\eta_0 R^{-8},$ $R>R_0,$ and $\varphi\in C^{\infty}(\bar\Omega)$ that satisfies $\varphi\leq0$ and $\|\varphi\|_{C^2}\leq \ev,$  there is a unique admissible solution
$\ure\in C^{\infty}(\ol{\Omega\setminus B_{1/R}})$
satisfying \eqref{eq-appr}. Moreover, we have
\[\|\ure\|_{C^2}<CR^{3-\frac{1}{n-1}}.\]
Analogously,
for any $0<\e<\eta_0 R^{-8},$ $R>R_0$ there is a unique admissible solution
$\hure\in C^{\infty}(\ol{\Omega^R\setminus B_1})$
satisfying \eqref{eq-appr-s}. Moreover, we have
\[\|\hure\|_{C^2}<CR^{2-\frac{2}{n-1}}.\]
Here, $C=C(\Omega, \{0\}, n)>0$ is some positive constant that is independent of $R$ and $\e.$
\end{theorem}

\subsection{$C^0$-estimates}
\label{c0-est}
We shall choose a positive constant $\bb=\bb(\Omega, \{0\}, n)>0$ such that $\bb r^{\frac{n-2}{n-1}}<9/10<1-\e_\varphi$ on $\p\Omega$ (used our assumption that $\e_\varphi<1/20$ here). Recall that by our assumption that $B_1(0)\ssubset\Omega,$ we get $\bb<9/10.$ It is also clear that $\ba>\bb.$
As we mentioned before, all bold letters in this paper are fixed constants. So, $\bb$ is a fixed constant throughout this paper.
Now, let $\hat\Psi^{R, \e}:=\bb\lt(r+\frac{2\e}{\bb}\rt)^{\frac{n-2}{n-1}},$ a straightforward calculation gives
\[F(D^2\hat\Psi^{R,\e})=\bb f_{\frac{2\e}{\bb}}=\frac{2\e}{\lt(\frac{n-1}{n-2}\rt)r\lt(r+\frac{2\e}{\bb}\rt)^{\frac{1}{n-1}}\lt[\frac{n}{2}r+(n-1)\frac{2\e}{\bb}\rt]}.\]
Since $\bb<9/10$ is fixed and $\e\ll R^{-1},$ we can see when $R>R_0>0$ with $R_0$ being sufficiently large, we have
\be\label{rangef}
2f_\e>F(D^2\hat\Psi^{R,\e})>f_\e\,\,\mbox{in $\Omega^R\setminus B_1,$}
\ee
 $$\hat\Psi^{R, \e}=\bb\lt(1+\frac{2\e}{\bb}\rt)^{\frac{n-2}{n-1}}<\ba\,\,\mbox{on $\p B_1,$}$$
and
$$\hat\Psi^{R, \e}=\bb\lt(R+\frac{2\e}{\bb}\rt)^{\frac{n-2}{n-1}}<\frac{19}{20}R^{\frac{n-2}{n-1}}<R^{\frac{n-2}{n-1}}+R^{\frac{n-2}{n-1}}\varphi\lt(\frac{x}{R}\rt)\,\,\mbox{on $\p\Omega^R$}.$$
Therefore, $\hat\Psi^{R, \e}$ is a subsolution of \eqref{eq-appr-s}. Consequently, $\Psi^{R, \e}(x)=R^{-\frac{n-2}{n-1}}\hat\Psi^{R, \e}(Rx)=\bb\lt(r+\frac{2\e}{\bb R}\rt)^{\frac{n-2}{n-1}}$
is a subsolution of \eqref{eq-appr}. On the other hand, it is easy to see that the function $$\bar\Psi:=\ba r^{\frac{n-2}{n-1}}$$ is an admissible supersolution to both equations \eqref{eq-appr}
and \eqref{eq-appr-s}. The standard maximum principle implies
$$\bb\lt(r+\frac{2\e}{\bb R}\rt)^{\frac{n-2}{n-1}}<\ure\leq\ba r^{\frac{n-2}{n-1}}\,\,\mbox{on $\ol{\Omega\setminus B_{1/R}}.$}$$
We conclude
\begin{proposition}
\label{c0-prop}
Given $\eta_0=\eta_0(\Omega)>0, \ev=\ev(\Omega, \{0\}, n)>0$ sufficiently small, $R_0=R_0(\Omega)>0$ sufficiently large, for any $0<\e<\eta_0 R^{-8},$ $R>R_0,$ and $\varphi\in C^{\infty}(\bar\Omega)$ that satisfies $\varphi\leq0$ and $\|\varphi\|_{C^2}\leq \ev$ let $\ure$ be the admissible solution of \eqref{eq-appr}. Then $\ure$ satisfies
\be\label{c0}
\bb r^{\frac{n-2}{n-1}}<\ure\leq\ba r^{\frac{n-2}{n-1}}\,\,\mbox{on $\ol{\Omega\setminus B_{1/R}}.$}
\ee
Let $\hure$ be the admissible solution of \eqref{eq-appr-s}, then $\hure$ satisfies
\be\label{c0-s}
\bb r^{\frac{n-2}{n-1}}<\hure\leq\ba r^{\frac{n-2}{n-1}}\,\,\mbox{on $\ol{\Omega^R\setminus B_1}.$}
\ee
Here $\ba=\ba(\Omega, \{0\}, n), \bb=\bb(\Omega, \{0\}, n)>0$ are positive constants that only depend on $\Omega,$ $\{0\},$ and $n.$
\end{proposition}

\subsection{$C^1$-estimates}
\label{c1-est}
We shall start this subsection by studying gradient estimates on the inside boundary.
\subsubsection{$C^1$ estimates on $\p B_{1/R}$}
\label{subsub-c1-inside}
\begin{lemma}
\label{c1-inside-lem}
Given $\eta_0=\eta_0(\Omega)>0, \ev=\ev(\Omega, \{0\}, n)>0$ sufficiently small, $R_0=R_0(\Omega)>0$ sufficiently large, for any $0<\e<\eta_0 R^{-8},$ $R>R_0,$ and $\varphi\in C^{\infty}(\bar\Omega)$ that satisfies $\varphi\leq0$ and $\|\varphi\|_{C^2}\leq \ev$
let $\ure$ be the admissible solution of \eqref{eq-appr}. Then on $\p B_{1/R}$ we have
\be\label{c1-inside}
\frac{n-2}{n-1}\bb\lt(1+\frac{2\e}{\bb}\rt)^{-\frac{1}{n-1}} R^{\frac{1}{n-1}}<\frac{\p\ure}{\p\nu}<\frac{n-2}{n-1}\ba R^{\frac{1}{n-1}}.
\ee
Let $\hure$ be the admissible solution of \eqref{eq-appr-s}, then on $\p B_1$ we have
\be\label{c1-inside-s}
\frac{n-2}{n-1}\bb\lt(1+\frac{2\e}{\bb}\rt)^{-\frac{1}{n-1}} <\frac{\p\hure}{\p\nu}<\frac{n-2}{n-1}\ba.
\ee
Here, $\nu$ is the inward unit normal (i.e., pointing away from the origin).
\end{lemma}
\begin{proof}
We shall consider equation \eqref{eq-appr-s} and prove inequality \eqref{c1-inside-s}. Inequality \eqref{c1-inside} then follows by rescaling.
Consider
$$\ubar\phi^{R, \e}=\bb\lt(r+\frac{2\e}{\bb}\rt)^{\frac{n-2}{n-1}}-\bb\lt(1+\frac{2\e}{\bb}\rt)^{\frac{n-2}{n-1}}+\ba.$$
It is clear that on $\p B_1$ we have
\[\ubar\phi^{R, \e}=\ba,\] and
on $\p\Omega^R$ when $R>R_0>0$ with $R_0$ being sufficiently large we have
$$\ubar\phi^{R, \e}<\frac{19}{20}R^{\frac{n-2}{n-1}}<R^\frac{n-2}{n-1}+R^{\frac{n-2}{n-1}}\varphi\lt(\frac{x}{R}\rt).$$
Moreover, recall \eqref{rangef} we have $F(D^2\ubar\phi^{R, \e})=\bb f_{\frac{2\e}{\bb}}>f_\e,$ the standard strong maximum principle then implies
\[\ubar\phi^{R, \e}<\hure<\bar\Psi\,\,\mbox{in $\Omega^R\setminus \bar B_1,$}\]
where $\bar\Psi=\ba r^{\frac{n-2}{n-1}}$ is defined in Subsection \ref{c0-est}. Since $\ubar\phi^{R, \e}=\hure=\bar\Psi=\ba$ on $\p B_1,$ \eqref{c1-inside-s} follows
immediately. In view of the relation $\ure(x)=R^{-\frac{n-2}{n-1}}\hure(Rx),$ we obtain \eqref{c1-inside}.
\end{proof}

\subsubsection{$C^1$ estimates on $\p\Omega$}
\label{subsub-c1-outside}
To establish $C^1$ gradient estimates on the outside boundary $\p\Omega,$ we consider the original equation \eqref{eq-appr}. Note that, since in general $\p\Omega$ is not a sphere and the boundary value is not constant, the construction of barriers on $\p\Omega$ is more complicated.

We want to recall a well known fact: a strictly convex domain $\Omega$ is strictly star-shaped with respect to any point $p\in\Omega.$
Therefore, by Appendix \ref{app}, there exists a regularized distance defined in $\Omega.$ We will use this regularized distance to construct barriers on $\p\Omega.$ In particular, since $\Omega$ is strictly star-shaped with respect to the origin, we may parameterize $\p\Omega$ as a graph of radial function $\rho(\theta): \mathbb S^{n-1}\goto\R,$ i.e.,
$\p\Omega:=\{\rho(\theta)\theta: \theta\in\mathbb S^{n-1}\}.$ Note that since $B_1\ssubset\Omega,$ we have $\rho(\theta)>1.$
Let $\fb:=\frac{r}{\rho(\theta)}$ be a function defined on $\mathbb S^{n-1}\times \R_+,$ as we pointed out in Appendix \ref{app}, $1-\fb$ is a regularized distance function for $\Omega.$ In the following, we will use $\fb$ to construct barriers on $\p\Omega.$

We start with the construction of the lower barrier on $\p\Omega.$ Consider
\[\ubar\phi:=2(\fb-1)+1+\varphi.\]
Since $\fb=\frac{r}{\rho(\theta)}=1$ on $\p\Omega,$ we have \[\ubar\phi=1+\varphi\,\,\mbox{on $\p\Omega.$}\]
On $B_{1/R},$ we have
\begin{align*}
\ubar\phi&=2\lt(\frac{1}{\rho(\theta)R}-1\rt)+1+\varphi\\
&<\frac{2}{R}-1<\ba R^{-\frac{n-2}{n-1}},
\end{align*}
here we have used our assumptions $\rho(\theta)>1$ and $\varphi\leq 0.$
Now, by virtue of \eqref{hessian-b} we know for any $x\in\Omega\setminus B_{1/R},$ we can find a local coordinate system such that at $x=(p, r)=(\frac{x}{|x|}, |x|)\in\mathbb S^{n-1}\times\R_+$
\[
\begin{aligned}
\text{Hessian}(\ubar\phi)&=\left[\ju{ccccc}{\frac{2w^3}{r}a_{11}+\varphi_{11}&\frac{2w^2}{r}a_{12}+\varphi_{12}&\cdots&\frac{2w^2}{r}a_{1n-1}+\varphi_{1n-1}&\varphi_{1n}\\
\frac{2w^2}{r}a_{12}+\varphi_{12}&\frac{2w}{r}a_{22}+\varphi_{22}&\cdots&\varphi_{2n-1}&\varphi_{2n}\\
\vdots&\vdots&\ddots&\vdots\\
\frac{2w^2}{r}a_{1n-1}+\varphi_{1n-1}&\varphi_{2n-1}&\cdots&\frac{2w}{r}a_{n-1n-1}+\varphi_{n-1n-1}&\varphi_{n-1n}\\
\varphi_{1n}&\varphi_{2n}&\cdots&\varphi_{n-1n}&\varphi_{nn}}\right],
\end{aligned}
\]
where the eigenvalues of $(a_{ij})$ are the principal curvatures of $\partial\Omega$ at the point $\rho(p)p$ and
$w=\sqrt{1+\frac{|\nabla\rho(p)|^2}{\rho(p)^2}}.$
Note that for any $n\times n$ matrices $A$ and $B,$ we can express $\s_k(A+B)$ as the sum of all $k\times k$ principal minors of $A+B.$ In view of equation (2.5) in \cite{Tru95},
a straightforward calculation yields for any $x\in\Omega\setminus B_{1/R}$
\be\label{c1-outside-1}
\begin{aligned}
\s_{n-1}(D^2\ubar\phi)&\geq w^2\lt(\frac{2w}{r}\rt)^{n-1}a_{11}\s_{n-2}(a_{\al\beta})-w^2\sum\limits_{s=2}^{n-1}\lt(\frac{2w}{r}\rt)^{n-1}\s_{n-3}(a_{\al\beta}|s)a_{1s}^2
-\sum\limits_{l=1}^{n-1}c_lr^{-(n-1-l)}\ev^{l}\\
&= w^2\lt(\frac{2w}{r}\rt)^{n-1}\s_{n-1}(a_{ij})-\sum\limits_{l=1}^{n-1}c_lr^{-(n-1-l)}\ev^{l}\\
&\geq\frac{c_0}{r^{n-1}}-\sum\limits_{l=1}^{n-1}c_lr^{-(n-1-l)}\ev^{l},
\end{aligned}
\ee
where $2\leq\al, \beta\leq n-1,$ $1\leq i, j\leq n-1,$ $(a_{\al\beta}|s)$ is obtained by deleting the $s^{th}$ row and $s^{th}$ column of the $(n-2)\times (n-2)$ matrix $(a_{\al\beta}),$ and $c_0=c_0(\Omega)>0,$ $c_{n-1}=c_{n-1}(n)>0,$ $c_l=c_l(\Omega, \{0\}, n)>0$ for $1\leq l\leq n-2.$
Therefore, when $\ev=\ev(\Omega, \{0\}, n)>0$ is sufficiently small, we have
\[\s_{n-1}(D^2\ubar\phi)>\frac{c_0}{2r^{n-1}}.\]
Similarly, we can show that when $\ev=\ev(\Omega, \{0\}, n)>0$ is sufficiently small, there exists $a_k=a_k(\Omega)>0$ so that
\[\begin{aligned}
\s_{k}(D^2\ubar\phi)&\geq
\lt(\frac{2w}{r}\rt)^{k}\lt\{w^2\lt[a_{11}\s_{k-1}(a_{\al\beta})-\sum\limits_{s=2}^{n-1}a^2_{1s}\s_{k-2}(a_{\al\beta}|s)\rt]+\s_k(a_{\al\beta})\rt\}
-\sum\limits_{l=1}^{k}c_lr^{-(k-l)}\ev^{l}\\
&\geq \lt(\frac{2w}{r}\rt)^{k}\s_k(a_{ij})-\sum\limits_{l=1}^{k}c_lr^{-(k-l)}\ev^{l}\\
&\geq \frac{a_k}{r^{k}}>0,\,\,\mbox{for all}\,\, 1\leq k\leq n-2,
\end{aligned}\]
 this implies that $\la[D^2\ubar\phi]\in\Gamma_{n-1}.$ Here, to derive the second inequality, we used the fact that $(a_{ij})_{(n-1)\times(n-1)}$ is a positive definite matrix (thus all principal minors are positive) which gives $$a_{11}\s_{k-1}(a_{\al\beta})-\sum\limits_{s=2}^{n-1}a^2_{1s}\s_{k-2}(a_{\al\beta}|s)>0.$$
Moreover, we can bound $\s_{n-2}(D^2\ubar\phi)$ from above
\be\label{c1-outside-2}
\begin{aligned}
0&<\s_{n-2}(D^2\ubar\phi)< w^2\lt(\frac{2w}{r}\rt)^{n-2}\s_{n-2}(a_{ij})+\sum\limits_{l=1}^{n-2}b_lr^{-(n-2-l)}\ev^l\\
&\leq\frac{b_0}{r^{n-2}}+\sum\limits_{l=1}^{n-2}b_lr^{-(n-2-l)}\ev^l,
\end{aligned}
\ee
where $b_0=b_0(\Omega, \{0\})>0,$ $b_{n-2}=b_{n-2}(n)>0,$ and $b_l=b_l(\Omega, \{0\}, n)>0$ for $1\leq l\leq n-3.$ Combining \eqref{c1-outside-1} and \eqref{c1-outside-2} together we get, when $\ev=\ev(\Omega, \{0\}, n)>0$ sufficiently small
\be\label{c1-outside-3}
F(D^2\ubar\phi)>\frac{c_0}{2rb_0+2\sum\limits_{l=1}^{n-2}b_lr^{l+1}\ev^l}>\frac{d_0}{r},
\ee
where, $d_0=d_0(\Omega, \{0\}, n)>0.$

Now, notice that in $\Omega\setminus B_{1/R}$ we have
\[f_{\frac{\e}{R}}<\frac{2(n-2)\e R^{\frac{n}{n-1}}}{n(n-1)}<\eta_0 R^{-6}<\frac{d_0}{r}\]
for all $0<\e<\eta_0 R^{-8}$ and $R>R_0$. Therefore, we conclude that $\ubar\phi$ is a subsolution of \eqref{eq-appr}.

Next, we construct the upper barrier on $\p\Omega.$ Since $\ure$ is the admissible solution of \eqref{eq-appr} we have $\s_1(D^2\ure)>0,$ i.e., $\ure$ is a subharmonic function.
In order to construct
an upper barrier for $\ure,$ we only need to construct a superharmonic function.
By Proposition \ref{c0-prop} we always have $\ure<\ba r^{\frac{n-2}{n-1}}$ in $\Omega\setminus B_{1/R}.$ This gives
\[\ure<\frac{1}{4}\,\,\mbox{on $\p B_{r_0}=\{x: |x|=r_0=\lt(\frac{1}{4\ba}\rt)^{\frac{n-1}{n-2}}$\}}.\]
In view of our assumption that $\ba>1,$ $B_1\ssubset\Omega,$ and $R>R_0$ for some $R_0$ sufficiently large, we get $\p B_{r_0}\subset \Omega\setminus\bar B_{1/R}.$
Now, consider
\[\bar\phi=-\frac{1}{2}e^{-N\fb}+\frac{1}{2}e^{-N}+1+\varphi,\]
where $N=N(\Omega, \{0\}, n)>0$ is a positive constant to be determined. In the following, we will denote
 \[\phi(\fb):=-\frac{1}{2}e^{-N\fb}.\]
 We shall show that for an appropriately chosen $N,$ when $\ev$ is sufficiently small $\bar\phi$ is a superharmonic function
in $\Omega\setminus B_{r_0}$.
Similar to the construction of $\ubar\phi,$  we apply \eqref{hessian-phi} and obtain for any $x\in\Omega\setminus B_{r_0},$ we can find a local coordinate system such that at $x=(p, r)=(\frac{x}{|x|}, |x|)\in\mathbb S^{n-1}\times\R_+$
 \begin{align*}
\text{Hessian}(\phi)&=\left[\ju{ccccc}{\frac{Mw^3}{r}a_{11}+B\rho^{-4}\rho_1^2&\frac{Mw^2}{r}a_{12}&\cdots&\frac{Mw^2}{r}a_{1n-1}&-B\rho^{-3}\rho_1\\
\frac{Mw^2}{r}a_{12}&\frac{Mw}{r}a_{22}&\cdots&0&0\\
\vdots&\vdots&\ddots&\vdots\\
\frac{Mw^2}{r}a_{1n-1}&0&\cdots&\frac{Mw}{r}a_{n-1n-1}&0\\
-B\rho^{-3}\rho_1&0&\cdots&0&B\rho^{-2}}\right],
\end{align*}
where the eigenvalues of $(a_{ij})$ are the principal curvatures of $\partial\Omega$ at the point $\rho(p)p,$ $M=\frac{d\phi}{d\fb}=\frac{1}{2}Ne^{-N\fb},$ $B=\frac{d^2\phi}{d\fb^2}=-\frac{1}{2}N^2e^{-N\fb}=-NM,$ and $\rho_1(p)=|\nabla\rho(p)|.$
Since $w=\sqrt{1+\frac{|\nabla\rho|^2}{\rho^2}},$ we obtain on $\Omega\setminus B_{r_0}$
\begin{align*}
\sigma_1(D^2\phi)&=\frac{Mw^3}{r}a_{11}+\sum\limits_{\al=2}^{n-1}\frac{Mw}{r}a_{\al\al}+B\rho^{-2}w^2\\
&\leq\frac{M}{r_0}w^3\sum\limits_{i=1}^{n-1}a_{ii}-NM\rho^{-2}\\
&<\frac{M}{r_0}C_1-NMC_2,
\end{align*}
where $C_1, C_2>0$ are constants that only depend on $\Omega$ and $\{0\}.$ More specifically, $C_1=C_1(\|\rho\|_{C^2(\mathbb S^{n-1})})$ and $C_2=C_2(\|\rho\|_{C^0(\mathbb S^{n-1})}).$
From now on, we shall fix the constant $N$ to be
\be\label{eq-N-1}
N=\frac{2C_1}{C_2r_0}.
\ee
Then on the set $\Omega\setminus B_{r_0}$ we get
\[\sigma_1(D^2\phi)<-\frac{M}{r_0}C_1<-\frac{C_1^2}{C_2r_0^2}e^{-\frac{2C_1}{C_2r_0}}=:-C_3,\]
for some $C_3=C_3(\Omega, \{0\}, n)>0.$ Here, we have used $\fb<1$ in $\Omega\setminus B_{r_0}.$
When $\e_\varphi=\ev(\Omega, \{0\}, n)>0$ is sufficiently small we have
\[\sigma_1(D^2\bar\phi)<-C_3+n\e_\varphi<0\,\,\mbox{in $\Omega\setminus B_{r_0}.$}\]
On the other hand, it is not hard to check
\[\mbox{on $\p\Omega,$}\,\, \bar\phi=1+\varphi\]
and
\[\mbox{on $\p B_{r_0},$}\,\, \bar\phi>-\frac{1}{2}+1+\varphi>\frac{1}{4}.\]
Since $\sigma_1(D^2\ure)>0$, we conclude that $\bar\phi$ is the desired upper barrier for $\ure$ in $\Omega\setminus B_{r_0}.$

The standard maximum principle then implies
\[\ubar\phi<\ure<\bar\phi\,\,\mbox{in $\Omega\setminus B_{r_0}.$}\]
By virtue of the fact that $\ubar\phi=\ure=\bar\phi=1+\varphi$ on $\p\Omega,$ we obtain
\begin{lemma}
\label{c1-outside-lem}
Given $\eta_0=\eta_0(\Omega)>0, \ev=\ev(\Omega, \{0\}, n)>0$ sufficiently small, $R_0=R_0(\Omega)>0$ sufficiently large, for any $0<\e<\eta_0 R^{-8},$ $R>R_0,$ and $\varphi\in C^{\infty}(\bar\Omega)$ that satisfies $\varphi\leq0$ and $\|\varphi\|_{C^2}\leq \ev$ let $\ure$ be the admissible solution of \eqref{eq-appr}. Then on $\p\Omega$ we have
\be\label{c1-outside}
-d<\frac{\p\ure}{\p\nu}<-c.
\ee
Let $\hure$ be the admissible solution of \eqref{eq-appr-s}, then on $\p \Omega^R$ we have
\be\label{c1-outside-s}
-d R^{-\frac{1}{n-1}} <\frac{\p\hure}{\p\nu}<-c R^{-\frac{1}{n-1}}.
\ee
Here, $d=d(\Omega, \{0\}, n), c=c(\Omega, \{0\}, n)>0$ are positive constants that are independent of $R$ and $\e;$ $\nu$ is the inward unit normal (i.e., pointing to the origin).
\end{lemma}

\subsubsection{$C^1$ global estimates}
\label{subsub-c1-global} Finally, we shall investigate gradient estimates of $\ure$ in $\Omega\setminus B_{1/R}.$
\begin{lemma}
\label{c1-global-lem}
Given $\eta_0=\eta_0(\Omega)>0, \ev=\ev(\Omega, \{0\}, n)>0$ sufficiently small, $R_0=R_0(\Omega)>0$ sufficiently large, for any $0<\e<\eta_0 R^{-8},$ $R>R_0,$ and $\varphi\in C^{\infty}(\bar\Omega)$ that satisfies $\varphi\leq0$ and $\|\varphi\|_{C^2}\leq \ev$ let $\ure$ be the admissible solution of \eqref{eq-appr}. Then in $\ol{\Omega\setminus B_{1/R}}$ we have
\be\label{c1-global}
\max\limits_{x\in\ol{\Omega\setminus B_{1/R}}}\lt\{|D\ure|+4R\ure\rt\}=\max\limits_{x\in\p\Omega\cup \p B_{1/R}}\lt\{|D\ure|+4R\ure\rt\}
\ee
Let $\hure$ be the admissible solution of \eqref{eq-appr-s}, then in $\ol{\Omega^R\setminus B_1}$ we have
\be\label{c1-global-s}
\max\limits_{x\in\ol{\Omega^R\setminus B_1}}\lt\{|D\hure|+4\hure\rt\}=\max\limits_{x\in\p\Omega^R\cup \p B_1}\lt\{|D\hure|+4\hure\rt\}
\ee
\end{lemma}
\begin{proof}
We shall only prove \eqref{c1-global}, \eqref{c1-global-s} can be proved in the same way. Consider
$$W:=\max\limits_{x\in\ol{\Omega\setminus B_{1/R}}, \xi\in\mathbb{S}^{n}}\lt\{D_\xi\ure+4R\ure\rt\}.$$
Suppose $W$ is achieved at an interior point
$(x_0, \xi_0)\in\lt(\Omega\setminus\bar{B}_{1/R}\rt)\times\mathbb{S}^n.$ Rotating the coordinates we may assume at $x_0,$
$$|D\ure(x_0)|=D_{\xi_0}\ure=(\ure)_1,\,\,\mbox{and $(\ure)_{\alpha\beta}=\lambda_\alpha\delta_{\alpha\beta}$ for $2\leq\alpha,\,\, \beta\leq n.$}$$ Then at $x_0,$ a straightforward calculation gives
$(\ure)_{1i}+4R(\ure)_i=0$ for $1\leq i\leq n,$ which implies
$$(\ure)_{1i}=0\,\, \mbox{for $2\leq i\leq n,$ and $(\ure)_{11}=-4R(\ure)_1.$}$$
Therefore, under the chosen coordinates, $(\ure_{ij})$ is diagonalized at $x_0,$ i.e.,
$\ure_{ij}(x_0)=\ure_{ii}(x_0)\delta_{ij}.$
Now, differentiating $f_{\frac{\e}{R}}$ we get
\begin{equation}
\label{fr}
\lt(f_{\frac{\e}{R}}\rt)_i=\frac{\p f_{\frac{\e}{R}}}{\p r}r_i=f_{\frac{\e}{R}}\frac{x_i}{r}\lt(-\frac{1}{r}-\frac{1}{(n-1)(r+\e/R)}-\frac{1}{r+2(n-1)\e/(nR)}\rt),
\end{equation}
Therefore, at $x_0$ we have
\begin{equation}
\label{c1-g-1}
\begin{aligned}
0&\geq F^{ii}(\ure)_{1ii}+4RF^{ii}(\ure)_{ii}\\
&=f_{\frac{\e}{R}}\frac{\lt<x_0, \xi_0\rt>}{r_0}\lt(-\frac{1}{r_0}-\frac{1}{(n-1)(r_0+\e/R)}-\frac{1}{r_0+2(n-1)\e/(nR)}\rt)+4Rf_{\frac{\e}{R}},
\end{aligned}
\end{equation}
where $r_0=|x_0|>1/R$ and $F^{ii}=F^{ii}|_{\ure}.$
It is clear that the right hand side of \eqref{c1-g-1} is positive. This leads to a contradiction. Thus $W$ is achieved on the boundaries. This completes the proof of \eqref{c1-global}.
\end{proof}

Combining Proposition \ref{c0-prop}, Lemma \ref{c1-inside-lem}, and Lemma \ref{c1-outside-lem} with Lemma \ref{c1-global-lem} we obtain
\begin{corollary}
\label{c1-bound-cor}
Given $\eta_0=\eta_0(\Omega)>0, \ev=\ev(\Omega, \{0\}, n)>0$ sufficiently small, $R_0=R_0(\Omega)>0$ sufficiently large, for any $0<\e<\eta_0 R^{-8},$ $R>R_0,$ and $\varphi\in C^{\infty}(\bar\Omega)$ that satisfies $\varphi\leq0$ and $\|\varphi\|_{C^2}\leq \ev$ let $\ure$ be the admissible solution of \eqref{eq-appr}. Then in $\ol{\Omega\setminus B_{1/R}}$
we have
\be\label{c1-bound}
|D\ure|<5R.
\ee
Let $\hure$ be the admissible solution of \eqref{eq-appr-s}, then in $\ol{\Omega^R\setminus B_1}$
we have
\be\label{c1-bound-s}
|D\hure|<5R^{\frac{n-2}{n-1}}.
\ee
\end{corollary}

 \subsection{$C^2$-estimates}
 \label{c2-est}
 We shall follow the same order as in Subsection \ref{c1-est} and study the $C^2$ estimates on the inside boundary first.

 \subsubsection{$C^2$-estimates on $B_{1/R}$}
 \label{subsub-c2-inside}
 In this subsubsection, we shall study the $C^2$ estimates of $\hure$ on $\p B_1,$ and the $C^2$ estimates of $\ure$ on $\p B_{1/R}$ are established by rescaling.

 Let $x_0\in\p B_1$ be an arbitrary point. We may choose a local coordinate system $\{\td x_1, \td x_2, \cdots, \td x_n\}$ in a neighborhood of $x_0$ such that $x_0$ is the
 origin and $\td x_n$ axis is the inward normal of $\p B_1$ (pointing into $\Omega^R\setminus B_1$) at $x_0$. Since
 $\hure\equiv \ba$ on $\p B_1,$  it is trivial to see that at $x_0$
 \be\label{c2-inside-0}
 (\hure)_{\td\alpha\td\beta}=(\hure)_{\td n}\delta_{\td\alpha\td\beta}\,\,\mbox{for $\td\alpha, \td\beta\leq n-1.$}
 \ee
 Applying Lemma \ref{c1-inside-lem} we have
 \begin{lemma}
 \label{c2-inside-tan-lem}
Given $\eta_0=\eta_0(\Omega)>0, \ev=\ev(\Omega, \{0\}, n)>0$ sufficiently small, $R_0=R_0(\Omega)>0$ sufficiently large, for any $0<\e<\eta_0 R^{-8},$ $R>R_0,$ and $\varphi\in C^{\infty}(\bar\Omega)$ that satisfies $\varphi\leq0$ and $\|\varphi\|_{C^2}\leq \ev$ let $\ure$ be the admissible solution of \eqref{eq-appr}. Then on $\p B_{1/R}$ the double tangential derivatives of $\ure$ satisfy
\be\label{c2-inside-tan}
\frac{n-2}{n-1}\bb\lt(1+\frac{2\e}{\bb}\rt)^{-\frac{1}{n-1}}R^{\frac{n}{n-1}}\delta_{\td\alpha\td\beta}\leq(\ure)_{\td\alpha\td\beta}
\leq\frac{n-2}{n-1}\ba R^{\frac{n}{n-1}}\delta_{\td\alpha\td\beta}.
\ee
Let $\hure$ be the admissible solution of \eqref{eq-appr-s}, then on $\p B_1$
the double tangential derivatives of $\hure$ satisfy
\be\label{c2-inside-tan-s}
\frac{n-2}{n-1}\bb\lt(1+\frac{2\e}{\bb}\rt)^{-\frac{1}{n-1}}\delta_{\td\alpha\td\beta}\leq(\hure)_{\td\alpha\td\beta}\leq\frac{n-2}{n-1}\ba\delta_{\td\alpha\td\beta}.
\ee
 \end{lemma}
 Note that \eqref{c2-inside-tan} is derived from \eqref{c2-inside-tan-s} by rescaling.

 \begin{lemma}
 \label{c2-inside-mix-lem}
Given $\eta_0=\eta_0(\Omega)>0, \ev=\ev(\Omega, \{0\}, n)>0$ sufficiently small, $R_0=R_0(\Omega)>0$ sufficiently large, for any $0<\e<\eta_0 R^{-8},$ $R>R_0,$ and $\varphi\in C^{\infty}(\bar\Omega)$ that satisfies $\varphi\leq0$ and $\|\varphi\|_{C^2}\leq \ev$ let $\ure$ be the admissible solution of \eqref{eq-appr}. Then on $\p B_{1/R},$ $\ure$ satisfies
 \be\label{c2-inside-mix}
 \lt|\lt(\ure\rt)_{\tau\nu}\rt|<CR^{2}.
 \ee
 Let $\hure$ be the admissible solution of \eqref{eq-appr-s}, then on $\p B_1,$ $\hure$ satisfies
 \be\label{c2-inside-mix-s}
 \lt|\lt(\hure\rt)_{\tau\nu}\rt|<CR^{\frac{n-2}{n-1}}.
 \ee
Here, $\tau$ is an arbitrary unit tangential vector of $\p B_1,$ $\nu$ is the inward unit normal of $\p B_1,$ and $C=C(\Omega, \{0\}, n)>0$ is a constant that is independent of $R$ and $\e.$
 \end{lemma}
 \begin{proof}
 We use the coordinate system $\{\td x_1, \td x_2, \cdots, \td x_n\}$ chosen above. Let $T:=\p_{\td\alpha}-(\td x_{\al}\p_{\td n}-\td x_{n}\p_{\td\al}),$
 then on $\p B_1$ near $x_0$ we have
 \be\label{add-mix-1}|T\hure|\leq C|\td x'|^2\,\,\mbox{for $\td x'=(\td x_1, \td x_2, \cdots, \td x_{n-1})$}\ee
 where $C=C(\Omega, \{0\}, n)>0$ is some positive constant that is independent of $R$ and $\e.$
 Denote $\frak L:=F^{\td i\td j}|_{\hure}\p_{\tilde i}\p_{\tilde j}$ and $\td B_{\delta_0}:=B_{\delta_0}(x_0)\setminus\bar B_1(0),$ where $\delta_0>0$ is a fixed small constant.
 From \eqref{fr} it is straightforward to see that
 \be\label{add-mix-2}|\frak LT\hure|<4f_\e\,\,\mbox{in $\td B_{\delta_0}.$}\ee
 Now, consider
 \[\ubar v=\bb(r+1)^{\frac{n-2}{n-1}}-2^{\frac{n-2}{n-1}}\bb+\ba.\]
 It is easy to check that $\ubar v$ is an admissible subsolution of \eqref{eq-appr-s} and it satisfies
 \[F(D^2\ubar v)=\frac{\bb}{\lt(\frac{n-1}{n-2}\rt)r(r+1)^{\frac{1}{n-1}}\lt[\frac{n}{2}r+(n-1)\rt]}\gg f_\e.\]
 Therefore, there exists $\theta=\theta(\bb, \delta_0, n)>0$ such that
 \[\lambda(D^2(\ubar v-\theta|\td x|^2))\in\Gamma_{n-1}\,\,\mbox{and $F(D^2(\ubar v-\theta|\td x|^2))>10 f_\e$ in $\td B_{\delta_0}.$}\]
 Let $h=\hure-\ubar v+\theta|\td x|^2,$ then
 \be\label{add-mix-3}\mbox{on $\p\td B_{\delta_0}\cap\p B_1(0), $ $h=\theta|\td x|^2$}\ee
 and
 \be\label{add-mix-4}\mbox{on $\p\td B_{\delta_0}\setminus\p B_1(0), $ $h\geq\theta\delta_0^2.$}\ee
 Moreover, by the concavity of $F$ we have
 \be\label{add-mix-5}\begin{aligned}
 \frak L h&=f_\e-\frak L(\ubar v-\theta|\td x|^2)\\
 &\leq f_\e-F(D^2(\ubar v-\theta|\td x|^2))<-9 f_\e.
 \end{aligned}
 \ee
 In view of the $C^1$ estimate \eqref{c1-bound-s} we get
 \be\label{c2-inside-1}
 |T\hure|<6R^{\frac{n-2}{n-1}}\,\,\mbox{ on $\p\td B_{\delta_0}\setminus\p B_1$.}
 \ee
 By virtue of \eqref{add-mix-1}, \eqref{add-mix-3}, \eqref{add-mix-4}, and \eqref{c2-inside-1}, we know there exists $B=B(\Omega, \{0\}, n, \theta, \delta_0)>0$ large such that
 on $\p\td B_{\delta_0}$ it holds
 \be\label{v1.1}-BR^{\frac{n-2}{n-1}}h\leq T\hure\leq BR^{\frac{n-2}{n-1}}h;\ee
 and in $\td B_{\delta_0}$ it holds
 \[-BR^{\frac{n-2}{n-1}}\frak L h>\frak LT\hure>BR^{\frac{n-2}{n-1}}\frak Lh. \]
 By the standard maximum principle we can see that \eqref{v1.1} holds in $\td B_{\delta_0}.$ This implies at $x_0,$ we have
  \[-BR^{\frac{n-2}{n-1}}h_{\td n}< (T\hure)_{\td n}<BR^{\frac{n-2}{n-1}}h_{\td n}.\]
  Similarly, we have
   \[-BR^{\frac{n-2}{n-1}}h_{\td n}< (-T\hure)_{\td n}<BR^{\frac{n-2}{n-1}}h_{\td n}.\]
  Since $x_0\in \p B_1$ is arbitrary, we have completed the proof of \eqref{c2-inside-mix-s} and \eqref{c2-inside-mix} follows from rescaling.
\end{proof}
Since we have estimated $|(\hure)_{\td\al\td\beta}|$ and $|(\hure)_{\td\al\td n}|,$ to bound $|D^2\hure|$ on $\p B_1$ it suffices to bound $|(\hure)_{\td n\td n}|.$
We may solve equation \eqref{eq-appr-s} for $(\hure)_{\td n\td n}(x_0)$ directly. In fact, \eqref{eq-appr-s} can be written as
$\s_{n-1}(D^2 u)=f_\e\s_{n-2}(D^2u).$ Therefore, by \eqref{c2-inside-0} we obtain at $x_0\in\p B_1$
\[\begin{aligned}
&\lt|D_{\td n}\hure\rt|^{n-1}+C_{n-1}^{n-2}\lt|D_{\td n}\hure\rt|^{n-2}\hure_{\td n\td n}-C_{n-2}^{n-3}\lt|D_{\td n}\hure\rt|^{n-3}\sum\limits_{\td s=1}^{n-1}
(\hure_{\td n\td s})^2\\
&=f_\e\lt(C_{n-1}^{n-2}\lt|D_{\td n}\hure\rt|^{n-2}+C_{n-1}^{n-3}\lt|D_{\td n}\hure\rt|^{n-3}\hure_{\td n\td n}
-C_{n-2}^{n-4}\lt|D_{\td n}\hure\rt|^{n-4}\sum\limits_{\td s=1}^{n-1}(\hure_{\td n\td s})^2\rt).\end{aligned}\]
Applying Lemma \ref{c1-inside-lem} and Lemma \ref{c2-inside-mix-lem} to the equality above, and using the arbitrariness of $x_0,$ we conclude the following lemma.
\begin{lemma}
 \label{c2-inside-nu-lem}
Given $\eta_0=\eta_0(\Omega)>0, \ev=\ev(\Omega, \{0\}, n)>0$ sufficiently small, $R_0=R_0(\Omega)>0$ sufficiently large, for any $0<\e<\eta_0 R^{-8},$ $R>R_0,$ and $\varphi\in C^{\infty}(\bar\Omega)$ that satisfies $\varphi\leq0$ and $\|\varphi\|_{C^2}\leq \ev$ let $\ure$ be the admissible solution of \eqref{eq-appr}.
 Then on $\p B_{1/R},$ $\ure$ satisfies
 \be\label{c2-inside-nu}
 \lt|\lt(\ure\rt)_{\nu\nu}\rt|<CR^{3-\frac{1}{n-1}}.
 \ee
 Let $\hure$ be the admissible solution of \eqref{eq-appr-s}, then on $\p B_1,$ $\hure$ satisfies
 \be\label{c2-inside-nu-s}
 \lt|\lt(\hure\rt)_{\nu\nu}\rt|<CR^{\frac{2(n-2)}{n-1}}.
 \ee
Here, $\nu$ is the inward unit normal of $\p B_1,$ and $C=C(\Omega, \{0\}, n)>0$ is a constant that is independent of $R$ and $\e.$
 \end{lemma}
 Note that \eqref{c2-inside-nu} is derived from \eqref{c2-inside-nu-s} by rescaling.

 \subsubsection{$C^2$-estimates on $\p\Omega$}
 \label{subsub-c2-outside}
 In this subsubsection, we shall study the $C^2$ estimates of $\ure$ on $\p\Omega,$ and the $C^2$ estimates of $\hure$ on $\p\Omega^R$ are established by rescaling.

 Same as in Subsubsection \ref{subsub-c2-inside} let $x_0\in\p\Omega$ be an arbitrary point. We may choose a local coordinate system $\{\td x_1, \td x_2, \cdots, \td x_n\}$ in a neighborhood of $x_0$ such that $x_0$ is the
 origin and $\td x_n$ axis is the inward normal of $\p\Omega$ (pointing into $\Omega\setminus B_{1/R}$) at $x_0$.
 Denote $\td B_{\delta_0}:= B_{\delta_0}(x_0)\cap\Omega$ for $\delta_0>0$ being a fixed small constant. Then the boundary $\p\td B_{\delta_0}\cap\p\Omega$ can be expressed
 as a vertical graph near $x_0,$ that is,
 \[\td x_n=\Upsilon(\td x')=\frac{1}{2}\sum\limits_{\al=1}^{n-1}\kappa_\al\td x_\al^2+O(|\td x'|^3),\]
 where $\kappa_1, \cdots,  \kappa_{n-1}$ are the principal curvatures of $\p\Omega$ at $x_0$ and $\td x'=(\td x_1, \cdots, \td x_{n-1}).$
 Recall that on $\p\Omega$ we have
 \[\ure(\td x', \Upsilon(\td x'))=1+\varphi,\]
by differentiating it twice we can see that at $x_0$
 \[(\ure)_{\td\al\td\beta}=-(\ure)_{\td n}\kappa_\al\delta_{\td\al\td\beta}+\varphi_{\td\al\td\beta}\,\,
 \mbox{for}\,\, \td\al, \td\beta\leq n-1.\]

 In view of the assumptions that $\p\Omega$ is strictly convex and $\e_\varphi>0$ is sufficiently small, we can apply Lemma \ref{c1-outside-lem} and obtain
 \begin{lemma}
 \label{c2-outside-tan-lem}
 Given $\eta_0=\eta_0(\Omega)>0, \ev=\ev(\Omega, \{0\}, n)>0$ sufficiently small, $R_0=R_0(\Omega)>0$ sufficiently large, for any $0<\e<\eta_0 R^{-8},$ $R>R_0,$ and $\varphi\in C^{\infty}(\bar\Omega)$ that satisfies $\varphi\leq0$ and $\|\varphi\|_{C^2}\leq \ev$ let $\ure$ be the admissible solution of \eqref{eq-appr}. Then at any $x\in\p\Omega$ we can choose a local coordinate system at $x$ such that the double tangential derivatives of $\ure$ at $x$ satisfy
\be\label{c2-outside-tan}
e \delta_{\td\al\td\beta}\leq(\ure)_{\td\alpha\td\beta}\leq q \delta_{\td\alpha\td\beta}.
\ee
Let $\hure$ be the admissible solution of \eqref{eq-appr-s}, then at any $x\in\p\Omega^R$ we can choose a local coordinate system at $x$ such that the double tangential derivatives of $\hure$ at $x$ satisfy
\be\label{c2-outside-tan-s}
e R^{-\frac{n}{n-1}}\delta_{\td\alpha\td\beta}\leq(\hure)_{\td\alpha\td\beta}\leq q R^{-\frac{n}{n-1}}\delta_{\td\alpha\td\beta}.
\ee
Here, $e=e(\Omega, \{0\}, n), q=q(\Omega, \{0\}, n)>0$ are positive constants that are independent of $R$ and $\e.$
\end{lemma}

\begin{lemma}
 \label{c2-outside-mix-lem}
Given $\eta_0=\eta_0(\Omega)>0, \ev=\ev(\Omega, \{0\}, n)>0$ sufficiently small, $R_0=R_0(\Omega)>0$ sufficiently large, for any $0<\e<\eta_0 R^{-8},$ $R>R_0,$ and $\varphi\in C^{\infty}(\bar\Omega)$ that satisfies $\varphi\leq0$ and $\|\varphi\|_{C^2}\leq \ev$ let $\ure$ be the admissible solution of \eqref{eq-appr}. Then on $\p\Omega,$ $\ure$ satisfies
 \be\label{c2-outside-mix}
 \lt|\lt(\ure\rt)_{\tau\nu}\rt|<CR.
 \ee
 Let $\hure$ be the admissible solution of \eqref{eq-appr-s}, then on $\p\Omega^R,$ $\hure$ satisfies
 \be\label{c2-outside-mix-s}
 \lt|\lt(\hure\rt)_{\tau\nu}\rt|<CR^{-\frac{1}{n-1}}.
 \ee
Here, $\tau$ is an arbitrary unit tangential vector of $\p\Omega,$ $\nu$ is the inward unit normal of $\p\Omega,$ and
$C>0$ is a positive constant that only depends on $\Omega,$ $\{0\},$ $n,$ and $\|\varphi\|_{C^3}.$
 \end{lemma}
 \begin{proof}
 We use the coordinate system $\{\td x_1, \td x_2, \cdots, \td x_n\}$ chosen above. Let $T:=\p_{\td\alpha}+\kappa_{\al}(\td x_{\al}\p_{\td n}-\td x_{n}\p_{\td\al}),$
 then on $\p\td B_{\delta_0}\cap\p\Omega$ we have
\be\label{add-o-mix1}
 |T\lt(\ure-(1+\varphi)\rt)|\leq C|\td x'|^2\,\,\mbox{for $\td x'=(\td x_1, \cdots, \td x_{n-1}),$}
 \ee
 where $C=C(\Omega, \{0\}, n, \varphi)>0.$ However, since $\|\varphi\|_{C^2}\leq\ev$ for some small $\ev=\ev(\Omega, \{0\}, n),$ we say $C=C(\Omega, \{0\}, n).$

Denote $\frak L:=F^{\td i\td j}|_{\ure}\p_{\td i}\p_{\td j}$ then from \eqref{fr} it is straightforward to see that
 \[|\frak LT\ure|<4f_{\frac{\e}{R}}\,\,\mbox{in $\td B_{\delta_0},$}\]
 where we used the assumption that $B_1(0)\ssubset\Omega$ and $\delta_0>0$ is small.
 Moreover, it is easy to derive
 \[|\frak LT\varphi|<E_\varphi\sum F^{\td i\td i}\,\,\mbox{for}\,\,E_\varphi=E_\varphi(\|\varphi\|_{C^3}, n).\]
 Thus
 \be\label{add-o-mix2}
 |\frak LT\lt(\ure-(1+\varphi)\rt)|<4f_{\frac{\e}{R}}+E_\varphi\sum F^{\td i\td i}.
 \ee
 In this proof, we shall use the lower barrier $\ubar\phi$ constructed in Subsubsection \ref{subsub-c1-outside}, that is,
 \[\ubar\phi=2(\fb-1)+1+\varphi.\]
We have shown in Subsubsection \ref{subsub-c1-outside} that $\ubar\phi$ is an admissible subsolution of \eqref{eq-appr} and it satisfies
 \[F(D^2\ubar\phi)>\frac{d_0}{r}\gg f_{\frac{\e}{R}}\]
 for some $d_0=d_0(\Omega, \{0\}, n)>0.$
 Therefore, there exists $\theta=\theta(\Omega, \{0\}, \delta_0, n)>0$ such that
 \[\lambda(D^2(\ubar\phi-2\theta|\td x|^2))\in\Gamma_{n-1}\,\,\mbox{and $F(D^2(\ubar\phi-2\theta|\td x|^2))>10 f_{\frac{\e}{R}}$ in $\td B_{\delta_0}.$}\]
 Let $h=\ure-(\ubar\phi-2\theta|\td x|^2)-\theta|\td x|^2=\ure-\ubar\phi+\theta|\td x|^2,$ then
 \be\label{add-o-mix3}
 \mbox{on $\p\td B_{\delta_0}\cap\p\Omega, $ $h=\theta|\td x|^2$}
 \ee
 and
 \be\label{add-o-mix4}
 \mbox{on $\p\td B_{\delta_0}\setminus\p\Omega, $ $h\geq\theta\delta_0^2.$}
 \ee
 Moreover, by the concavity of $F$ we have
 \be\label{add-o-mix5}
 \begin{aligned}
 \frak L h&=\frak L\ure-\frak L(\ubar\phi-2\theta|\td x|^2)-2\theta\sum F^{\td i\td i}\\
 &\leq f_{\frac{\e}{R}}-F(D^2(\ubar\phi-2\theta|\td x|^2))-2\theta\sum F^{\td i\td i}\\
 &<-9 f_{\frac{\e}{R}}-2\theta\sum F^{\td i\td i}.
 \end{aligned}
 \ee
 In view of the $C^1$ estimate \eqref{c1-bound} we get
 \be\label{c2-outside-1}
 |T\lt(\ure-(1+\varphi)\rt)|<6R\,\,\mbox{ on $\p\td B_{\delta_0}\setminus\p\Omega$.}
 \ee
 By virtue of \eqref{add-o-mix1}, \eqref{add-o-mix2}, \eqref{add-o-mix3}, \eqref{add-o-mix4}, \eqref{add-o-mix5}, and \eqref{c2-outside-1} we know there exists$B=B(\Omega, \{0\}, n, \theta, \delta_0, E_\varphi)>0$ large such that
 on $\p\td B_{\delta_0}$ it holds
 \be\label{add-inequality}-BRh\leq T\lt(\ure-(1+\varphi)\rt)\leq BRh;\ee
 and in $\td B_{\delta_0}$ it holds
 \[-BR\frak L h>\frak LT\lt(\ure-(1+\varphi)\rt)>BR\frak Lh. \]
 The standard maximum principle then implies \eqref{add-inequality} holds in $\td B_{\delta_0}.$ This yeilds at $x_0$ we have
  \[-BRh_{\td n}< \lt(T\lt(\ure-(1+\varphi)\rt)\rt)_{\td n}<BRh_{\td n}.\]
  Similarly, we have
   \[-BRh_{\td n}< \lt(-T\lt(\ure-(1+\varphi)\rt)\rt)_{\td n}<BRh_{\td n}.\]
  Since $x_0\in \p\Omega$ is an arbitrary point, we have completed the proof of \eqref{c2-outside-mix} and \eqref{c2-outside-mix-s} follows by rescaling.
\end{proof}
 To obtain the $C^2$ boundary estimate in the double normal direction, i.e., $|(\ure)_{\nu\nu}|,$ we can follow the same argument as in the Subsubsection \ref{subsub-c2-inside}. Applying Lemma \ref{c2-outside-tan-lem} and Lemma \ref{c2-outside-mix-lem} we conclude
\begin{lemma}
 \label{c2-outside-nu-lem}
 Given $\eta_0=\eta_0(\Omega)>0, \ev=\ev(\Omega, \{0\}, n)>0$ sufficiently small, $R_0=R_0(\Omega)>0$ sufficiently large, for any $0<\e<\eta_0 R^{-8},$ $R>R_0,$ and $\varphi\in C^{\infty}(\bar\Omega)$ that satisfies $\varphi\leq0$ and $\|\varphi\|_{C^2}\leq \ev$ let $\ure$ be the admissible solution of \eqref{eq-appr}.
 Then on $\p\Omega,$ $\ure$ satisfies
 \be\label{c2-outside-nu}
 \lt|\lt(\ure\rt)_{\nu\nu}\rt|<CR^2.
 \ee
 Let $\hure$ be the admissible solution of \eqref{eq-appr-s}, then on $\p\Omega^R,$ $\hure$ satisfies
 \be\label{c2-outside-nu-s}
 \lt|\lt(\hure\rt)_{\nu\nu}\rt|<CR^{\frac{n-2}{n-1}}.
 \ee
Here, $\nu$ is the inward unit normal of $\p\Omega,$ and $C=C(\Omega, \{0\}, \|\varphi\|_{C^3}, n)>0$ is a constant that is independent of $R$ and $\e.$
 \end{lemma}

\subsubsection{$C^2$ global estimates}
\label{subsub-c2-global} Finally, we shall investigate the $C^2$ estimates of $\ure$ in $\Omega\setminus B_{1/R}.$
\begin{lemma}
 \label{c2-global-lem}
Given $\eta_0=\eta_0(\Omega)>0, \ev=\ev(\Omega, \{0\}, n)>0$ sufficiently small, $R_0=R_0(\Omega)>0$ sufficiently large, for any $0<\e<\eta_0 R^{-8},$ $R>R_0,$ and $\varphi\in C^{\infty}(\bar\Omega)$ that satisfies $\varphi\leq0$ and $\|\varphi\|_{C^2}\leq \ev$ let $\ure$ be the admissible solution of \eqref{eq-appr}.
Then on $\ol{\Omega\setminus B_{1/R}},$ $\ure$ satisfies
 \be\label{c2-global}
\lt|D^2 \ure\rt|<CR^{3-\frac{1}{n-1}}.
 \ee
 Let $\hure$ be the admissible solution of \eqref{eq-appr-s}, then on $\ol{\Omega^R\setminus B_1},$ $\hure$ satisfies
 \be\label{c2-global-s}
 \lt|D^2\hure\rt|<CR^{2-\frac{2}{n-1}}.
 \ee
Here, $C=C(\Omega, \{0\}, n)>0$ is a constant that is independent of $R$ and $\e.$
 \end{lemma}
 \begin{proof}
 We shall prove \eqref{c2-global} and \eqref{c2-global-s} follows by rescaling.
 Let $\frak L:=F^{ij}|_{\ure}\partial_{i}\p_j,$ then
 \[\frak L(\Laplace\ure)=-\sum\limits_{p, q, r, s,i}F^{pq, rs}(\ure)_{pqi}(\ure)_{rsi}+\Laplace f_{\frac{\e}{R}}.\]
 A straightforward calculation yields
 \be\label{frr}
 \begin{aligned}
 \frac{\p^2f_{\frac{\e}{R}}}{\p r^2}&=f_{\frac{\e}{R}}\lt[r^{-1}+\frac{1}{n-1}\lt(r+\frac{\e}{R}\rt)^{-1}+\lt(r+\frac{2(n-1)\e}{nR}\rt)^{-1}\rt]^2\\
 &+f_{\frac{\e}{R}}\lt[r^{-2}+\frac{1}{n-1}\lt(r+\frac{\e}{R}\rt)^{-2}+\lt(r+\frac{2(n-1)\e}{nR}\rt)^{-2}\rt].
 \end{aligned}
 \ee
 By virtue of \eqref{fr} and \eqref{hess1.1}--\eqref{hess1.3} we derive
 \begin{align*}
 \Laplace f_{\frac{\e}{R}}&=\frac{n-1}{r}\frac{\p f_{\frac{\e}{R}}}{\p r}+\frac{\p^2 f_{\frac{\e}{R}}}{\p r^2}\\
 &=\frac{n-1}{r} f_{\frac{\e}{R}}\lt[-r^{-1}-\frac{1}{n-1}\lt(r+\frac{\e}{R}\rt)^{-1}-\lt(r+\frac{2(n-1)\e}{nR}\rt)^{-1}\rt]\\
 &+f_{\frac{\e}{R}}\lt[r^{-1}+\frac{1}{n-1}\lt(r+\frac{\e}{R}\rt)^{-1}+\lt(r+\frac{2(n-1)\e}{nR}\rt)^{-1}\rt]^2\\
 &+f_{\frac{\e}{R}}\lt[r^{-2}+\frac{1}{n-1}\lt(r+\frac{\e}{R}\rt)^{-2}+\lt(r+\frac{2(n-1)\e}{nR}\rt)^{-2}\rt].\\
 \end{align*}
 In view of the concavity of $F$ we can easily see that
 $$\frak L\lt(\Laplace\ure+2(n-1)R^2\ure\rt)>0.$$
 This gives
 \[\max\limits_{\ol{\Omega\setminus B_{1/R}}}\Laplace\ure+2(n-1)R^2\ure=\max\limits_{\p(\Omega\setminus B_{1/R})}\Laplace\ure+2(n-1)R^2\ure.\]
 Since $(\Laplace\ure)^2-2\s_2(D^2\ure)=\sum_{i, j}|(\ure)_{ij}|^2,$ we can see that \eqref{c2-global} follows from Subsubsection \ref{subsub-c2-inside} and
 Subsubsection \ref{subsub-c2-outside}.
 \end{proof}
 \subsection{Conclusion on the solvability of \eqref{eq-appr}}
 \label{sub-sol-conclusion}
 Theorem \ref{solve-appr-thm} can be proved using continuity method and a priori estimates obtained in Proposition \ref{c0-prop}, Corollary \ref{c1-bound-cor}, and Lemma \ref{c2-global-lem}. The higher order regularity of the solution $\ure$ of \eqref{eq-appr} (respectively, the solution $\hure$ of \eqref{eq-appr-s}) comes from the Evans-Krylov theorem and the Schauder estimates. Therefore, we proved Theorem \ref{solve-appr-thm}.

 \section{Improved $C^1$ and $C^2$ estimates}
 \label{improve}
 In this section, we shall first follow the argument in \cite{CW01} (see also \cite{Xiao22}) to obtain a better $C^1$ estimate of $\hure,$ which implies a better $C^1$ estimate of $\ure.$ Then, we shall use this improved $C^1$ estimate to get a better $C^2$ estimate of $\hure,$ which again yields a better $C^2$ estimate of $\ure.$

 \subsection{Improved $C^1$ upper bound}
 \label{sub-imp-c1-up}
 Our goal in this subsection is to show
 \be\label{im-up-1}
 |D\hure(x)|<\frac{C}{|x|^{\frac{1}{n-1}}}
 \ee
 for some constant $C=C(\Omega, \{0\}, n)>0$ that is independent of $R$ and $\e.$
 We shall prove \eqref{im-up-1} in 3 cases. First, we prove that \eqref{im-up-1} holds when $x\in \ol{B_{R/2}\setminus B_2};$ secondly, we prove that \eqref{im-up-1}
 holds when $x\in \ol{B_2\setminus B_1};$ finally, we prove that \eqref{im-up-1}
 holds when $x\in \ol{\Omega^R\setminus B_{R/2}}.$
 \subsubsection{Inequality \eqref{im-up-1} for $x\in \ol{B_{R/2}\setminus B_2}$ }
 \label{subsub-im-c1-mid}
 We shall prove the following Lemma.
 \begin{lemma}
 \label{im-c1-mid-lem}
Given $\eta_0=\eta_0(\Omega)>0, \ev=\ev(\Omega, \{0\}, n)>0$ sufficiently small, $R_0=R_0(\Omega)>0$ sufficiently large, for any $0<\e<\eta_0 R^{-8},$ $R>R_0,$ and $\varphi\in C^{\infty}(\bar\Omega)$ that satisfies $\varphi\leq0$ and $\|\varphi\|_{C^2}\leq \ev$ let $\ure$ be the admissible solution of \eqref{eq-appr}.
Then on $\ol{B_{1/2}\setminus B_{2/R}},$ $\ure$ satisfies
 \be\label{im-c1-mid}
\lt|D\ure\rt|<C|x|^{-\frac{1}{n-1}}.
 \ee
 Let $\hure$ be the admissible solution of \eqref{eq-appr-s}, then on $\ol{B_{R/2}\setminus B_2},$ $\hure$ satisfies
 \be\label{im-c1-mid-s}
 \lt|D\hure\rt|<C|x|^{-\frac{1}{n-1}}.
 \ee
Here, $C=C(\Omega, \{0\}, n)>0$ is a constant that is independent of $R$ and $\e.$
\end{lemma}
\begin{proof}
We shall only prove \eqref{im-c1-mid-s}, \eqref{im-c1-mid} follows by rescaling.
In this proof, we shall use the idea in \cite{CW01} (see also \cite{Xiao22}). Notice that for any $x\in \ol{B_{R/2}\setminus B_2},$
we have,
\[\text{dist}(x, \p B_1)=|x|-1\geq\frac{1}{2}|x|\]
and
\[\text{dist}(x, \p\Omega^R)>\frac{R}{2}>|x|,\]
where we have used $B_1\ssubset\Omega.$
Thus, for any $x\in\ol{B_{R/2}\setminus B_2},$ we get
\[B_{\frac{1}{3}|x|}(x)\ssubset\Omega^R\setminus B_1.\]

Now, we fix an arbitrary $x_0\in\ol{B_{R/2}(0)\setminus B_2(0)}.$ Denote $\td r:=\frac{1}{6}|x_0|,$ then $B_{\td r}(x_0)\ssubset\Omega^R\setminus B_1$ and for any $x\in B_{\td r}(x_0)$ we have
\[\frac{5}{6}|x_0|<|x|<\frac{7}{6}|x_0|.\]
Therefore, by Proposition \ref{c0-prop} we get for any $x\in B_{\td r}(x_0)$
\be\label{im-c1-mid-1}
\frac{5}{6}\bb|x_0|^{\frac{n-2}{n-1}}<\hure(x)<\frac{7}{6}\ba|x_0|^{\frac{n-2}{n-1}}.
\ee
In this proof, we denote $M:=4\times\frac{7}{6}\ba|x_0|^{\frac{n-2}{n-1}}$ and for our convenience, in the rest of this proof, we shall write $u$ instead of $\hure.$

Consider the test function
\[G(x, \xi)=u_\xi(x)\psi(u)\rho(x),\]
where $\xi$ is some unit vector filed, $\psi(\tau)=(M-\tau)^{-1/2},$ and $\rho(x)=1-\frac{|x-x_0|^2}{\td r^2}.$
Suppose $G$ attains its maximum at $x=\hat x$ and $\xi=e_1,$ that is $u_1(\hat x)=|Du(\hat x)|.$ We may rotate the coordinates such that
$u_{\al\beta}(\hat x)=u_{\al\al}(\hat x)\delta_{\al\beta}$ for $2\leq \al, \beta\leq n.$ Then at $\hat x$ we have
$G_i=0$ for $i\geq 1$ and $(G_{ij})\leq 0.$  That is,
\be\label{im-c1-mid-2}
\frac{u_{1i}}{u_1}=-\lt(\frac{\psi_i}{\psi}+\frac{\rho_i}{\rho}\rt)
\ee
and
\be\label{im-c1-mid-3}
\begin{aligned}
0&\geq F^{ij}\lt(\frac{u_{1ij}}{u_1}-\frac{u_{1i}u_{1j}}{u_1^2}+\frac{\psi_{ij}}{\psi}-\frac{\psi_i\psi_j}{\psi^2}
+\frac{\rho_{ij}}{\rho}-\frac{\rho_i\rho_j}{\rho^2}\rt)\\
&=\frac{(f_\e)_1}{u_1}-F^{ij}\lt(\frac{\psi_i}{\psi}+\frac{\rho_i}{\rho}\rt)\lt(\frac{\psi_j}{\psi}+\frac{\rho_j}{\rho}\rt)\\
&+F^{ij}\lt(\frac{\psi_{ij}}{\psi}-\frac{\psi_i\psi_j}{\psi^2}
+\frac{\rho_{ij}}{\rho}-\frac{\rho_i\rho_j}{\rho^2}\rt)\\
&=\frac{(f_\e)_1}{u_1}+\frac{1}{\psi}F^{ij}\lt(\psi_{ij}-2\frac{\psi_i\psi_j}{\psi}\rt)-\frac{1}{\psi\rho}F^{ij}(\psi_i\rho_j+\psi_j\rho_i)\\
&+\frac{1}{\rho}F^{ij}\rho_{ij}-\frac{2}{\rho^2}F^{ij}\rho_i\rho_j,
\end{aligned}
\ee
where $F^{ij}=F^{ij}|_{\hure}.$
Now, multiplying \eqref{im-c1-mid-3} by $u_1\psi\rho$ we obtain
\be\label{im-c1-mid-4}
\begin{aligned}
0&\geq\psi\rho(f_\e)_1+u_1\rho F^{ij}\lt(\psi_{ij}-2\frac{\psi_i\psi_j}{\psi}\rt)\\
&-u_1F^{ij}(\psi_i\rho_j+\psi_j\rho_i)+u_1\psi F^{ij}\rho_{ij}-2\frac{u_1\psi}{\rho}F^{ij}\rho_i\rho_j
\end{aligned}
\ee
Note that $\psi_i=\psi'u_i$ and $\psi_{ij}=\psi'u_{ij}+\psi''u_iu_j,$ \eqref{im-c1-mid-4} becomes
\be\label{im-c1-mid-5}
\begin{aligned}
0&\geq\psi\rho(f_\e)_1+u_1\rho\psi'f_\e+u_1\rho \lt(\psi''-2\frac{\psi'^2}{\psi}\rt)F^{ij}u_iu_j\\
&-u_1\psi'F^{ij}(u_i\rho_j+u_j\rho_i)+u_1\psi F^{ij}\rho_{ij}-2\frac{u_1\psi}{\rho}F^{ij}\rho_i\rho_j.
\end{aligned}
\ee
Since $\psi'=\frac{1}{2}(M-\tau)^{-3/2}$ and $\psi''=\frac{3}{4}(M-\tau)^{-5/2},$ we can see that
\[\psi''-2\frac{\psi'^2}{\psi}=\frac{1}{4}(M-u)^{-5/2}>\frac{1}{4}M^{-5/2}.\]
Inserting this inequality into \eqref{im-c1-mid-5} we derive
\be\label{im-c1-mid-6}
\begin{aligned}
0&\geq-C_1\frac{\psi\rho f_\e}{\td r}+\frac{1}{4}M^{-5/2}\rho F^{11}u_1^3-2\frac{u_1\psi}{\td r^2}\sum_i F^{ii}\\
&-\lt[\psi'u_1^2|D\rho|+\frac{u_1\psi}{\rho}|D\rho|^2\rt]C_2\sum_iF^{ii},
\end{aligned}
\ee
where $C_1=C_1(n)>0$ and $C_2=C_2(n)>0$.
We want to point out that for deriving the right hand side of \eqref{im-c1-mid-6} we used
\eqref{fr} and $u_1\rho\psi'f_\e>0$ at $\hx$.
By our choice of $M$, we have $$\psi<\lt(\frac{3}{4}M\rt)^{-1/2}=\frac{2}{\sqrt{3}}M^{-1/2}<2M^{-1/2}$$
and $$\psi'<\frac{1}{2}\lt(\frac{3}{4}M\rt)^{-3/2}=\frac{4}{3\sqrt{3}}M^{-3/2}<M^{-3/2}.$$
Moreover, $|D\rho|=\frac{2|\hat x-x_0|}{\td r^2}\leq \frac{2}{\td r}.$ Thus, \eqref{im-c1-mid-6} implies
\be\label{im-c1-mid-7}
\begin{aligned}
0&\geq-2\frac{C_1M^2 f_\e}{\td r}+\frac{1}{4}\rho F^{11}u_1^3-4\frac{u_1M^2}{\td r^2}\sum F^{ii}\\
&-\lt(2\frac{u_1^2M}{\td r}+8\frac{u_1M^2}{\rho\td r^2}\rt)C_2\sum_iF^{ii}
\end{aligned}
\ee
Now, we shall prove by contradiction and assume $|Du(x_0)|>C\frac{M}{\td r}$ for some undetermined large constant $C>0.$
By $G(\hat x)\geq G(x_0)$ we have
\[u_1(\hat x)\psi(u(\hx))\rho(\hat x)\geq C\frac{M}{\td r}\psi(u(x_0)),\] which
gives \[u_1(\hat x)\rho(\hat x)\geq \frac{\sqrt{3}}{2}C\frac{M}{\td r}.\]
By virtue of \eqref{im-c1-mid-2} we can see that at $\hat x$
\[\begin{aligned}
u_{11}&=-u_1\lt(\frac{\psi'}{\psi}u_1+\frac{\rho_1}{\rho}\rt)\\
&=-\frac{u_1^2\psi'}{2\psi}-u_1\lt(\frac{\psi'}{2\psi}u_1+\frac{\rho_1}{\rho}\rt).
\end{aligned}\]
Note also
\[\frac{\psi'}{2\psi}u_1\rho+\rho_1\geq\frac{1}{4}(M-u(\hat x))^{-1}\cdot\frac{\sqrt{3}}{2}C\frac{M}{\td r}-\frac{2}{\td r}>\frac{\sqrt{3}C}{8\td r}-\frac{2}{\td r}.\]
Therefore, when $C>\frac{16}{\sqrt{3}}$
we have, at $\hat x$
$$u_{11}<-\frac{u_1^2\psi'}{2\psi}<0,$$
and
\[|u_{11}|>\frac{u^2_1}{4}(M-u(\hx))^{-1}>\frac{3}{16}C^2\frac{M}{\td r^2}>C\td r^{-\frac{n}{n-1}}>CR^{-\frac{n}{n-1}},\]
for some $C=C(\Omega, \{0\}, n)>0.$
Note that at $\hx$ we have
\[\s_{n-1}(D^2u)=u_{11}\s_{n-2}(u_{\al\al})+\s_{n-1}(u_{\al\al})-\sum\limits_{s=2}^nu^2_{1s}\s_{n-3}(u_{\al\al}|s)>0\]
and
\[\s_{n-2}(D^2u)=u_{11}\s_{n-3}(u_{\al\al})+\s_{n-2}(u_{\al\al})-\sum\limits_{s=2}^nu^2_{1s}\s_{n-4}(u_{\al\al}|s)>0,\]
where $2\leq\al\leq n.$ This gives
\[\s_{n-1}^{11}=\s_{n-2}(u_{\al\al})>|u_{11}|\s_{n-3}(u_{\al\al})=|u_{11}|\s_{n-2}^{11}>CR^{-\frac{n}{n-1}}\s_{n-2}^{11}.\]
By virtue of $\s_{n-2}F^{ii}=\s_{n-1}^{ii}-f_\e\s_{n-2}^{ii},$ $\e<\eta_0R^{-8},$ and equation (3.10) of \cite{CW01} we obtain
\[\s_{n-2}F^{11}>\frac{1}{2}\s_{n-1}^{11}>\theta\sum\limits_i\s_{n-1}^{ii}>\theta\s_{n-2}\sum\limits_iF^{ii}\]
for some $\theta=\theta(n)>0.$
Thus \eqref{im-c1-mid-7} becomes
\be\label{im-c1-mid-8}
\begin{aligned}
0&\geq-2\frac{C_1M^2 f_\e}{\td r}+\frac{1}{4}\theta\rho u_1^3\sum_iF^{ii}-4\frac{u_1M^2}{\td r^2}\sum F^{ii}\\
&-\lt(2\frac{u_1^2M}{\td r}+8\frac{u_1M^2}{\rho\td r^2}\rt)C_2\sum_iF^{ii}.
\end{aligned}
\ee
Using the concavity of $F$ we get \be\label{v1change*}\sum\limits_iF^{ii}\geq F(I)=\frac{2}{n-1},\ee
here $I$ is the $n\times n$ identity matrix.
Denote $\Lambda\frac{M}{\td r}:=u_1(\hat x)\rho(\hat x)\geq \frac{\sqrt{3}}{2}C\frac{M}{\td r},$
\eqref{im-c1-mid-8} implies
\[
\begin{aligned}
0&\geq-(n-1)\frac{C_1M^2 f_\e}{\td r}+\frac{1}{4}\theta\Lambda^3\frac{M^3}{\td r^3}-4\frac{\Lambda M^3}{\td r^3}\\
&-C_2\lt(2\Lambda^2\frac{M^3}{\td r^3}+8\Lambda\frac{M^3}{\td r^3}\rt).
\end{aligned}
\]
This yields
\[0\geq\frac{M^3}{\td r^3}\lt(\frac{1}{4}\theta\Lambda^3-4\Lambda-2C_2\Lambda^2-8C_2\Lambda-C_3\e,\rt)\]
where $C_3=C_3(C_1, \ba, n)=C_3(\Omega, \{0\}, n)>0.$ Clearly, when $\Lambda>C_4=C_4(C_2, \theta)=C_4(n)>0$ the right hand side of the above inequality would be strictly positive, which leads to a contradiction. Therefore, $|Du(x_0)|<C_5\frac{M}{\td r}$ for some $C_5=C_5(n)>\frac{16}{\sqrt{3}}.$ The lemma follows from the arbitrariness of the choice of $x_0$.
\end{proof}

 \subsubsection{Inequality \eqref{im-up-1} for $x\in\ol{B_2(0)\setminus B_1(0)}$ }
 \label{subsub-im-c1-inside}
 We shall prove the following Lemma.
\begin{lemma}
 \label{im-c1-inside-lem}
Given $\eta_0=\eta_0(\Omega)>0, \ev=\ev(\Omega, \{0\}, n)>0$ sufficiently small, $R_0=R_0(\Omega)>0$ sufficiently large, for any $0<\e<\eta_0 R^{-8},$ $R>R_0,$ and $\varphi\in C^{\infty}(\bar\Omega)$ that satisfies $\varphi\leq0$ and $\|\varphi\|_{C^2}\leq \ev$ let $\ure$ be the admissible solution of \eqref{eq-appr}.
Then on $\ol{B_{2/R}\setminus B_{1/R}},$ $\ure$ satisfies
 \be\label{im-c1-inside}
\lt|D\ure\rt|<C|x|^{-\frac{1}{n-1}}.
 \ee
 Let $\hure$ be the admissible solution of \eqref{eq-appr-s}, then on $\ol{B_2\setminus B_1},$ $\hure$ satisfies
 \be\label{im-c1-inside-s}
 \lt|D\hure\rt|<C|x|^{-\frac{1}{n-1}}.
 \ee
Here, $C=C(\Omega, \{0\}, n)>0$ is a positive constant that is independent of $R$ and $\e.$
\end{lemma}
\begin{proof}
We shall follow the proof of Lemma \ref{c1-global-lem} and consider
$$W:=\max\limits_{x\in\ol{B_2\setminus B_1}, \xi\in\mathbb{S}^{n}}\lt\{D_\xi\hure+4\hure\rt\}.$$
By the same argument as in the proof of Lemma \ref{c1-global-lem} we know
$$\max\limits_{x\in\ol{B_2\setminus B_1}}W=\max\limits_{x\in\p(B_2\setminus B_1)}W\leq C,$$
where the last inequality comes from Proposition \ref{c0-prop}, Lemma \ref{c1-inside-lem}, Lemma \ref{im-c1-mid-lem}, and $C=C(\Omega, \{0\}, n)>0$.
This yields \eqref{im-c1-inside-s} directly and \eqref{im-c1-inside} follows by rescaling.
\end{proof}

\subsubsection{Inequality \eqref{im-up-1} for $x\in\ol{\Omega^R\setminus B_{R/2}}$ }
 \label{subsub-im-c1-outside}
 Similar to the Subsubsection \ref{subsub-im-c1-inside} we can show
 \begin{lemma}
 \label{im-c1-outside-lem}
Given $\eta_0=\eta_0(\Omega)>0, \ev=\ev(\Omega, \{0\}, n)>0$ sufficiently small, $R_0=R_0(\Omega)>0$ sufficiently large, for any $0<\e<\eta_0 R^{-8},$ $R>R_0,$ and $\varphi\in C^{\infty}(\bar\Omega)$ that satisfies $\varphi\leq0$ and $\|\varphi\|_{C^2}\leq \ev$ let $\ure$ be the admissible solution of \eqref{eq-appr}.
 Then on $\ol{\Omega\setminus B_{1/2}},$ $\ure$ satisfies
 \be\label{im-c1-outside}
\lt|D\ure\rt|<C|x|^{-\frac{1}{n-1}}.
 \ee
 Let $\hure$ be the admissible solution of \eqref{eq-appr-s}, then on $\ol{\Omega^R\setminus B_{R/2}},$ $\hure$ satisfies
 \be\label{im-c1-outside-s}
 \lt|D\hure\rt|<C|x|^{-\frac{1}{n-1}}.
 \ee
Here, $C=C(\Omega, \{0\}, n)>0$ is a constant that is independent of $R$ and $\e.$
\end{lemma}
\begin{proof}
This time we consider
$$W:=\max\limits_{x\in\ol{\Omega\setminus B_{1/2}}, \xi\in\mathbb{S}^{n}}\lt\{D_\xi\ure+8\ure\rt\}.$$
By the same argument as in the proof of Lemma \ref{c1-global-lem} we know
$$\max\limits_{x\in\ol{\Omega\setminus B_{1/2} }}W=\max\limits_{x\in\p\lt(\Omega\setminus B_{1/2}\rt)}W\leq C,$$
where the last inequality comes from Proposition \ref{c0-prop}, Lemma \ref{c1-outside-lem}, Lemma \ref{im-c1-mid-lem}, and $C=C(\Omega, \{0\}, n)>0$.
This yields \eqref{im-c1-outside} directly and \eqref{im-c1-outside-s} follows by rescaling.
\end{proof}

Combining Lemma \ref{im-c1-mid-lem}, Lemma \ref{im-c1-inside-lem}, and Lemma \ref{im-c1-outside-lem} we conclude
 \begin{lemma}
 \label{im-c1-upper-lem}
Given $\eta_0=\eta_0(\Omega)>0, \ev=\ev(\Omega, \{0\}, n)>0$ sufficiently small, $R_0=R_0(\Omega)>0$ sufficiently large, for any $0<\e<\eta_0 R^{-8},$ $R>R_0,$ and $\varphi\in C^{\infty}(\bar\Omega)$ that satisfies $\varphi\leq0$ and $\|\varphi\|_{C^2}\leq \ev$ let $\ure$ be the admissible solution of \eqref{eq-appr}.
Then $\ure$ satisfies
 \be\label{im-c1-upper}
\lt|D\ure\rt|<C|x|^{-\frac{1}{n-1}}.
 \ee
 Let $\hure$ be the admissible solution of \eqref{eq-appr-s}, then $\hure$ satisfies
 \be\label{im-c1-upper-s}
 \lt|D\hure\rt|<C|x|^{-\frac{1}{n-1}}.
 \ee
Here, $C=C(\Omega, \{0\}, n)>0$ is a constant that is independent of $R$ and $\e.$
\end{lemma}

\subsection{Improved $C^1$ lower bound}
 \label{sub-imp-c1-lower}
 Our goal in this subsection is to show
 \be\label{im-lower-1}
 |D\hure(x)|>\frac{c}{|x|^{\frac{1}{n-1}}}
 \ee
 for some constant $c=c(\Omega, \{0\}, n)>0$ that is independent of $R$ and $\e.$ We prove
  \begin{lemma}
 \label{im-c1-lower-lem}
 Given $\eta_0=\eta_0(\Omega)>0, \ev=\ev(\Omega, \{0\}, n)>0$ sufficiently small, $R_0=R_0(\Omega)>0$ sufficiently large, then for any $0<\e<\eta_0 R^{-8},$ $R>R_0,$ and $\varphi\in C^{\infty}(\bar\Omega)$ that satisfies $\varphi\leq0$ and $\|\varphi\|_{C^2}\leq \ev$ let $\ure$ be the admissible solution of \eqref{eq-appr}.
 Then $\ure$ satisfies
 \be\label{im-c1-lower}
\lt|D\ure\rt|>c|x|^{-\frac{1}{n-1}}.
 \ee
 Let $\hure$ be the admissible solution of \eqref{eq-appr-s}, then $\hure$ satisfies
 \be\label{im-c1-lower-s}
 \lt|D\hure\rt|>c|x|^{-\frac{1}{n-1}}.
 \ee
Here, $c=c(\Omega, \{0\}, n)>0$ is a constant that is independent of $R$ and $\e.$
\end{lemma}
\begin{proof}
We shall prove inequality \eqref{im-c1-lower-s}, inequality \eqref{im-c1-lower} then follows by rescaling.  Consider
$$\psi:=x\cdot D\hure-\beta\hure,$$
where $\beta>0$ to be determined.
For any $x\in\p\Omega^R,$ choose a local coordinate system
$\{\tau_1, \tau_2, \cdots, \tau_n\}$ such that $\tau_n$ is the outward unit normal of $\p\Omega^R$ at $x$ and $\tau_1, \cdots, \tau_{n-1}$ are the tangential vectors of
$\p\Omega^R$ at $x.$ Then at this point we have
\[x\cdot D\hure=\sum_{i=1}^n\lt<x, \tau_i\rt>\lt<D\hure, \tau_i\rt>\geq c_1R^{\frac{n-2}{n-1}}+\sum_{i=1}^{n-1}\lt<x, \tau_i\rt>\lt<D\hure, \tau_i\rt>.\]
Here, $c_1=c_1(\Omega, \{0\}, n)>0$ and the inequality follows from the fact that $\Omega^R$ is star-shaped and Lemma \ref{c1-outside-lem}.
Notice that $\hure-\lt(R^{\frac{n-2}{n-1}}+R^{\frac{n-2}{n-1}}\varphi\lt(\frac{x}{R}\rt)\rt)\equiv 0$ on $\p\Omega^R,$ we obtain
for $i=1, \cdots, n-1$
\[\lt|\lt<D\hure, \tau_i\rt>\rt|=\lt|\lt<R^{-\frac{1}{n-1}}D\varphi\lt(\frac{x}{R}\rt), \tau_i\rt>\rt|\leq\ev R^{-\frac{1}{n-1}}.\]
Therefore, when $\ev>0$ is sufficiently small, there exists
$\beta=\beta(\Omega, \{0\}, n)>0$ such that $\psi>0$ on $\p\Omega^R.$
Similarly, by Lemma \ref{c1-inside-lem} we can show there exists $\beta=\beta(\Omega, \{0\}, n)>0$ (it may be different from the one we chose earlier) such that $\psi>0$ on $\p B_1.$
In sum, we know there exists a $\beta=\beta(\Omega, \{0\}, n)>0$ such that $\psi>0$ on $\p\lt(\Omega^R\setminus B_1\rt).$

Now, since
\[\psi_{ij}=2(\hure)_{ij}+x_l(\hure)_{lij}-\beta(\hure)_{ij},\]
we get
\[
\begin{aligned}
F^{ij}\psi_{ij}&=(2-\beta)f_\e+f_\e\sum_lx_l\cdot\frac{x_l}{r}\lt(-r^{-1}-\frac{1}{n-1}(r+\e)^{-1}-\lt(r+\frac{2(n-1)\e}{n}\rt)^{-1}\rt)\\
&=(2-\beta)f_\e-f_\e\lt(1+\frac{r}{{(n-1)}(r+\e)}+\frac{r}{r+2(n-1)\e/n}\rt)<0,
\end{aligned}
\]
where we have used \eqref{fr}.
The maximum principle then implies
\[\min\limits_{x\in\ol{\Omega^R\setminus B_1}}\psi=\min\limits_{x\in\p(\Omega^R\setminus B_1)}\psi>0.\]
Therefore, we obtain
\[|x||D\hure|\geq x\cdot D\hure>\beta\hure.\]
Applying Proposition \ref{c0-prop} we complete the proof of \eqref{im-c1-lower-s} and \eqref{im-c1-lower} follows by rescaling.
\end{proof}

\subsection{Conclusion on the $C^1$ estimate}
\label{sub-c1-conclution}
In this subsection, we summerize the results in Subsection \ref{sub-imp-c1-up} and Subsection \ref{sub-imp-c1-lower}.

\begin{proposition}
\label{c1-prop}
Given $\eta_0=\eta_0(\Omega)>0, \ev=\ev(\Omega, \{0\}, n)>0$ sufficiently small, $R_0=R_0(\Omega)>0$ sufficiently large, for any $0<\e<\eta_0 R^{-8},$ $R>R_0,$ and $\varphi\in C^{\infty}(\bar\Omega)$ that satisfies $\varphi\leq0$ and $\|\varphi\|_{C^2}\leq \ev$ let $\ure$ be the admissible solution of \eqref{eq-appr}.
Then $\ure$ satisfies
\be\label{c1}
\bb_1 r^{-\frac{1}{n-1}}<|D\ure|\leq\ba_1 r^{-\frac{1}{n-1}}\,\,\mbox{on $\ol{\Omega\setminus B_{1/R}}.$}
\ee
Let $\hure$ be the admissible solution of \eqref{eq-appr-s}, then $\hure$ satisfies
\be\label{c1-s}
\bb_1 r^{-\frac{1}{n-1}}<|D\hure|\leq\ba_1 r^{-\frac{1}{n-1}}\,\,\mbox{on $\ol{\Omega^R\setminus B_1}.$}
\ee
Here $\ba_1=\ba_1(\Omega, \{0\}, n), \bb_1=\bb_1(\Omega, \{0\}, n)>0$ are positive constants that are independent of $R$ and $\e.$
\end{proposition}

As we can see that Proposition \ref{c1-prop} significantly improves the results in Corollary \ref{c1-bound-cor}. We shall use this improved $C^1$ estimate to
obtain better $C^2$ estimates on $\p(\Omega\setminus B_{1/R})$, which will be used in Section \ref{sec-AB}.

\subsection{Improved $C^2$ estimates}
\label{sub-imp-c2}
\subsubsection{Improved $C^2$ estimates on boundaries}
\label{subsub-im-c2-bd}
In this subsubsection, we shall start with proving improved $C^2$ estimates on $\p B_1$ for $\hure$. Then we shall carry out a similar argument on $\p\Omega$ for $\ure.$ We want to mention that the $C^2$ boundary estimates in the double tangential directions, that is, Lemma \ref{c2-inside-tan-lem} and Lemma \ref{c2-outside-tan-lem} are already optimal (i.e., they are of the same magnitude as the second derivative of the Riesz kernel), we only need to improve the rest.
\begin{lemma}
 \label{im-c2-inside-mix-lem}
 Given $\eta_0=\eta_0(\Omega)>0, \ev=\ev(\Omega, \{0\}, n)>0$ sufficiently small, $R_0=R_0(\Omega)>0$ sufficiently large, for any $0<\e<\eta_0 R^{-8},$ $R>R_0,$ and $\varphi\in C^{\infty}(\bar\Omega)$ that satisfies $\varphi\leq0$ and $\|\varphi\|_{C^2}\leq \ev$ let $\ure$ be the admissible solution of \eqref{eq-appr}.
 Then on $\p B_{1/R},$ $\ure$ satisfies
 \be\label{im-c2-inside-mix}
 \lt|\lt(\ure\rt)_{\tau\nu}\rt|<CR^{\frac{n}{n-1}}.
 \ee
 Let $\hure$ be the admissible solution of \eqref{eq-appr-s}, then on $\p B_1,$ $\hure$ satisfies
 \be\label{im-c2-inside-mix-s}
 \lt|\lt(\hure\rt)_{\tau\nu}\rt|<C.
 \ee
Here, $\tau$ is an arbitrary unit tangential vector of $\p B_1,$ $\nu$ is the inward unit normal of $\p B_1,$ and $C=C(\Omega, \{0\}, n)>0$ is a constant that is
independent of $R$ and $\e.$
 \end{lemma}
\begin{proof}
The proof here is almost the same as the proof of Lemma \ref{c2-inside-mix-lem}. The only difference is that instead of applying Corollary \ref{c1-bound-cor} to get
\eqref{c2-inside-1}, we now apply Proposition \ref{c1-prop} and get
\be\label{im-c2-inside-1}
 |T\hure|<C\,\,\mbox{ on $\p\td B_{\delta_0}\setminus\p B_1$}
 \ee
 for some $C=C(\Omega, \{0\}, n)>0.$ Therefore, there exists some $B=B(\theta, \delta_0, \Omega, \{0\}, n)>0$ such that
  on $\p\td B_{\delta_0}$ it holds
\be\label{v1.2}-Bh\leq T\hure\leq Bh;\ee
 and in $\td B_{\delta_0}$ it holds
 \[-B\frak L h>\frak LT\hure>B\frak Lh. \]
 By the standard maximum principle it is easy to see that \eqref{v1.2} holds in $\td B_{\delta_0}.$ This implies at $x_0$ we have
  \[-Bh_{\td n}< (T\hure)_{\td n}<Bh_{\td n}.\]
  Since $x_0\in \p B_1$ is arbitrary, we have completed the proof of \eqref{im-c2-inside-mix-s} and \eqref{im-c2-inside-mix} follows by rescaling.
\end{proof}

Consequently, we obtain a better estimate of $D^2\hure$ in the double normal direction.
\begin{lemma}
 \label{im-c2-inside-nu-lem}
Given $\eta_0=\eta_0(\Omega)>0, \ev=\ev(\Omega, \{0\}, n)>0$ sufficiently small, $R_0=R_0(\Omega)>0$ sufficiently large, for any $0<\e<\eta_0 R^{-8},$ $R>R_0,$ and $\varphi\in C^{\infty}(\bar\Omega)$ that satisfies $\varphi\leq0$ and $\|\varphi\|_{C^2}\leq \ev$ let $\ure$ be the admissible solution of \eqref{eq-appr}. Then on $\p B_{1/R},$ $\ure$ satisfies
 \be\label{im-c2-inside-nu}
 \lt|\lt(\ure\rt)_{\nu\nu}\rt|<CR^{\frac{n}{n-1}}.
 \ee
 Let $\hure$ be the admissible solution of \eqref{eq-appr-s}, then on $\p B_1,$ $\hure$ satisfies
 \be\label{im-c2-inside-nu-s}
 \lt|\lt(\hure\rt)_{\nu\nu}\rt|<C.
 \ee
Here, $\nu$ is the inward unit normal of $\p B_1,$ and $C=C(\Omega, \{0\}, n)>0$ is a constant that is independent of $R$ and $\e.$
 \end{lemma}
Next, we shall investigate $C^2$ estimates on $\p\Omega.$
\begin{lemma}
 \label{im-c2-outside-mix-lem}
Given $\eta_0=\eta_0(\Omega)>0, \ev=\ev(\Omega, \{0\}, n)>0$ sufficiently small, $R_0=R_0(\Omega)>0$ sufficiently large, for any $0<\e<\eta_0 R^{-8},$ $R>R_0,$ and $\varphi\in C^{\infty}(\bar\Omega)$ that satisfies $\varphi\leq0$ and $\|\varphi\|_{C^2}\leq \ev$ let $\ure$ be the admissible solution of \eqref{eq-appr}.
Then on $\p\Omega,$ $\ure$ satisfies
 \be\label{im-c2-outside-mix}
 \lt|\lt(\ure\rt)_{\tau\nu}\rt|<C.
 \ee
 Let $\hure$ be the admissible solution of \eqref{eq-appr-s}, then on $\p\Omega^R,$ $\hure$ satisfies
 \be\label{im-c2-outside-mix-s}
 \lt|\lt(\hure\rt)_{\tau\nu}\rt|<CR^{-\frac{n}{n-1}}.
 \ee
Here, $\tau$ is an arbitrary unit tangential vector of $\p\Omega,$ $\nu$ is the inward unit normal of $\p\Omega,$ and $C=C(\Omega, \{0\}, n)>0$ is a constant that is independent of $R$ and $\e.$
\end{lemma}
\begin{proof}
The proof here is almost the same as the proof of Lemma \ref{c2-outside-mix-lem}. The only difference is instead of applying Corollary \ref{c1-bound-cor} to get
\eqref{c2-outside-1}, we can now apply Proposition \ref{c1-prop} and get
\be\label{im-c2-outside-1}
 |T\lt(\ure-(1+\varphi)\rt)|<C\,\,\mbox{ on $\p\td B_{\delta_0}\setminus\p\Omega$}
 \ee
 for some $C=C(\Omega, \{0\}, n)>0.$
 Therefore, there exists some $B=B(\theta, \delta_0, \Omega, \{0\}, n)>0$ such that
 on $\p\td B_{\delta_0}$ it holds
 \be\label{v1.3}-Bh\leq T\lt(\ure-(1+\varphi)\rt)\leq Bh;\ee
  and in $\td B_{\delta_0}$ it holds
 \[-B\frak L h>\frak LT\lt(\ure-(1+\varphi)\rt)>B\frak Lh. \]
 By the standard maximum principle it is easy to see that \eqref{v1.3} holds in $\td B_{\delta_0}.$ This implies at $x_0$ we have
  \[-Bh_{\td n}< \lt(T\lt(\ure-(1+\varphi)\rt)\rt)_{\td n}<Bh_{\td n}.\]
  Since $x_0\in \p\Omega$ is an arbitrary point, we have completed the proof of \eqref{im-c2-outside-mix} and \eqref{im-c2-outside-mix-s} follows by rescaling.
\end{proof}
Consequently, we obtain a better estimate of $D^2\hure$ in the double normal direction.
\begin{lemma}
 \label{im-c2-outside-nu-lem}
Given $\eta_0=\eta_0(\Omega)>0, \ev=\ev(\Omega, \{0\}, n)>0$ sufficiently small, $R_0=R_0(\Omega)>0$ sufficiently large, for any $0<\e<\eta_0 R^{-8},$ $R>R_0,$ and $\varphi\in C^{\infty}(\bar\Omega)$ that satisfies $\varphi\leq0$ and $\|\varphi\|_{C^2}\leq \ev$ let $\ure$ be the admissible solution of \eqref{eq-appr}.
Then on $\p\Omega,$ $\ure$ satisfies
 \be\label{im-c2-outside-nu}
 \lt|\lt(\ure\rt)_{\nu\nu}\rt|<C.
 \ee
 Let $\hure$ be the admissible solution of \eqref{eq-appr-s}, then on $\p\Omega^R,$ $\hure$ satisfies
 \be\label{im-c2-outside-nu-s}
 \lt|\lt(\hure\rt)_{\nu\nu}\rt|<CR^{-\frac{n}{n-1}}.
 \ee
Here, $\nu$ is the inward unit normal of $\p\Omega,$ and $C=C(\Omega, \{0\}, n)>0$ is a constant that is independent of $R$ and $\e.$
\end{lemma}

 \section{Deformation process}
 \label{sec-deformation}
 Let $\hure$ be the admissible solution of \eqref{eq-appr-s}, in Section \ref{sec-deformation}, \ref{sec-step3}, and \ref{sec-AB}, we shall show there exists some universal positive constants
 $\ba_2=\ba_2(\Omega, \{0\}, n), \bb_2=\bb_2(\Omega, \{0\}, n)>0$ such that
 \[\bb_2|x|^{-\frac{n}{n-1}}<|D^2\hure(x)|<\ba_2|x|^{-\frac{n}{n-1}}\,\,\mbox{for all $x\in\ol{\Omega^R\setminus B_1}.$}\]
 We want to emphasize that $\ba_2$ and $\bb_2$ are independent of $R$ and $\e.$

 In this section, we shall set up our argument. Inspired by the approach underlying the proof of the constant rank theorem, which has been employed by various authors (see \cite{CF85,Kor87, GM03, WX25} for examples), we shall use a deformation argument.

 From now on, for our convenience, we shall only look at the rescaled equation \eqref{eq-appr-s}. Let $\hure$ be the admissible solution of \eqref{eq-appr-s}, assume the the eigenvalues of $D^2\hure$ are
 $\la(D^2\hure)=(\la_1, \la_2, \cdots, \la_{n})$ and $\la_1\geq\cdots\geq\la_{n-1}\geq\la_n.$ First, we want to prove that
 \be\label{def-1}\la_{n-1}\geq m_0|x|^{-\frac{n}{n-1}}\ee for some positive constant $m_0=m_0(\Omega, \{0\}, n)>0$ that is independent of $R$ and $\e.$ Then we shall use this result to prove $\la_1\leq M_0|x|^{-\frac{n}{n-1}}$ for some positive constant $M_0=M_0(\Omega, \{0\}, n)>0$ that is independent of $R$ and $\e.$ As one might expect, the first part is the key part of this paper.

 \subsection{Set up}
 \label{sub-setup} Our deformation argument goes as follows. Let $u^0=\bb\lt(r+\frac{2\e}{\bb}\rt)^{\frac{n-2}{n-1}}-\bb\lt(1+\frac{2\e}{\bb}\rt)^{\frac{n-2}{n-1}}+\ba,$ as we have shown in the proof of Lemma \ref{c1-inside-lem}, $u^0$ is a subsolution of \eqref{eq-appr-s} satisfying $F(D^2 u^0)=\bb f_{\frac{2\e}{\bb}}$. Now, let $R_1=R_1(\bb, R)=R_1(\Omega, \{0\}, n, R)>0$ (since $\bb=\bb(\Omega, \{0\}, n)$) be a positive constant such that $u^0=R^{\frac{n-2}{n-1}}$ on $\p B_{R_1}.$ It is clear that $R_1=O(R)$ and
 $\bar\Omega^R\subset B_{R_1}.$ We shall consider the following family of Dirichlet problems
 \be\label{eq-appr-t}
\left\{\begin{aligned}
F(D^2u)&=(1-t)\bb f_{\frac{2\e}{\bb}}+tf_{\e}\,\,&\mbox{in $\Omega^R(t)\setminus \bar B_1,$}\\
u&=R^\frac{n-2}{n-1}+tR^{\frac{n-2}{n-1}}\varphi\lt(\frac{x}{R}\rt)\,\,&\mbox{on $\p\Omega^R(t),$}\\
u&=\ba \,\,&\mbox{on $\p B_1,$}
\end{aligned}
\right.
 \ee
where $t\in[0, 1]$ and $\Omega^{R(t)}$ is the Minkowski sum of $(1-t)B_{R_1}$ and $t\Omega^R,$ i.e., $\Omega^R(t)=(1-t)B_{R_1}+t\Omega^R.$

\begin{lemma}
\label{convexity-lem}
For any $t\in[0, 1],$ the domain $\Omega^R(t)$ defined above is strictly convex. Moreover, let $\kappa_{\min}(\p\Omega^R(t))$ denote the minimum principal curvature of
$\p\Omega^R(t)$ and let $\kappa_{\max}(\p\Omega^R(t))$ denote the maximum principal curvature of
$\p\Omega^R(t).$ Then there exist positive constants $c_0=c_0(\Omega, \{0\}), C_0=C_0(\Omega, \{0\})>0$ that are independent of $R, \e,$ and $t$  such that
\[\kappa_{\min}(\p\Omega^R(t))\geq\frac{c_0}{R}\]
and
\[\kappa_{\max}(\p\Omega^R(t))\leq\frac{C_0}{R}.\]
\end{lemma}
\begin{proof}
We will denote by $h_A: \mathbb S^{n-1}\goto\R$ the support function of a convex domain $A\subset\mathbb R^n.$ It is well known that (see page 3 of \cite{Guan} for example)
\[h_{\Omega^R(t)}=(1-t)h_{B_{R_1}}+th_{\Omega^R}.\]
Denote $$\Lambda^A_{ij}:=\bar\nabla_{ij}h_A+h_A\delta_{ij},$$
where $\bar\nabla_{ij}h_A$ denotes the second order covariant derivatives of $h_A$ with respect to an arbitrary orthonormal
frame on $\mathbb S^{n-1}.$ It is also well known that the eigenvalues of $\Lambda_{ij}^A$ are principal radii of $\p A.$ Therefore, to prove this lemma, it suffices to show that
there exist some constants $c_0=c_0(\Omega, \{0\})$ and $C_0=C_0(\Omega, \{0\})$ such that
\[\frac{R}{C_0}\leq\la_{\min}\lt(\Lambda^{\Omega^{R}(t)}_{ij}\rt)\,\,\mbox{and}\,\,\la_{\max}\lt(\Lambda^{\Omega^{R}(t)}_{ij}\rt)\leq\frac{R}{c_0}.\]
In this proof, for any symmetric matrix $\lt(\Lambda_{ij}\rt),$ we use $\la_{\min}(\Lambda_{ij})$ to denote the smallest eigenvalue of $\lt(\Lambda_{ij}\rt)$
and $\la_{\max}(\Lambda_{ij})$ to denote the largest eigenvalue of $\lt(\Lambda_{ij}\rt).$

It is straightforward to verify that
\[\Lambda^{\Omega^{R}(t)}_{ij}=(1-t)\Lambda^{B_{R_1}}_{ij}+t\Lambda^{\Omega^R}_{ij}.\]
Therefore, we have
\[\la_{\min}\lt(\Lambda^{\Omega^{R}(t)}_{ij}\rt)\geq\max\lt\{(1-t)\la_{\min}\lt(\Lambda^{B_{R_1}}_{ij}\rt), t\la_{\min}\lt(\Lambda^{\Omega^{R}}_{ij}\rt)\rt\}\]
and
\[\la_{\max}\lt(\Lambda^{\Omega^{R}(t)}_{ij}\rt)\leq\max\lt\{\la_{\max}\lt(\Lambda^{B_{R_1}}_{ij}\rt), \la_{\max}\lt(\Lambda^{\Omega^{R}}_{ij}\rt)\rt\}.\]
The lemma follows immediately.
\end{proof}

From the lemma above, we see that—-with only minor modifications—-all results proved in Section \ref{solv-appr} and Section \ref{improve} continue to hold for $u^t$, the solution of \eqref{eq-appr-t}, for every $t\in[0, 1]$. We emphasize that, since there exist uniform upper and lower bounds for $\text{diam}(\Omega^R(t))$ and for the principal curvatures of $\p\Omega^R(t)$ , all independent of $t$, it follows that $u^t$ admits uniform estimates. In particular, $u^t$ and its first and second derivatives possess bounds that are independent of $t$.

For technical reasons, instead of proving \eqref{def-1}, we shall prove
\be\label{goal-1}
\la_{n-1}(D^2u^t)> m_0(u^t)^{\boldsymbol{\delta}-\frac{n}{n-2}},
\ee
where $\la_{n-1}(D^2u^t)$ is the second smallest eigenvalue of $D^2 u^t.$
Note that here and throughout this paper $\bd:=1/R$ is a fixed number.
By Proposition \ref{c0-prop}, it is easy to see that \eqref{goal-1} implies \eqref{def-1}. We also want to point out that the value of $m_0$ may change from line to line. The point is that there exists an $m_0=m_0(\Omega, \{0\}, n)>0$ that is independent of $R, \e$ and $t,$ such that \eqref{goal-1} holds.
In the following, we shall show that \eqref{goal-1} holds for $t=1$ in three steps:\\

1. We show that \eqref{goal-1} holds for $t=0.$\\

2. We show that \eqref{goal-1} holds on $\p(\Omega^R(t)\setminus B_1)$ for all $t\in(0, 1].$\\

3. We show that \eqref{goal-1} holds in $\Omega^R(t)\setminus \bar B_1$ for all $t\in(0, 1].$\\

\subsection{Proof of step 1 and step 2}
\label{sub-step12}
\begin{lemma}
\label{step1-lem}
Let $u^0=\bb\lt(r+\frac{2\e}{\bb}\rt)^{\frac{n-2}{n-1}}-\bb\lt(1+\frac{2\e}{\bb}\rt)^{\frac{n-2}{n-1}}+\ba,$ assume the eigenvalues of $D^2 u^0$ are
 $\la(D^2 u^0)=(\la_1, \la_2, \cdots, \la_n)$ and $\la_1\geq\cdots\geq\la_n.$ Then \eqref{goal-1} holds for any $m_0\in (0, \frac{n-2}{2(n-1)}\bb^{\frac{2(n-1)}{n-2}}].$
\end{lemma}
\begin{proof}
A straightforward calculation yields for $u^0$
\[\la_1=\cdots=\la_{n-1}=\frac{(n-2)\bb}{(n-1)r\lt(r+\frac{2\e}{\bb}\rt)^{\frac{1}{n-1}}}.\]
Therefore, for any $R>R_0$ sufficiently large
\[
\begin{aligned}
\frac{\la_{n-1}}{(u^0)^{\bd-\frac{n}{n-2}}}&=\frac{(n-2)\bb}{(n-1)r\lt(r+\frac{2\e}{\bb}\rt)^{\frac{1}{n-1}}}\cdot
\frac{\lt[\bb\lt(r+\frac{2\e}{\bb}\rt)^{\frac{n-2}{n-1}}-\bb\lt(1+\frac{2\e}{\bb}\rt)^{\frac{n-2}{n-1}}+\ba\rt]^{\frac{n}{n-2}}}
{\lt[\bb\lt(r+\frac{2\e}{\bb}\rt)^{\frac{n-2}{n-1}}-\bb\lt(1+\frac{2\e}{\bb}\rt)^{\frac{n-2}{n-1}}+\ba\rt]^{1/R}}\\
&>\frac{(n-2)\bb^{1+\frac{n}{n-2}}\lt(r+\frac{2\e}{\bb}\rt)}{(n-1)rR^{\frac{n-2}{(n-1)R}}}>\frac{n-2}{2(n-1)}\bb^{\frac{2(n-1)}{n-2}}.
\end{aligned}
\]
This proves the lemma.
\end{proof}

\begin{lemma}
\label{step2-lem}
Given $\eta_0=\eta_0(\Omega)>0, \ev=\ev(\Omega, \{0\}, n)>0$ sufficiently small, $R_0=R_0(\Omega)>0$ sufficiently large, for any $0<\e<\eta_0 R^{-8},$ $R>R_0,$ and $\varphi\in C^{\infty}(\bar\Omega)$ that satisfies $\varphi\leq0$ and $\|\varphi\|_{C^2}\leq \ev$ let $u^t$ be the admissible solution of \eqref{eq-appr-t}. Assume the eigenvalues of $D^2 u^t$ are
 $\la(D^2 u^t)=(\la_1, \la_2, \cdots, \la_n)$ and $\la_1\geq\cdots\geq\la_n.$ Then there exists some $m_0=m_0(\Omega, \{0\}, n)>0$ such that
 \eqref{goal-1} holds on $\p(\Omega^R(t)\setminus B_1)$ for all $t\in(0, 1].$
\end{lemma}
\begin{proof}We shall prove \eqref{goal-1} on $\p\Omega^R(t),$ the proof on $\p B_1$ is similar and will therefore be omitted.
By Lemma \ref{convexity-lem} we know there exists $C_0=C_0(\Omega, \{0\}), c_0=c_0(\Omega, \{0\})>0$ that are independent of $R, \e,$ and $t$ such that
\[\frac{C_0}{R}\geq\kappa_{\max}(\p\Omega(t))\geq\kappa_{\min}(\p\Omega^R(t))\geq \frac{c_0}{R}.\]
Following the proof of Lemma \ref{c1-outside-lem} and Lemma \ref{c2-outside-tan-lem} we obtain, there exist some positive constants $C_1=C_1(\Omega, \{0\}, n),\,\,c_1=c_1(\Omega, \{0\}, n)>0$ that are independent of $R, \e,$ and $t$ such that for any $x\in\p\Omega^R(t)$ we can choose a local coordinate system at $x,$ and under this coordinate system the double tangential derivatives of $u^t$ at $x$ satisfy
\be\label{step2-1}
c_1R^{-\frac{n}{n-1}}\delta_{\al\beta}\leq u^t_{\al\beta}\leq C_1R^{-\frac{n}{n-1}}\delta_{\al\beta}.
\ee

In particular, consider an arbitrary point $p\in\p\Omega^R(t),$ we shall choose a local coordinate system $\{x_1, \cdots, x_n\}$ in a neighborhood of $p$ such that the $x_n$ axis is the inward normal of $\p\Omega^R(t)$ at $p$ and $u^t_{ij}(p)=u^t_{ii}(p)\delta_{ij}$ for $1\leq i, j\leq n-1.$
Then the eigenvalues of $D^2u^t$ at $p$ satisfies
\be\label{eq-eigen}f(\la):=\prod\limits_{i=1}^n(\la-u_{ii}^t)-\sum\limits_{s=1}^{n-1}(u^t_{ns})^2\prod\limits_{i\neq s, n}(\la-u^t_{ii})=0.\ee
Let $\la_1,\cdots, \la_n $ denote the $n$ roots of \eqref{eq-eigen}.
To prove \eqref{goal-1} holds on $\p\Omega^R(t)$, we only need to prove there are $n-1$ positive roots of \eqref{eq-eigen} that are greater than $\frac{c_2}{R^{\frac{n}{n-1}}}$ for some $c_2=c_2(\Omega, \{0\}, n)>0.$\\

Without loss of generality, we may assume $u^t_{ii}\neq u^t_{jj}$ for any $1\leq i\neq j\leq n-1$ and we can always arrange them such that $u_{11}^t>\cdots>u^t_{n-1n-1}.$ We may also assume $u_{ni}^t\neq 0$ for $1\leq i\leq n-1.$ If not, we may always consider a sequence of symmetric matrix $U^k:=(u^k_{ij})$ such that $u_{11}^k>\cdots>u^k_{n-1n-1},$ $u_{ni}^k\neq 0$ for $1\leq i\leq n-1,$ $u^k_{ij}=0$ for $1\leq i\neq j\leq n-1,$ and $\{U^{k}\}\goto (u^t_{ij})$ as $k\goto\infty.$ The result will follow from the continuity of eigenvalues.
It is clear that $$\lim\limits_{\la\goto\infty}f(\la)\goto\infty,$$ and under our assumption we also have
\[(-1)^if(u_{ii})>0\,\,\mbox{for}\,\,1\leq i\leq n-1.\]
By the intermediate value theorem we know the roots of \eqref{eq-eigen} satisfy $\la_1>u^t_{11}$ and $u^t_{ii}<\la_i<u^t_{i-1i-1}$ for $2\leq i\leq n-1.$
In view of \eqref{step2-1} the lemma is proved.
\end{proof}

Finally, we conclude that
\begin{lemma}
\label{step12-lem}
Given $\eta_0=\eta_0(\Omega)>0, \ev=\ev(\Omega, \{0\}, n)>0$ sufficiently small, $R_0=R_0(\Omega)>0$ sufficiently large, for any $0<\e<\eta_0 R^{-8},$ $R>R_0,$ and $\varphi\in C^{\infty}(\bar\Omega)$ that satisfies $\varphi\leq0$ and $\|\varphi\|_{C^2}\leq \ev$ let $u^t$ be the admissible solution of \eqref{eq-appr-t}. Assume the eigenvalues of $D^2 u^t$ are $\la(D^2 u^t)=(\la_1, \cdots, \la_n)$ and $\la_1\geq\cdots\geq\la_n.$ Then, there exists some positive constant $m_0=m_0(\Omega, \{0\}, n)$ that is independent of
$R, \e,$ and $t$ such that
\eqref{goal-1} holds at $t=0$ and on $\p(\Omega^R(t)\setminus B_1)$ for all $t\in(0, 1].$
\end{lemma}

\section{Proof of step 3}
 \label{sec-step3}

 In this section, we shall prove \eqref{goal-1} holds on $\Omega^R(t)\setminus \bar B_1$ for all $t\in(0, 1].$ Before starting the proof, we need to review the following well-established regularity theory for fully nonlinear elliptic equations, i.e., Evans-Krylov theorem (see \cite{Evans, K85, GT83, CW98}).
 \begin{theorem}
 \label{regularity}
 Consider the fully nonlinear, uniformly elliptic equation
 \be\label{reg-1}
\left\{\begin{aligned}
F(D^2u)&=f(x)\,\,&\mbox{in $\Omega,$}\\
 u&=\varphi\,\,&\mbox{on $\p\Omega,$}
 \end{aligned}
 \right.
 \ee
 where $\la|\xi|^2\leq \frac{\p F}{\p u_{ij}}\xi_i\xi_j\leq\Lambda|\xi^2|$ for all $\xi\in\mathbb R^n$ and $\Lambda>\lambda>0$ are some positive constants.
 Suppose $F$ is concave, $F\in C^{\infty},$ $f\in C^\infty(\bar\Omega),$ $\varphi\in C^{\infty}(\bar\Omega),$ $\p\Omega\in C^{\infty},$ and $u\in C^{\infty}(\bar\Omega)$ is a solution of \eqref{reg-1}. Then
 \[\|u\|_{C^{2, \alpha}(\bar\Omega)}\leq C(\|u\|_{C^2(\bar\Omega)}+\|\varphi\|_{C^{2, \alpha}(\bar\Omega)}+\|f\|_{C^{\alpha}(\bar\Omega)}),\]
 where $\al=\al(n, \la, \Lambda)\in (0, 1)$ and $C=C(n, \la, \Lambda, \alpha, \Omega).$
 \end{theorem}
 Differentiating equation \eqref{reg-1} and applying the Schauder estimates for linear uniformly elliptic equations, we obtain $C^{3, \alpha}$ estimates of $u.$
 In particular
 \begin{theorem}
 \label{higher-regularity}
 Consider the fully nonlinear, uniformly elliptic equation
 \be\label{reg-2}
\left\{\begin{aligned}
F(D^2u)&=f(x)\,\,&\mbox{in $\Omega,$}\\
 u&=\varphi\,\,&\mbox{on $\p\Omega,$}
 \end{aligned}
 \right.
 \ee
 where $\la|\xi|^2\leq \frac{\p F}{\p u_{ij}}\xi_i\xi_j\leq\Lambda|\xi^2|$ for all $\xi\in\mathbb R^n$ and $\Lambda>\lambda>0$ are some positive constants.
 Suppose $F$ is concave, $F\in C^{\infty},$ $f\in C^\infty(\bar\Omega),$ $\varphi\in C^{\infty}(\bar\Omega),$ $\p\Omega\in C^{\infty},$ and $u\in C^{\infty}(\bar\Omega)$ is a solution of \eqref{reg-2}. Then
 \[\|u\|_{C^{3, \alpha}(\bar\Omega)}\leq C(\|u\|_{C^1(\bar\Omega)}+\|\varphi\|_{C^{3, \alpha}(\bar\Omega)}+\|f\|_{C^{1,\alpha}(\bar\Omega)}),\]
 where $\al=\al(n, \la, \Lambda)\in (0, 1)$ and $C=C\left(n, \la, \Lambda, \|u\|_{C^{2, \alpha}(\bar\Omega)}, \alpha, \Omega\right).$
 \end{theorem}

 The proof of step 3 is inspired by \cite{GM03, BCD17, LT24}.

 Let $m_0$ be chosen in Lemma \ref{step12-lem}. Denote $Q:=\la_{n-1}(u^t)^{\frac{n}{n-2}-\bd}$ and let $t_0\in (0, 1]$ be the first time such that $Q=m_0$  is reached at some $x_0\in\ol{\Omega^R(t_0)\setminus B_1}.$ In view of Lemma \ref{step12-lem} we know that $x_0\in \Omega^R(t_0)\setminus\bar B_1$ is an interior point. We shall denote the eigenvalues of $(u_{ij}^{t_0})$ by $\la(u^{t_0}_{ij})=(\la_1, \cdots, \la_n)$
and we shall always assume $\la_1\geq\cdots\geq\la_n.$ Moreover, in the following, for convenience, we will write $u$ instead of $u^{t_0}$.
We shall start with the following important observations.\\

\textbf{Claim A:} For any $0<\e<\e_R,$ where $0<\e_R=\e_R(\Omega, \{0\}, n, R)<\eta_0R^{-8}$ is sufficiently small, we have
$\la_n<-\frac{\s_{n-1}(\la|n)}{2\s_{n-2}(\la|n)}\leq-c(n)\la_{n-1}<0$ for all $t\in [0, t_0],$ where $c(n)>0$ is a positive constant that only depends on $n.$

Before proving Claim A, we want to fix our notation. We shall denote
$h:=(1-t_0)\bb f_{\frac{2\e}{\bb}}+t_0f_\e.$ By \eqref{rangef} we can see that
\be\label{eq-h}
2f_\e> h\geq f_\e.
\ee
\textbf{Proof of Claim A:} We only need to prove Claim A for $t=t_0,$ when $t<t_0$ the proof is the same.
Recall that $u$ satisfies
\be\label{eq-DE}
\frac{\s_{n-1}}{\s_{n-2}}=h.
\ee
and $\la_{n-1}(D^2u)\geq m_0u^{\bd-\frac{n}{n-2}}.$ In view of Proposition \ref{c0-prop}, this implies there exists $m_1=m_1(\Omega, \{0\}, n)>0$ so that
$\la_{n-1}\geq m_1 r^{-\frac{n}{n-1}}$ for all $x\in\ol{\Omega^R(t_0)\setminus B_1}.$
Since \eqref{eq-DE} can be written as
\[\la_n\s_{n-2}(\la|n)=h\s_{n-2}(\la)-\s_{n-1}(\la|n),\]
where $(\la|i)=(\la_1, \cdots, \la_{i-1}, \la_{i+1}, \cdots, \la_n)\in\mathbb R^{n-1}.$
This implies
\be\label{gen-step3-1}\la_n=-\frac{\s_{n-1}(\la|n)}{\s_{n-2}(\la|n)}+\frac{h\s_{n-2}(\la)}{\s_{n-2}(\la|n)}.\ee
Note that Lemma \ref{c2-global-lem} (equation \eqref{c2-global-s}) gives
\be\label{v1.4}\s_{n-2}(\la)\leq CR^{\frac{2(n-2)^2}{n-1}},\ee
where $C=C(\Omega, \{0\}, n)>0.$
Moreover, by our assumption we have $\s_{n-1}(\la|n)\geq\la_{n-1}^{n-1}\geq m_2R^{-n}$ for some $m_2=m_2(\Omega, \{0\}, n)>0.$ It is easy to see that when
$\e_R=\e_R(\Omega, \{0\}, n, R)>0$ is sufficiently small the claim follows.\\

\textbf{Claim B:} For any $0<\e<\e_R,$ where $0<\e_R=\e_R(\Omega, \{0\}, n, R)<\eta_0R^{-8}$ is sufficiently small, there exits $m_3=m_3(\Omega, \{0\}, n)>0$
such that $\s_{n-2}(\la)>m_3 R^{-\frac{n(n-2)}{n-1}}$ for all $t\in [0, t_0].$\\
\textbf{Proof of Claim B:} By virtue of \eqref{gen-step3-1} and Newton's inequality we get
\[\begin{aligned}
\s_{n-2}(\la)&=\la_n\s_{n-3}(\la|n)+\s_{n-2}(\la|n)\\
&=\frac{[h\s_{n-2}(\la)-\s_{n-1}(\la|n)]}{\s_{n-2}(\la|n)}\s_{n-3}(\la|n)+\s_{n-2}(\la|n)\\
&\geq\frac{h\s_{n-2}(\la)\s_{n-3}(\la|n)}{\s_{n-2}(\la|n)}+\frac{n}{2(n-1)}\s_{n-2}(\la|n)\geq\frac{n}{2(n-1)}\s_{n-2}(\la|n).
\end{aligned}\]
Since $\la_{n-1}\geq m_1 r^{-\frac{n}{n-1}}$ for all $x\in\ol{\Omega^R(t_0)\setminus B_1}$ the claim follows.\\

\textbf{Claim C:} $\|u\|_{C^3(\overline{\Omega^R(t)\setminus B_1})}<C$ for all $t\in [0, t_0],$ where $C=C(\Omega, \{0\}, n, \varphi, R)>0$ is some positive constant that is independent of $\e$.\\
\textbf{Proof of Claim C:} We only need to prove Claim C for $t=t_0,$ when $t<t_0$ the proof is the same.
A straightforward calculation yields the eigenvalue of $(F^{ij})$ are
\be\label{fk}
\df^k=\frac{\s_{n-2}(\la|k)}{\s_{n-2}(\la)}-h\frac{\s_{n-3}(\la|k)}{\s_{n-2}(\la)},\,\,1\leq k\leq n.
\ee
By the concavity of $F$ and \eqref{gen-step3-1} we have
\[\begin{aligned}
\df^k&\geq\df^1=\frac{\s_{n-2}(\la|1)-h\s_{n-3}(\la|1)}{\s_{n-2}(\la)}\\
&=\frac{\la_n\s_{n-3}(\la|1n)+\s_{n-2}(\la|1n)-h\s_{n-3}(\la|1)}{\s_{n-2}(\la)}\\
&\geq\frac{1}{\s_{n-2}(\la)}\lt(-\frac{\s_{n-1}(\la|n)\s_{n-3}(\la|1n)}{\s_{n-2}(\la|n)}+\s_{n-2}(\la|1n)-h\s_{n-3}(\la|1)\rt).
\end{aligned}
\]
Since
\[\begin{aligned}
&\s_{n-2}(\la|1n)\s_{n-2}(\la|n)-\s_{n-1}(\la|n)\s_{n-3}(\la|1n)\\
&=\s_{n-2}(\la|1n)[\la_1\s_{n-3}(\la|1n)+\s_{n-2}(\la|1n)]-\la_1\s_{n-2}(\la|1n)\s_{n-3}(\la|1n)\\
&=\s_{n-2}^2(\la|1n),
\end{aligned}\]
when $\e_R>0$ is sufficiently small we have
\be\label{gen-step3-2}
\begin{aligned}
\df^k&\geq\frac{\s_{n-2}^2(\la|1n)-h\s_{n-3}(\la|1)\s_{n-2}(\la|n)}{\s_{n-2}(\la)\s_{n-2}(\la|n)}\\
&\geq m_4\lt(\frac{\la_{n-1}}{\la_1}\rt)^{2(n-2)}\\
&\geq m_4R^{-\lt(3-\frac{1}{n-1}\rt)2(n-2)}
\end{aligned}
\ee
for some $m_4=m_4(\Omega, \{0\}, n)>0.$
On the other hand, it is easy to see that
\be\label{gen-step3-3}
\sum\limits_k\df^k<\sum\limits_k\frac{\s_{n-2}(\la|k)}{\s_{n-2}(\la)}=2.
\ee
Therefore, at $t=t_0$ we get
\[ m_4R^{-\lt(3-\frac{1}{n-1}\rt)2(n-2)}|\xi|^2\leq F^{ij}\xi_i\xi_j\leq 2|\xi|^2\,\,\mbox{for all $\xi\in\mathbb R^n.$}\]
Applying  Theorem \ref{regularity} and Theorem \ref{higher-regularity} we confirm Claim C.

In view of Claim C, we shall introduce the following notation. For any two well-defined functions $h(x)$ and $k(x),$ we
say that $h(x)\lesssim k(x)$ provided there exists a positive constants $C$ such
that $h(x)\leq k(x)+C\e,$ where $C=C(\|u\|_{C^3}, R, \Omega, \{0\}, n)$ is independent of $\e.$ We also write $h(x)\simeq k(x)$ if
$h(x)\lesssim k(x)$ and $k(x)\lesssim h(x).$

\begin{lemma}
\label{step3-lem}
Given $\eta_0=\eta_0(\Omega)>0, \ev=\ev(\Omega, \{0\}, n)>0, \eta_0R^{-8}>\e_R=\e_R(\Omega, \{0\}, n, \varphi, R)>0$ sufficiently small, $R_0=R_0(\Omega)>0$ sufficiently large, for any $0<\e<\e_R,$ $R>R_0,$ and $\varphi\in C^{\infty}(\bar\Omega)$ that satisfies $\varphi\leq0$ and $\|\varphi\|_{C^2}\leq \ev$ let $u^t$ be the admissible solution of \eqref{eq-appr-t}. Assume the eigenvalues of $D^2 u^t$ are
 $\la(D^2 u^t)=(\la_1, \cdots, \la_n)$ and $\la_1\geq\cdots\geq\la_n.$ Then, for the $m_0$ chosen in Lemma \ref{step12-lem},
inequality \eqref{goal-1} holds on $\Omega^R(t)\setminus \bar B_1$ for all $t\in(0, 1].$
\end{lemma}
\begin{proof}
Let us assume $Q:=\la_{n-1}(u^t)^{\frac{n}{n-2}-\bd}=m_0$ is achieved for the first time at $t=t_0\in (0, 1]$ and $x_0\in {\Omega^R(t_0)\setminus \bar B_1}.$ We will show this cannot happen by using a proof by contradiction. In the following, we shall drop the superscript $t$ and write $u$ instead of $u^{t_0}.$
We may choose an orthonormal frame at $x_0$ such that $u_{ij}(x_0)=\la_i\delta_{ij}$ for $1\leq i, j\leq n.$

\textbf{Case 1.} At $x_0,$ $\la_{n-1}$ has multiplicity one.

In this case, by the implicit function theorem 
we know $\la_{n-1}(D^2u)$ is a smooth function in a small neighborhood of $x_0.$ Therefore,
at $x_0$ we have
\be\label{gen-7.1.1}
0=\frac{Q_i}{Q}=\frac{\la_{n-1i}}{\la_{n-1}}+\lt(\frac{n}{n-2}-\bd\rt)\frac{u_i}{u},\,\,1\leq i\leq n,
\ee
and
\be\label{gen-7.1.2}
\begin{aligned}
0&\leq\frac{\la_{n-1ii}}{\la_{n-1}}-\lt(\frac{\la_{n-1i}}{\la_{n-1}}\rt)^2+\lt(\frac{n}{n-2}-\bd\rt)\frac{u_{ii}}{u}-\lt(\frac{n}{n-2}-\bd\rt)\lt(\frac{u_i}{u}\rt)^2\\
&=\frac{\la_{n-1ii}}{\la_{n-1}}-\lt[1+\lt(\frac{n}{n-2}-\bd\rt)^{-1}\rt]\lt(\frac{\la_{n-1i}}{\la_{n-1}}\rt)^2+\lt(\frac{n}{n-2}-\bd\rt)\frac{u_{ii}}{u}.
\end{aligned}
\ee
Denote
\be\label{gen-eta}
\begin{aligned}
\bde:&=1+\lt(\frac{n}{n-2}-\bd\rt)^{-1}-\lt(1+\frac{n-2}{n}\rt)\\
&=\lt(\frac{n}{n-2}-\bd\rt)^{-1}-\frac{n-2}{n}=\frac{(n-2)^2\bd}{n[n-(n-2)\bd]}=O(\frac{1}{R}),
\end{aligned}\ee
contracting with $F^{ii},$ \eqref{gen-7.1.2} becomes
\be\label{gen-7.1.3}
0\leq F^{ii}\frac{\la_{n-1ii}}{\la_{n-1}}-\lt[\lt(1+\frac{n-2}{n}\rt)+\bde\rt]F^{ii}\lt(\frac{\la_{n-1i}}{\la_{n-1}}\rt)^2+\lt(\frac{n}{n-2}-\bd\rt)\frac{h}{u},
\ee
where $h=(1-t_0)\bb f_{\frac{2\e}{\bb}}+t_0f_\e.$
Recall Lemma \ref{np-lem1} we obtain at $x_0$ for $1\leq i\leq n,$
\be\label{gen-7.1.4}
\la_{n-1ii}=u_{n-1n-1ii}+2\sum\limits_{p\neq n-1}\frac{1}{\la_{n-1}-\la_p}u^2_{n-1pi}.
\ee
This gives
\[\mathcal L\la_{n-1}\simeq-F^{pq, rs}u_{pqn-1}u_{rsn-1}+2\sum\limits_i\sum\limits_{p\neq n-1}\frac{F^{ii}}{\la_{n-1}-\la_p}u^2_{n-1pi},\]
where $\mathcal L:=F^{ii}\p_{i}\p_i.$
A straightforward calculation yields
\[F^{pq}=\frac{\s_{n-1}^{pq}}{\s_{n-2}}-\frac{\s_{n-1}}{\s_{n-2}^2}\s_{n-2}^{pq},\]
\[\begin{aligned}
F^{pq, rs}&=\frac{\s_{n-1}^{pq, rs}}{\s_{n-2}}-\frac{\s_{n-1}^{pq}\s_{n-2}^{rs}}{\s_{n-2}^2}-\frac{\s_{n-1}^{rs}\s_{n-2}^{pq}}{\s_{n-2}^2}\\
&+\frac{2\s_{n-1}}{\s^3_{n-2}}\s^{pq}_{n-2}\s^{rs}_{n-2}-\frac{\s_{n-1}}{\s^2_{n-2}}\s^{pq, rs}_{n-2}.
\end{aligned}
\]
Therefore, we get
\[
\begin{aligned}
F^{pq, rs}u_{pqn-1}u_{rsn-1}&=\frac{\s^{pq, rs}_{n-1}u_{pqn-1}u_{rsn-1}}{\s_{n-2}}
-2\frac{(\s_{n-1})_{n-1}(\s_{n-2})_{n-1}}{\s^2_{n-2}}\\
&+\frac{2\s_{n-1}}{\s^3_{n-2}}(\s_{n-2})^2_{n-1}-\frac{\s_{n-1}}{\s^2_{n-2}}\s^{pq, rs}_{n-2}u_{pqn-1}u_{rsn-1}.
\end{aligned}
\]
Note that $\frac{\s_{n-1}}{\s_{n-2}}\simeq\lt(\frac{\s_{n-1}}{\s_{n-2}}\rt)_i\simeq 0$ for $1\leq i\leq n,$
in view of \eqref{v1.4}, Claim B, and Claim C we obtain
\be\label{gen-7.1.0}
\s_{n-1}\simeq(\s_{n-1})_i\simeq 0, 1\leq i\leq n.
\ee
This implies at $x_0$
\be\label{gen-7.1.5}
\mathcal L\la_{n-1}\simeq-\frac{\G^{pq, rs}u_{pqn-1}u_{rsn-1}}{\s_{n-2}}+2\sum\limits_i\sum\limits_{p\neq n-1}\frac{\G^{ii}u^2_{n-1pi}}{\s_{n-2}(\la_{n-1}-\la_p)}.
\ee
In view of Lemma \ref{np-lem1} and Lemma \ref{np-lem2}, plugging \eqref{gen-7.1.5} into \eqref{gen-7.1.3} we have at $x_0$
\be\label{gen-7.1.6}
\begin{aligned}
0&\lesssim-2\sum\limits_{p< q}\G^{pp, qq}u_{ppn-1}u_{qqn-1}+2\sum_{p< q}\G^{pp, qq}u^2_{pqn-1}\\
&+2\sum_i\sum_{p\neq n-1}\frac{\G^{ii}u^2_{n-1pi}}{\la_{n-1}-\la_p}-\lt[\lt(1+\frac{n-2}{n}\rt)+\bde\rt]\sum_i\frac{\G^{ii}u^2_{n-1n-1i}}{\la_{n-1}}.
\end{aligned}
\ee
Denote $\clubsuit:=-\s_n(\la),$ by Claim A we know $\clubsuit>c(\Omega, \{0\}, n, R)>0.$ Following the ingenious observation in \cite{LT24} we obtain,
\be\label{gen-*}\la_i\G^{ii}=\la_i\s_{n-2}(\la|i)=\s_{n-1}-\s_{n-1}(\la|i)\simeq-\s_{n-1}(\la|i)=\frac{\clubsuit}{\la_i},\ee
which yields
\be\label{gen-gi}\G^{ii}\simeq\frac{\C}{\la_i^2}.\ee
For $p\neq q,$ since
\[
\begin{aligned}
\la_p\G^{pp, qq}&=\la_p\s_{n-3}(\la|pq)=\s_{n-2}(\la|q)-\s_{n-2}(\la|pq)\\
&=\G^{qq}-\frac{\s_n}{\la_p\la_q}\simeq\frac{\C}{\la^2_q}+\frac{\C}{\la_p\la_q},
\end{aligned}
\]
we get
\be\label{gen-gpq}
\G^{pp, qq}\simeq\frac{\C(\la_p+\la_q)}{\la_p^2\la_q^2}\,\,\mbox{for}\,\,p\neq q.
\ee
In the following, we shall focus on the right hand side of \eqref{gen-7.1.6}. Denote
$$\text{Co}(\text{term A}):=\text{Coefficient of term A },$$ then when $p<q$ and $p, q\neq n-1$ we have
\be\label{gen-7.1.7}
\text{Co}(u^2_{pqn-1})=2\G^{pp, qq}+\frac{2\G^{pp}}{\la_{n-1}-\la_q}+2\frac{\G^{qq}}{\la_{n-1}-\la_p}.
\ee
For our convenience, for any two functions $A$ and $B,$ we will denote $A\propto B$ if $\frac{A}{B}>0.$
Therefore, when $p<q<n-1$ we have
\be\label{gen-7.1.8}
\begin{aligned}
\text{Co}(u^2_{pqn-1})&\simeq2\C\lt[\frac{\la_p+\la_q}{\la_p^2\la_q^2}+\frac{1}{\la_p^2(\la_{n-1}-\la_q)}+\frac{1}{\la_q^2(\la_{n-1}-\la_p)}\rt]\\
&\propto(\la_p+\la_q)(\la_{n-1}-\la_q)(\la_{n-1}-\la_q)+\la_q^2(\la_{n-1}-\la_p)+\la_p^2(\la_{n-1}-\la_q)\\
&=\la_{n-1}^2(\la_p+\la_q)-2\la_p\la_q\la_{n-1}\\
&\propto (a+b)-2ab,
\end{aligned}
\ee
where $a:=\frac{\la_p}{\la_{n-1}},$ $b:=\frac{\la_q}{\la_{n-1}}.$ By our assumption we have $a, b>1,$ thus we conclude
\be\label{gen-7.1.10}\text{Co}(u^2_{pqn-1})\lesssim 0\,\,\mbox{for}\,\,p< q<n-1.\ee

When $p<n-1,$ since $\la_n<0$ and $\la\in\Gamma_{n-1}$ (that is $\la_{n-1}+\la_n>0$), we derive
\be\label{gen-7.1.9}
\begin{aligned}
\text{Co}(u^2_{pnn-1})&\simeq2\C\lt[\frac{\la_p+\la_n}{\la_p^2\la_n^2}+\frac{1}{\la_p^2(\la_{n-1}-\la_n)}-\frac{1}{\la_n^2(\la_p-\la_{n-1})}\rt]\\
&\propto(\la_p+\la_n)(\la_{n-1}-\la_n)(\la_p-\la_{n-1})+\la_n^2(\la_p-\la_{n-1})-\la_p^2(\la_{n-1}-\la_n)\\
&=-\la_{n-1}^2(\la_p+\la_n)+2\la_p\la_n\la_{n-1}<0.\\
\end{aligned}
\ee

Combining \eqref{gen-7.1.10} and \eqref{gen-7.1.9} with \eqref{gen-7.1.6} we obtain
\be\label{gen-7.1.6*}
\begin{aligned}
0&\lesssim-2\sum_{p<q}\G^{pp, qq}u_{ppn-1}u_{qqn-1}+2\sum_{p\neq n-1}\G^{pp, n-1n-1}u^2_{n-1n-1p}
+2\sum_{p\neq n-1}\frac{\G^{n-1n-1}u^2_{n-1n-1p}}{\la_{n-1}-\la_p}\\
&+2\sum_{p\neq n-1}\frac{\G^{pp}u^2_{ppn-1}}{\la_{n-1}-\la_p}-\lt[\lt(1+\frac{n-2}{n}\rt)+\bde\rt]\sum_i\frac{\G^{ii}u^2_{n-1n-1i}}{\la_{n-1}}.
\end{aligned}
\ee
We can see that for $p\neq n-1$
\[\text{Co}(u^2_{n-1n-1p})=2\G^{pp, n-1n-1}+2\frac{\G^{n-1n-1}}{\la_{n-1}-\la_p}-\lt[\lt(1+\frac{n-2}{n}\rt)+\bde\rt]\frac{\G^{pp}}{\la_{n-1}}.\]
Therefore,
\[\text{Co}(u^2_{n-1n-1n})\simeq2\C\lt[\frac{\la_n+\la_{n-1}}{\la_n^2\la_{n-1}^2}+\frac{1}{\la^2_{n-1}(\la_{n-1}-\la_n)}-\frac{n-1}{n\la_n^2\la_{n-1}}\rt]
-\bde\frac{\G^{nn}}{\la_{n-1}}.\]
In view of \eqref{gen-step3-1} we have
\[\begin{aligned}
-\la_n&\simeq\frac{\s_{n-1}(\la|n)}{\s_{n-2}(\la|n)}=\la_{n-1}\frac{\s_{n-1}}{\s_{n-2}}\lt(\frac{\la_1}{\la_{n-1}}, \cdots, \frac{\la_{n-2}}{\la_{n-1}}, 1\rt)\\
&>\la_{n-1}\frac{\s_{n-1}}{\s_{n-2}}(1, \cdots, 1)=\frac{\la_{n-1}}{n-1}.
\end{aligned}\]
Denote $a:=\frac{\la_n}{\la_{n-1}}$ then $-1<a\lesssim-\frac{1}{n-1}.$ Note that $a>-1$ comes from $\la(D^2u)\in\Gamma_{n-1}.$
Thus,
\[\begin{aligned}
&\frac{\la_n+\la_{n-1}}{\la_n^2\la_{n-1}^2}+\frac{1}{\la^2_{n-1}(\la_{n-1}-\la_n)}-\frac{n-1}{n\la_n^2\la_{n-1}}\\
&\propto\frac{1+a}{a^2}+\frac{1}{1-a}-\frac{n-1}{na^2}\\
&\propto(1-a^2)+a^2-\frac{n-1}{n}(1-a)\\
&=\frac{1}{n}+\frac{n-1}{n}a\lesssim0
\end{aligned}\]
We obtain
\be\label{gen-7.1.11}
\text{Co}(u^2_{n-1n-1n})\lesssim-\bde\frac{\G^{nn}}{\la_{n-1}}.
\ee
Now, for $p<n-1$ we have
\[\begin{aligned}
&2\G^{pp, n-1n-1}+2\frac{\G^{n-1n-1}}{\la_{n-1}-\la_p}\\
&\simeq2\C\lt[\frac{\la_p+\la_{n-1}}{\la_p^2\la^2_{n-1}}-\frac{1}{\la^2_{n-1}(\la_p-\la_{n-1})}\rt]\\
&\propto(\la_p+\la_{n-1})(\la_p-\la_{n-1})-\la_p^2=-\la^2_{n-1}<0.
\end{aligned}\]
This implies when $p<n-1$
\be\label{gen-7.1.12}
\text{Co}(u^2_{n-1n-1p})\lesssim -\lt[\lt(1+\frac{n-2}{n}\rt)+\bde\rt]\frac{\G^{pp}}{\la_{n-1}}.
\ee
Plugging \eqref{gen-7.1.11} and \eqref{gen-7.1.12} into \eqref{gen-7.1.6*} we get
\be\label{gen-7.1.6**}
\begin{aligned}
0&\lesssim\lt[-2\sum_{p<q}\G^{pp, qq}u_{ppn-1}u_{qqn-1}+2\sum_{p\neq n-1}\frac{\G^{pp}u^2_{ppn-1}}{\la_{n-1}-\la_p}\right.\\
&\left.-\lt(1+\frac{n-2}{n}+\frac{\bde}{2}\rt)\frac{\G^{n-1n-1}u^2_{n-1n-1n-1}}{\la_{n-1}}\rt]
-\sum_{i}\frac{\bde}{2}\frac{\G^{ii}u^2_{n-1n-1i}}{\la_{n-1}}
\end{aligned}
\ee
In the following, we will show
\[-2\sum_{p<q}\G^{pp, qq}u_{ppn-1}u_{qqn-1}+2\sum_{p\neq n-1}\frac{\G^{pp}u^2_{ppn-1}}{\la_{n-1}-\la_p}
-\lt(1+\frac{n-2}{n}+\frac{\bde}{2}\rt)\frac{\G^{n-1n-1}u^2_{n-1n-1n-1}}{\la_{n-1}}\lesssim 0.\]
By \eqref{gen-gi} and \eqref{gen-gpq}, this is equivalent to showing
\[\sum_{p<q}\frac{\la_p+\la_q}{\la^2_p\la^2_q}u_{ppn-1}u_{qqn-1}-\sum_{p\neq n-1}\frac{u^2_{ppn-1}}{\la^2_p(\la_{n-1}-\la_p)}
+\lt(\frac{n-1}{n}+\frac{\bde}{4}\rt)\frac{u^2_{n-1n-1n-1}}{\la^3_{n-1}}\gtrsim 0.\]
Denote $\xi_p:=\frac{u_{ppn-1}}{\la_p^2}$ for $1\leq p\leq n,$ we will show
\be\label{gen-7.1.13}\sum_{p<q}(\la_p+\la_q)\xi_p\xi_q-\sum_{p\neq n-1}\frac{\xi^2_p\la^2_{p}}{\la_{n-1}-\la_p}
+\lt(\frac{n-1}{n}+\frac{\bde}{4}\rt)\xi^2_{n-1}\la_{n-1}\gtrsim 0.\ee
In view of \eqref{gen-7.1.0} we have $\sum\limits_i\G^{ii}u_{iin-1}\simeq 0.$ This gives
$\sum\limits_i\xi_i\simeq 0.$ Therefore, \eqref{gen-7.1.13} becomes
\[
\begin{aligned}
&-\sum_{p<n}(\la_p+\la_n)\xi_p\lt(\sum_{\al=1}^{n-1}\xi_\al\rt)+\sum_{1\leq p<q\leq n-1}(\la_p+\la_q)\xi_p\xi_q
+\sum_{p<n-1}\frac{\xi^2_p\la^2_p}{\la_p-\la_{n-1}}\\
&-\frac{\la^2_n}{\la_{n-1}-\la_n}\lt(\sum_{\al=1}^{n-1}\xi_\al\rt)^2+\lt(\frac{n-1}{n}+\frac{\bde}{4}\rt)\xi^2_{n-1}\la_{n-1}
\gtrsim 0
\end{aligned}
\]
In fact, we can show that when $\e_R=\e_R(\Omega, \{0\}, n, \varphi, R)>0$ is chosen to be sufficiently small
\be\label{gen-7.1.14}
\begin{aligned}
&-\sum_{p<n}(\la_p+\la_n)\xi_p\lt(\sum_{\al=1}^{n-1}\xi_\al\rt)+\sum_{1\leq p<q\leq n-1}(\la_p+\la_q)\xi_p\xi_q
+\sum_{p<n-1}\frac{\xi^2_p\la^2_p}{\la_p-\la_{n-1}}\\
&-\frac{\la^2_n}{\la_{n-1}-\la_n}\lt(\sum_{\al=1}^{n-1}\xi_\al\rt)^2+\lt(\frac{n-1}{n}+\frac{\bde}{4}\rt)\xi^2_{n-1}\la_{n-1}
=:\sum_{1\leq p, q\leq n-1}m_{pq}\xi_p\xi_q\geq 0.
\end{aligned}
\ee
We will show the corresponding symmetric matrix  $M:=(m_{pq})_{n-1\times n-1}$ is positive definite. A direct calculation gives, for
$p<q\leq n-1,$
\[
\text{Co}(\xi_q\xi_p)=-(\la_p+\la_n)-(\la_q+\la_n)+(\la_p+\la_q)-\frac{2\la_n^2}{\la_{n-1}-\la_n}
=\frac{2|\la_n|\la_{n-1}}{\la_{n-1}-\la_n}.
\]
Thus, $$m_{pq}=\frac{|\la_n|\la_{n-1}}{\la_{n-1}-\la_n}\,\,\mbox{for}\,\, 1\leq p\neq q\leq n-1.$$
For $p<n-1$ we have
\[\text{Co}(\xi^2_p)=-(\la_p+\la_n)+\frac{\la^2_p}{\la_p-\la_{n-1}}-\frac{\la^2_n}{\la_{n-1}-\la_n}
=\frac{|\la_n|\la_{n-1}}{\la_{n-1}-\la_n}+\frac{\la_p\la_{n-1}}{\la_p-\la_{n-1}}.\]
Thus, $$m_{pp}=\frac{|\la_n|\la_{n-1}}{\la_{n-1}-\la_n}+\frac{\la_p\la_{n-1}}{\la_p-\la_{n-1}}\,\,\mbox{for}\,\, 1\leq p< n-1.$$
For $p=n-1,$
\[\begin{aligned}
\text{Co}(\xi^2_{n-1})&=-(\la_{n-1}+\la_n)-\frac{\la^2_n}{\la_{n-1}-\la_n}+\lt(\frac{n-1}{n}+\frac{\bde}{4}\rt)\la_{n-1}\\
&=\frac{|\la_n|\la_{n-1}}{\la_{n-1}-\la_n}-\frac{1}{n}\la_{n-1}+\frac{\bde}{4}\la_{n-1}.
\end{aligned}\]
Thus, $$m_{n-1n-1}=\frac{|\la_n|\la_{n-1}}{\la_{n-1}-\la_n}-\frac{1}{n}\la_{n-1}+\frac{\bde}{4}\la_{n-1}.$$
For our convenience, denote $B:=\frac{|\la_n|\la_{n-1}}{\la_{n-1}-\la_n}.$ Same as in \cite{LT24}, let $s=\sqrt{B}(1, \cdots, 1)$
be a $1\times (n-1)$ matrix and let $D=\text{diag}(d_1, d_2, \cdots, d_{n-1})$ be an $(n-1)\times(n-1)$ diagonal matrix.
Here,
\[d_i=\frac{\la_i\la_{n-1}}{\la_i-\la_{n-1}}>0\,\,\mbox{for}\,\,i<n-1
\,\,\mbox{and}\,\,d_{n-1}=-\frac{\la_{n-1}}{n}+\frac{\bde\la_{n-1}}{4}<0.\]
Then, $M=D+s^\top s.$ Since the matrix $D$ only contains one negative entry,
by Sylvester's criterion, 
it is easy to derive that in order to prove $M$ is positive definite we only need to show $\det(M)>0.$
A straightforward calculation yields
\be\label{gen-7.1.15}
\begin{aligned}
\det(M)&=\det(D)\det(I_{n-1}+D^{-1}s^\top s)=\det(D)\lt(1+B\sum_{i=1}^{n-1}\frac{1}{d_i}\rt)\\
&=\det(D)\lt\{1+\frac{|\la_n|\la_{n-1}}{\la_{n-1}-\la_n}\lt[\sum_{i=1}^{n-2}\frac{\la_i-\la_{n-1}}{\la_i\la_{n-1}}
-\frac{4n}{(4-\bde n)\la_{n-1}}\rt]\rt\}\\
&=\det(D)\lt\{1+\frac{|\la_n|\la_{n-1}}{\la_{n-1}-\la_n}\lt[\frac{n-2}{\la_{n-1}}-\sum_{i=1}^{n-2}\frac{1}{\la_i}
-\frac{n}{\la_{n-1}}-\frac{n^2\bde}{(4-n\bde)\la_{n-1}}\rt]\rt\}.
\end{aligned}
\ee
Note that
$$\la_n\simeq-\frac{\s_{n-1}(\la|n)}{\s_{n-2}(\la|n)}
=-\frac{\la_{n-1}\s_{n-2}(\la|nn-1)}{\la_{n-1}\s_{n-3}(\la|nn-1)+\s_{n-2}(\la|nn-1)},$$
which gives
$$\frac{1}{\la_n}\simeq-\frac{\s_{n-3}(\la|nn-1)}{\s_{n-2}(\la|nn-1)}-\frac{1}{\la_{n-1}}.$$
Note also that
$$\sum\limits_{i=1}^{n-2}\frac{1}{\la_i}=\frac{\s_{n-3}(\la|nn-1)}{\s_{n-2}(\la|nn-1)}.$$
Therefore, we have
$$-\sum\limits_{i=1}^{n-2}\frac{1}{\la_i}\simeq\frac{1}{\la_n}+\frac{1}{\la_{n-1}}=-\frac{\la_{n-1}+\la_n}{|\la_n|\la_{n-1}}.$$
This implies
\[1+\frac{|\la_n|\la_{n-1}}{\la_{n-1}-\la_n}\lt[\frac{n-2}{\la_{n-1}}-\sum_{i=1}^{n-2}\frac{1}{\la_i}-\frac{n}{\la_{n-1}}\rt]\simeq 0.\]
Plugging it into \eqref{gen-7.1.15} we obtain that when $ \eta_0R^{-8}>\e_R=\e_R(\Omega, \{0\}, n, \varphi, R)>0$ is sufficiently small
\[\det(M)>-\frac{\det(D)n^2\bde|\la_n|}{2(4-n\bde)(\la_{n-1}-\la_n)}>0.\]
Therefore, at $x_0$ we have
\be\label{gen-7.1.16}
0<-\sum_{i}\frac{\bde}{2}\frac{\G^{ii}u^2_{n-1n-1i}}{\la_{n-1}}+C\e
<-\frac{\bde}{2}\frac{\C}{\la^2_1}\sum_i\frac{u^2_{n-1n-1i}}{\la_{n-1}}+C\e.
\ee
for some $C=C(\|u\|_{C^3}, \Omega, \{0\}, R, n)>0.$
Now, in view of \eqref{gen-7.1.1}, \eqref{gen-eta}, Proposition \ref{c0-prop}, Proposition \ref{c1-prop}, Lemma \ref{c2-global-lem}, and Claim A we get
\[
\begin{aligned}
\frac{\bde}{2}\frac{\C}{\la^2_1}\sum_i\frac{u^2_{n-1n-1i}}{\la_{n-1}}
&=\frac{\bde}{2}\frac{\C}{\la^2_1}\lt(\frac{n}{n-2}-\bd\rt)^2\frac{|Du|^2}{u^2}\la_{n-1}\\
&>cR^{-1}\frac{\la_{n-1}^{n+1}}{\la_1^2}R^{-2}>c\frac{R^{-3}\times R^{-\frac{n(n+1)}{n-1}}}{R^{\frac{4(n-2)}{n-1}}},
\end{aligned}
\]
where $c=c(\Omega, \{0\}, n)>0.$ Inserting it into \eqref{gen-7.1.16} we obtain,
when $\e_R=\e_R(\Omega, \{0\}, n, \varphi, R)>0$ is sufficiently small, for any $0<\e<\e_R,$ the right hand side of \eqref{gen-7.1.16} is negative.
This leads to a contradiction. Therefore, we conclude that case 1 does not happen.

\textbf{Case 2.} At $x_0,$ $\la_{n-1}$ has multiplicity $\mu>1.$

Let us define a smooth function $\psi$ in a small neighborhood of $x_0,$ denoted by $U_{x_0},$ as follows:
\[Q(x_0)=m_0\equiv\psi(x)u^{\frac{n}{n-2}-\bd}\,\,\mbox{in}\,\,U_{x_0}.\]
This gives
\[\psi(x)u^{\frac{n}{n-2}-\bd}(x)\leq\la_{n-1}(x)u^{\frac{n}{n-2}-\bd}(x)\,\,\mbox{in}\,\,U_{x_0}.\]
Thus, we have $\psi(x)\leq\la_{n-1}(x)$ in $U_{x_0}$ and $\psi(x_0)=\la_{n-1}(x_0).$
We will need following Claims.\\

\textbf{Claim D.} At $x_0,$ for any $1\leq i\leq n,$ we have $u_{kli}=\psi_i\delta_{kl},\,\,n-\mu\leq k, l\leq n-1.$\\
\textbf{Proof of Claim D:} Without loss of generality, in this proof we may assume $x_0$ to be the origin.
We can choose a local orthonormal frame $\{e_1, \cdots, e_n\}$ at $x_0$ such that $u_{ij}(0):=D^2u(e_i, e_j)(0)=\la_i(0)\delta_{ij}$
with $\la_n(0)<\la_{n-1}(0)=\cdots=\la_{n-\mu}(0)<\la_{n-\mu-1}(0)\leq\cdots\leq\la_1(0).$ For any $x\in U_{x_0},$ we will denote
$\{\tau_{1}(x), \cdots, \tau_n(x)\}$ to be the orthonormal frame at $x$ such that $D^2u(\tau_i(x), \tau_j(x))=\la_i(x)\delta_{ij}$
with $\la_n(x)\leq\la_{n-1}(x)\leq\cdots\leq\la_1(x).$ In particular, $\tau_i(0)=e_i$ for $1\leq i\leq n.$
We note that in view of Claim A, $\la_n$ always has multiplicity 1, thus by the implicit function theorem, we know there exists a small neighborhood of $x_0$ such that
both $\la_n$ and $\tau_n(x)$ are smooth.
We will always assume $U_{x_0}$ to be this nice neighborhood.

Now, let us consider $\td u(x):=u(x)+C_1\lt<x, e_n\rt>^2,$ where $C_1:=3\max\limits_{\bar U_{x_0}}|D^2u|.$
By our choice of $C_1,$ it is clear that  $2C_1+\la_n(0)>\la_1(0).$ Therefore, at $0$ the eigenvalues of
$D^2\td u,$ denoted by $\td\la_i(0),$ satisfy
\[\td\la_{n-1}(0)=\cdots=\td\la_{n-\mu}(0)<\td\la_{n-\mu-1}(0)\leq\cdots\leq\td\la_1(0)<\td\la_n(0).\]
We will show there exists $C_2>0$ so that
\be\label{D-1}
\la_{\min}(D^2\td u)\geq\psi(x)-C_2|x|^2=:\td\psi(x)\,\,\mbox{in}\,\,B_\delta(0)\ssubset U_{x_0}.
\ee
At the origin, we have $\la_{\min}(D^2\td u)=\td\la_{n-1}(0)=\la_{n-1}(0),$ it is clear that \eqref{D-1} holds.
For any $x\in B_\delta(0),$ by definition we get
\[\la_{\min}(D^2\td u)=\min\limits_{\xi\in\mathbb S^n}\lt\{\sum\limits_{i, j}u_{ij}(x)\lt<\xi, e_i\rt>\lt<\xi, e_j\rt>+2C_1\lt<\xi, e_n\rt>^2\rt\},\]
where $u_{ij}(x):=D^2u(e_i, e_j)(x).$
Assume $v\in\mathbb S^{n}$ is the vector such that at $x\in B_{\delta}(0),$
\[\la_{\min}(D^2\td u(x))=\sum\limits_{i, j}u_{ij}(x)\lt<v, e_i\rt>\lt<v, e_j\rt>+2C_1\lt<v, e_n\rt>^2.\]
Since $\sum\limits_{i, j}u_{ij}(x)\lt<v, e_i\rt>\lt<v, e_j\rt>$ is independent of the choice of the frame, we may express $v$ as
$v=\sum\limits_{i=1}^na_i\tau_i(x)$ where $a_i:=\lt<v, \tau_i(x)\rt>.$ Then
\[\begin{aligned}
\la_{\min}(D^2\td u(x))&=\sum\limits_{i, j}D^2u(\tau_i, \tau_j)a_ia_j+2C_1\lt<v, e_n\rt>^2\\
&=\sum\limits_i\la_i(x)a_i^2+2C_1\lt<v, e_n\rt>^2\\
&\geq(1-a_n^2)\la_{n-1}(x)+a^2_n\la_n(x)+2C_1\lt<v, e_n\rt>^2.
\end{aligned}\]
Note that
\[\lt<v, e_n\rt>=\lt<v, \tau_n(x)\rt>+\lt<v, e_n-\tau_n(x)\rt>=a_n+\lt<v, e_n-\tau_n(x)\rt>.\]
Moreover, by the smoothness of $\tau_n(x)$ in $U_{x_0}$ we know there exists some constant $\td C>0$ such that
$$|\tau_n(x)-e_n|\leq\td C|x|\,\,\mbox{in}\,\,B_\delta(0)\ssubset U_{x_0}.$$
Therefore, we have
\[
\begin{aligned}
&-a_n^2\la_{n-1}(x)+a^2_n\la_n(x)+2C_1\lt<v, e_n\rt>^2\\
&\geq a^2_n[-\la_{n-1}(x)+\la_n(x)]+2C_1(a_n^2+2a_n\lt<v, e_n-\tau_n(x)\rt>)\\
&\geq C_1a_n^2-4C_1\td Ca_n|x|\geq -4C_1\td C^2|x|^2.
\end{aligned}
\]
By letting $C_2:=4C_1\td C^2,$ we showed that \eqref{D-1} holds in $B_{\delta}(0).$

It is clear that at the origin we have
\[\la_{\min}(D^2\td u(0))=\td\la_{n-1}(0)=\cdots=\td\la_{n-\mu}(0)=\la_{n-1}(0)=\psi(0)=\td\psi(0)<\td\la_{n-\mu-1}(0)\leq\cdots\]
By Lemma 5 of \cite{BCD17} we obtain, at $0$ for any $1\leq i\leq n,$
\[\td u_{kli}=\td\psi(0)_i\delta_{kl}, n-\mu\leq k, l\leq n-1,\]
which is equivalent to
\[u_{kli}=\psi(0)_i\delta_{kl}, n-\mu\leq k, l\leq n-1.\]
This proves Claim D.\\

\textbf{Claim E.} At $x_0,$ for any $1\leq i\leq n$ we have
$$\psi_{ii}\leq u_{n-1n-1ii}+2\sum\limits_{p\notin Q_\mu}\frac{u^2_{n-1pi}}{\la_{n-1}-\la_p},$$
where $Q_\mu=\{n-\mu, \cdots, n-1\}.$\\
\textbf{Proof of Claim E:} We will still assume $x_0$ to  be the origin and choose a local orthonormal frame $\{e_1, \cdots, e_n\}$ at $x_0$ such that $u_{ij}(0):=D^2u(e_i, e_j)(0)=\la_i(0)\delta_{ij}$ with $\la_n(0)<\la_{n-1}(0)=\cdots=\la_{n-\mu}(0)<\la_{n-\mu-1}(0)\leq\cdots\leq\la_1(0).$
We also denote $\la_i(x):=\la_i(D^2u(x))$ as the $i$-th eigenvalue of $(u_{ij})$ with
$\la_n(x)\leq\la_{n-1}(x)\leq\cdots\leq\la_1(x).$

Now, we choose $\e_{n-\mu}>\e_{n-\mu+1}>\cdots>\e_{n-2}>0,$ where $\e_{n-\mu}>0$ satisfies
$\la_{n-1}(0)+2\e_{n-\mu}<\la_{n-\mu-1}(0).$ Consider
\[\hat u(x):=u(x)+\sum\limits_{\al=n-\mu}^{n-2}\e_\al x_\al^2.\]
We denote $\hat\la_i(x):=\la_i(D^2\hat u(x)).$
It is clear that at $0,$ $$\hat\la_i(0)=\la_i(0)+2\sum\limits_{\al=n-\mu}^{n-2}\e_\al\delta_{\al i},$$
thus $\hat\la_n(0)<\hat\la_{n-1}(0)<\hat\la_{n-2}(0)<\cdots\hat\la_{n-\mu}(0)<\hat\la_{n-\mu-1}(0)\leq\cdots\leq\hat\la_1(0).$
Therefore, in a small neighborhood of $0,$ denoted by $U_{x_0},$ $\hat\la_{n-1}(x)$ is a smooth function. Moreover,
it is easy to see
\[\hat\la_{n-1}(x)\geq\la_{n-1}(x)\,\,\mbox{in}\,\,U_{x_0}.\]
Since $\hat\la_{n-1}(0)=\la_{n-1}(0),$ we get
\[\hat\la_{n-1}(x)\geq\psi(x)\,\,\mbox{in}\,\,U_{x_0}\,\,\mbox{and}
\,\,\hat\la_{n-1}(0)=\psi(0).\]
By the standard second derivative test and Lemma \ref{np-lem1} we obtain, at $0$ for any $1\leq i\leq n$
\[\begin{aligned}
\psi_{ii}&\leq\hat\la_{n-1ii}=\hat u_{n-1n-1ii}+2\sum\limits_{p\neq n-1}\frac{\hat u^2_{pn-1i}}{\hat\la_{n-1}-\hat\la_p}\\
&=\hat u_{n-1n-1ii}+2\sum\limits_{p< n-1}\frac{\hat u^2_{pn-1i}}{\hat\la_{n-1}-\hat\la_p}
+2\frac{\hat u^2_{nn-1i}}{\hat\la_{n-1}-\hat\la_n}\\
&\leq\hat u_{n-1n-1ii}+2\sum\limits_{p< n-\mu}\frac{\hat u^2_{pn-1i}}{\hat\la_{n-1}-\hat\la_p}
+2\frac{\hat u^2_{nn-1i}}{\hat\la_{n-1}-\hat\la_n}\\
&=u_{n-1n-1ii}+2\sum\limits_{p\notin Q_\mu}\frac{u^2_{pn-1i}}{\la_{n-1}-\la_p}.
\end{aligned}\]
This confirms Claim E.

Now, let us focus on discussing what would happen in Case 2. Let $Q(x)=\psi u^{\frac{n}{n-2}-\bd},$
then $Q(x)\equiv m_0$ in $U_{x_0}.$ Therefore, we have for any $x\in U_{x_0}$ and $1\leq i\leq n$
\be\label{gen-7.2.1}
0=\frac{\psi_i}{\psi}+\lt(\frac{n}{n-2}-\bd\rt)\frac{u_i}{u}
\ee
and
\be\label{gen-7.2.2}
0=\frac{\psi_{ii}}{\psi}-\lt[\lt(1+\frac{n-2}{n}+\bde\rt)\rt]\frac{\psi^2_i}{\psi^2}+\lt(\frac{n}{n-2}-\bd\rt)\frac{u_{ii}}{u}.
\ee
In particular, by Claim D and E we obtain at $x_0$
\be\label{gen-7.2.3}
\begin{aligned}
0&=\sum_i\frac{F^{ii}\psi_{ii}}{\psi}-\lt[\lt(1+\frac{n-2}{n}\rt)+\bde\rt]\sum_i\frac{F^{ii}\psi_i^2}{\psi^2}
+\lt(\frac{n}{n-2}-\bd\rt)\sum_i\frac{F^{ii}u_{ii}}{u}\\
&\lesssim\frac{1}{\la_{n-1}}\sum_iF^{ii}\lt(u_{n-1n-1ii}+2\sum_{p\notin Q_\mu}\frac{u^2_{n-1pi}}{\la_{n-1}-\la_p}\rt)
-\lt[\lt(1+\frac{n-2}{n}\rt)+\bde\rt]\sum_i\frac{F^{ii}u^2_{n-1n-1i}}{\la^2_{n-1}}\\
&\simeq\frac{1}{\la_{n-1}\s_{n-2}}\lt\{-2\sum_{p<q}\G^{pp, qq}u_{ppn-1}u_{qqn-1}+2\sum_{p<q}\G^{pp, qq}u^2_{pqn-1}\right.\\
&\left.+2\sum_i\sum_{p\notin Q_\mu}\frac{\G^{ii}u^2_{n-1pi}}{\la_{n-1}-\la_p}-\lt[\lt(1+\frac{n-2}{n}\rt)+\bde\rt]\sum_i\frac{\G^{ii}u^2_{n-1n-1i}}{\la_{n-1}}\rt\}.
\end{aligned}
\ee
In view of Claim D we have
\[u_{n-1n-1n-1}=\psi_{n-1}=u_{n-2n-2n-1}=u_{n-2n-1n-2}=\psi_{n-2}\delta_{n-1n-2}=0,\] which implies
\be\label{gen-7.2.4}
u_{ppn-1}=\psi_{n-1}=0\,\,\mbox{for all}\,\, p\in Q_\mu.
\ee
Moreover, we also have
\be\label{gen-7.2.5}
u_{pqi}=0\,\,\mbox{for all}\,\,p\neq q\in Q_\mu\,\,\mbox{and}\,\,1\leq i\leq n.
\ee
By \eqref{gen-7.2.4} we get
\be\label{gen-7.2.6}\begin{aligned}
&2\sum_{p<q}\G^{pp, qq}u_{ppn-1}u_{qqn-1}\\
=&2\sum_{p<n-\mu}\G^{pp, nn}u_{ppn-1}u_{nnn-1}+2\sum_{p<q<n-\mu}\G^{pp, qq}u_{ppn-1}u_{qqn-1}.\\
\end{aligned}\ee
By \eqref{gen-7.2.4} and \eqref{gen-7.2.5} we get
\be\label{gen-7.2.7}\begin{aligned}
&2\sum_{p<q}\G^{pp, qq}u^2_{pqn-1}\\
=&2\sum_{p<n}\G^{pp, nn}u^2_{pnn-1}+2\sum_{p<n-\mu}\G^{pp, n-1n-1}u^2_{pn-1n-1}+2\sum_{p<q<n-\mu}\G^{pp, qq}u^2_{pqn-1}\\
=&2\G^{n-1n-1, nn}u^2_{n-1n-1n}+2\sum_{p<n-\mu}\G^{pp, nn}u^2_{pnn-1}\\
+&2\sum_{p<n-\mu}\G^{pp, n-1n-1}u^2_{n-1n-1p}+2\sum_{p<q<n-\mu}\G^{pp, qq}u^2_{pqn-1}
\end{aligned}\ee

Inserting \eqref{gen-7.2.6} and \eqref{gen-7.2.7} into \eqref{gen-7.2.3} we have
\be\label{gen-7.2.8}
\begin{aligned}
0&\lesssim-2\sum_{p<n-\mu}\G^{pp, nn}u_{ppn-1}u_{nnn-1}-2\sum_{p<q<n-\mu}\G^{pp, qq}u_{ppn-1}u_{qqn-1}\\
&+2\G^{n-1n-1, nn}u^2_{n-1n-1n}+2\sum_{p<n-\mu}\G^{pp, nn}u^2_{pnn-1}\\
&+2\sum_{p<n-\mu}\G^{pp, n-1n-1}u^2_{n-1n-1p}+2\sum_{p<q<n-\mu}\G^{pp, qq}u^2_{pqn-1}\\
&+2\sum_{\substack{i=n\\i=n-1\\i<n-\mu}}\sum_{p\notin Q_\mu}\frac{\G^{ii}u^2_{n-1pi}}{\la_{n-1}-\la_p}
-\lt[\lt(1+\frac{n-2}{n}\rt)+\bde\rt]\sum_{i\notin Q_\mu}\frac{\G^{ii}u^2_{n-1n-1i}}{\la_{n-1}}.
\end{aligned}
\ee
Following the same argument as in Case 1, we can see that \eqref{gen-7.2.8} implies
\be\label{gen-7.2.8*}
\begin{aligned}
0&\lesssim\lt[-2\sum_{p<n-\mu}\G^{pp, nn}u_{ppn-1}u_{nnn-1}-2\sum_{p<q<n-\mu}\G^{pp, qq}u_{ppn-1}u_{qqn-1}\right.\\
&\left.+2\frac{\G^{nn}u^2_{nnn-1}}{\la_{n-1}-\la_n}-2\sum_{p<n-\mu}\frac{\G^{pp}u^2_{ppn-1}}{\la_p-\la_{n-1}}\rt]-\sum_{i\notin Q_\mu}\frac{\bde}{2}\frac{\G^{ii}u^2_{n-1n-1i}}{\la_{n-1}}.
\end{aligned}
\ee
We want to point out that when $\mu=n-1$ we would have derived
\[0\lesssim 2\frac{\G^{nn}u^2_{nnn-1}}{\la_{n-1}-\la_n}-\sum_{i\notin Q_\mu}\frac{\bde}{2}\frac{\G^{ii}u^2_{n-1n-1i}}{\la_{n-1}}
\simeq-\frac{\bde}{2}\frac{\G^{nn}u^2_{n-1n-1n}}{\la_{n-1}}\]
directly. Here, $\frac{\G^{nn}u^2_{nnn-1}}{\la_{n-1}-\la_n}\simeq 0$ comes from $(\G)_{n-1}\simeq0$ and $u_{ppn-1}=0$ for all $1\leq p\leq n-1.$
Therefore, in the following, we will always assume $\mu<n-1.$
We will show
\be\label{gen-7.2.9}\begin{aligned}
&-2\sum_{p<n-\mu}\G^{pp, nn}u_{ppn-1}u_{nnn-1}-2\sum_{p<q<n-\mu}\G^{pp, qq}u_{ppn-1}u_{qqn-1}\\
&+2\frac{\G^{nn}u^2_{nnn-1}}{\la_{n-1}-\la_n}-2\sum_{p<n-\mu}\frac{\G^{pp}u^2_{ppn-1}}{\la_p-\la_{n-1}}\lesssim 0.
\end{aligned}\ee
Same as in Case 1 we denote $\xi_i:=\frac{u_{iin-1}}{\la^2_i},$ then by \eqref{gen-7.2.4} we know $\xi_p=0$ for all $p\in Q_\mu.$
Thus, \eqref{gen-7.1.0} implies $\sum\limits_{i\notin Q_\mu}\xi_i\simeq 0$ and to prove \eqref{gen-7.2.9} we only need to show
\[\begin{aligned}
&-\sum_{p<n-\mu}(\la_n+\la_p)\lt(\sum_{\al<n-\mu}\xi_\al\rt)\xi_p+\sum_{p<q<n-\mu}(\la_p+\la_q)\xi_p\xi_q\\
&-\frac{\la_n^2}{\la_{n-1}-\la_n}(\sum_{\al<n-\mu}\xi_\al)^2+\sum_{p<n-\mu}\frac{\la_p^2}{\la_p-\la_{n-1}}\xi_p^2\\
&=:\sum\limits_{1\leq p, q\leq n-\mu-1}m_{pq}\xi_q\xi_p\geq 0.
\end{aligned}\]
By calculations in Case 1 we know $m_{pp}=B+\frac{\la_p\la_{n-1}}{\la_p-\la_{n-1}}$ for all $1\leq p\leq n-\mu-1$ and $m_{pq}=B$ for $1\leq p\neq q\leq n-\mu-1,$
where $B=\frac{|\la_n|\la_{n-1}}{\la_{n-1}-\la_n}.$
Therefore, we obtain that
\[M:=(m_{pq})_{(n-\mu-1)\times(n-\mu-1)}=D+s^\top s\]
is positive definite.
Here $s=\sqrt{B}(1, \cdots, 1)$ is a $1\times(n-\mu-1)$ matrix and $D=\text{diag}(d_1, \cdots, d_{n-\mu-1})$ with $d_i=\frac{\la_i\la_{n-1}}{\la_i-\la_{n-1}}>0.$

We conclude that at $x_0$
\be\label{gen-7.2.10}
0<-\sum_{i\notin Q_\mu}\frac{\bde}{2}\frac{\G^{ii}u^2_{n-1n-1i}}{\la_{n-1}}+C\e
<-\frac{\bde}{2}\frac{\C}{\la^2_1}\sum_{i\notin Q_\mu}\frac{u^2_{n-1n-1i}}{\la_{n-1}}+C\e,
\ee
for some $C=C(\|u\|_{C^3}, \Omega, \{0\}, R, n)>0.$
Following the same argument as in Case 1, we obtain
when $\e_R=\e_R(\Omega, \{0\}, n, \varphi, R)>0$ is sufficiently small, for any $0<\e<\e_R,$ the right hand side of \eqref{gen-7.2.10} is negative.
This leads to a contradiction. Therefore, Case 2 does not happen.

This completes the proof of Lemma \ref{step3-lem}.
\end{proof}

Putting Lemma \ref{step1-lem}, Lemma \ref{step2-lem}, and Lemma \ref{step3-lem} together we get
\begin{lemma}
 \label{c2-lowerdecay-lem}
Given $\eta_0=\eta_0(\Omega)>0, \ev=\ev(\Omega, \{0\}, n)>0, \eta_0R^{-8}>\e_R=\e_R(\Omega, \{0\}, n, \varphi, R)>0$ sufficiently small, $R_0=R_0(\Omega)>0$ sufficiently large, for any $0<\e<\e_R,$ $R>R_0,$ and $\varphi\in C^{\infty}(\bar\Omega)$ that satisfies $\varphi\leq0$ and $\|\varphi\|_{C^2}\leq \ev$ let $\ure$ be the admissible solution of \eqref{eq-appr}. Denote $\la(D^2\ure)=(\la_1, \cdots, \la_n)$ with $\la_1\geq\la_2\geq\cdots\geq\la_n,$ then we have
\be\label{la2-lowb}
\la_{n-1}>\bb_2r^{-\frac{n}{n-1}}.
\ee
Let $\hure$ be the admissible solution of \eqref{eq-appr-s} and denote $\la(D^2\hure)=(\hat\la_1, \cdots, \hat\la_n)$ with $\hat\la_1\geq\hat\la_2\geq\cdots\geq\hat\la_n,$
then we have
\be\label{la2-lowb}
\hat\la_2>\bb_2r^{-\frac{n}{n-1}}.
\ee
Here, $\bb_2=\bb_2(\Omega, \{0\}, n)>0$ is a positive constant that is independent of $R$ and $\e.$
Consequently, on $\ol{\Omega\setminus B_{1/R}},$ $\ure$ satisfies
 \be\label{c2-lowd}
\lt|D^2 \ure\rt|>\bb_2r^{-\frac{n}{n-1}},
 \ee
and on $\ol{\Omega^R\setminus B_1},$ $\hure$ satisfies
 \be\label{c2-lowd-s}
 \lt|D^2\hure\rt|>\bb_2r^{-\frac{n}{n-1}}.
 \ee
\end{lemma}

\section{Asymptotic behavior near the singularity}
 \label{sec-AB}
 Recall that our goal is to solve \eqref{eq-main1}. To achieve this, we study the approximate problem \eqref{eq-appr}
 and, for our convenience, its rescaling \eqref{eq-appr-s}. However, so far we only obtained (in Lemma \ref{c2-global-lem})
 \[|D^2\ure|\leq CR^{3-\frac{1}{n-1}}\,\,\mbox{on}\,\,\ol{\Omega\setminus B_{1/R}},\]
 for some $C>0$ that is independent of $R$ and $\e.$
 Clearly, this estimate is not good enough for us to say that there exists a subsequence $\{\ure\}$ converges smoothly to
 a solution of \eqref{eq-main1}. We need a better estimate. In this section, we shall prove
 \begin{lemma}
 \label{c2-upperdecay-lem}
 Given $\eta_0=\eta_0(\Omega)>0, \ev=\ev(\Omega, \{0\}, n)>0, \eta_0R^{-8}>\e_R=\e_R(\Omega, \{0\}, n, \varphi, R)>0$ sufficiently small, $R_0=R_0(\Omega)>0$ sufficiently large, for any $0<\e<\e_R,$ $R>R_0,$ and $\varphi\in C^{\infty}(\bar\Omega)$ that satisfies $\varphi\leq0$ and $\|\varphi\|_{C^2}\leq \ev$ let $\ure$ be the admissible solution of \eqref{eq-appr}.
 Then on $\ol{\Omega\setminus B_{1/R}},$ $\ure$ satisfies
 \be\label{c2-upd}
\lt|D^2 \ure\rt|<\ba_2r^{-\frac{n}{n-1}}.
 \ee
 Let $\hure$ be the admissible solution of \eqref{eq-appr-s}, then on $\ol{\Omega^R\setminus B_1},$ $\hure$ satisfies
 \be\label{c2-upd-s}
 \lt|D^2\hure\rt|<\ba_2r^{-\frac{n}{n-1}}.
 \ee
Here, $\ba_2=\ba_2(\Omega, \{0\}, n)>0$ is a positive constant that is independent of $R$ and $\e.$
\end{lemma}
\begin{proof} Having established Lemma \ref{c2-global-lem} and Lemma \ref{c2-lowerdecay-lem}, we now can follow the approach of \cite{CW01} to prove inequality \eqref{c2-upd-s}. Inequality \eqref{c2-upd} then follows by rescaling.

In this proof, for our convenience, we shall denote the solution of \eqref{eq-appr-s} by $u$ instead of $\hure.$
Consider the domain $\Omega^R\setminus B_{A_0},$ where $A_0=A_0(\Omega, \{0\}, n)>0$ is chosen such that
$\bb A_0^{\frac{n-2}{n-1}}>3\ba.$ In view of Proposition \ref{c0-prop} we have,
\[\min\limits_{x\in\ol{\Omega^R\setminus B_{A_0}}}u(x)\geq 3\ba.\]

Now, for any fixed $x_0\in\Omega^R\setminus B_{A_0},$ we denote $r_0:=|x_0|$ and
$\Omega_u:=\{x\in\Omega^R\setminus B_1: 2u(x)>u(x_0)\}.$ By our choice of $A_0,$
we know that $\Omega_u\subset\Omega^R\setminus\bar{B}_1.$ Moreover, applying Proposition \ref{c0-prop} again we get,
for any $x\in\Omega_u,$ $\ba|x|^{\frac{n-2}{n-1}}>\frac{\bb}{2} r_0^{\frac{n-2}{n-1}},$ which implies $|x|>\lt(\frac{\bb}{2\ba}\rt)^{\frac{n-1}{n-2}}r_0.$ By virtue of
Proposition \ref{c1-prop} we can see that for any $x\in\Omega_u$
\[|Du(x)|<\ba_1|x|^{-\frac{1}{n-1}}<Cr_0^{-\frac{1}{n-1}},\]
where $C=C(\Omega, \{0\}, n)>0$ is a positive constant that is independent of $x_0,$ $R,$ and $\e.$

Denote $V:=|Du|^2$ and $M:=2n^3\max\limits_{x\in\bar\Omega_u}V,$ by Proposition \ref{c1-prop} it is easy to check that $M\sim r_0^{-\frac{2}{n-1}}.$
Let $\psi(V):=(M-V)^{-\frac{n}{2}},$ then $\psi(V)\sim r_0^{\frac{n}{n-1}}.$ We also denote $\zeta:=\frac{2u}{u_0}-1,$ where $u_0:=u(x_0).$
For any $x\in\Omega_u,$ suppose $D^2u$ is diagonalized at $x,$ that is, $u_{ij}(x)=u_{ii}(x)\delta_{ij}=:\la_i(x)\delta_{ij}.$ We shall always assume $\la_1(x)\geq\la_2(x)\geq\cdots\geq\la_n(x).$

Let $G(x)=\zeta^{\frac{n-1}{n-2}}\la_1\psi(V),$ then we have
\[G=0\,\,\mbox{on $\p\Omega_u\setminus\p\Omega^R.$}\]
Moreover, By Lemma \ref{c2-outside-tan-lem}, Lemma \ref{im-c2-outside-mix-lem}, and Lemma \ref{im-c2-outside-nu-lem} we get
On $\p\Omega_u\cap\p\Omega^R=\p\Omega^R,$
\[
\begin{aligned}
G&\leq\lt(\frac{2R^{\frac{n-2}{n-1}}}{u_0}-1\rt)^{\frac{n-1}{n-2}}CR^{-\frac{n}{n-1}}r_0^{\frac{n}{n-1}}\\
&<C\frac{R}{r_0}R^{-\frac{n}{n-1}}r_0^{\frac{n}{n-1}}=C\lt(\frac{r_0}{R}\rt)^{\frac{1}{n-1}}\leq C
\end{aligned}
\]
for some $C=C(\Omega, \{0\}, n)>0$ that is independent of $x_0,$ $R,$ and $\e.$
Our goal is to show
\[\max\limits_{\bar\Omega_u}G\leq C\]
for some $C=C(\Omega, \{0\}, n)>0$ that is independent of $x_0,$ $R,$ and $\e.$

Without loss of generality, we may assume $G_{\max}:=\max\limits_{\bar\Omega_u}G$ is achieved at an interior point
$\hx\in \Omega_u.$ Assuming that $D^2u(\hat x)$ is diagonal at $\hx$
and $\la_1(\hx)=u_{11}(\hx)$ has multiplicity $\mu,$ that is,
\[\la_1=\cdots=\la_\mu>\la_{\mu+1}\geq\cdots\geq\la_n,\]
then at $\hat x,$ we have
\be\label{AB-case1}
\delta_{kl}\la_{1i}=u_{kli}\,\,\mbox{for}\,\,1\leq k, l\leq\mu,
\ee
and
\be\label{AB-case1*}
\la_{1ii}\geq u_{11ii}+2\sum_{p>\mu}\frac{u^2_{1pi}}{\la_1-\la_p}.
\ee
Note that when $\mu=1,$ the equality in \eqref{AB-case1*} holds and when $\mu>1$ \eqref{AB-case1}, \eqref{AB-case1*}  hold in the viscosity sense (see \cite{BCD17}).

Differentiating $G$ at $\hx$ we obtain
\be\label{AB-critic}
0=\frac{G_i}{G}=\frac{n-1}{n-2}\frac{\zeta_i}{\zeta}+\frac{\la_{1i}}{\la_1}+\frac{\psi_i}{\psi}
\ee
and
\be\label{AB-test}
0\geq\frac{n-1}{n-2}F^{ii}\lt(\frac{\zeta_{ii}}{\zeta}-\frac{\zeta_i^2}{\zeta^2}\rt)+F^{ii}\lt(\frac{\la_{1ii}}{\la_1}-\frac{\la_{1i}^2}{\la_1^2}\rt)
+F^{ii}\lt(\frac{\psi_{ii}}{\psi}-\frac{\psi_i^2}{\psi^2}\rt).
\ee
Plugging \eqref{AB-case1} and \eqref{AB-case1*} into \eqref{AB-test} we get
\be\label{AB-test1}
\begin{aligned}
0&\geq \frac{n-1}{n-2}F^{ii}\lt(\frac{\zeta_{ii}}{\zeta}-\frac{\zeta_i^2}{\zeta^2}\rt)+F^{ii}\lt(\frac{\psi_{ii}}{\psi}-\frac{\psi_i^2}{\psi^2}\rt)\\
&+F^{ii}\lt[\frac{u_{11ii}}{\la_1}+2\sum_{p>\mu}\frac{u^2_{1pi}}{\la_1(\la_1-\la_p)}-\frac{u^2_{11i}}{\la_1^2}\rt]\\
&=\frac{n-1}{n-2}F^{ii}\lt(\frac{\zeta_{ii}}{\zeta}-\frac{\zeta_i^2}{\zeta^2}\rt)+F^{ii}\lt(\frac{\psi_{ii}}{\psi}-\frac{\psi_i^2}{\psi^2}\rt)\\
&+\frac{1}{\la_1}\bigg[(f_\e)_{11}-F^{pq, rs}u_{pq1}u_{rs1}+2F^{ii}\sum_{p>\mu}\frac{u^2_{1pi}}{\la_1-\la_p}-F^{ii}\frac{u^2_{11i}}{\la_1}\bigg].
\end{aligned}
\ee

\textbf{Case 1.} $u_{n-1n-1}>\frac{1}{5n}u_{11}$ at $\hx$. In view of
\eqref{AB-critic} we have
\[\lt(\frac{u_{11i}}{u_{11}}\rt)^2\leq 2\frac{\psi^2_i}{\psi^2}+2\lt(\frac{n-1}{n-2}\rt)^2\lt(\frac{\zeta_i^2}{\zeta^2}\rt).\]
By the concavity of $F,$ \eqref{AB-test1} implies
\be\label{v1change1}
\begin{aligned}
0&\geq\frac{n-1}{n-2}F^{ii}\lt\{\frac{\zeta_{ii}}{\zeta}-\lt[1+2\lt(\frac{n-1}{n-2}\rt)\rt]\frac{\zeta^2_{i}}{\zeta^2}\rt\}\\
&+F^{ii}\lt[\frac{\psi_{ii}}{\psi}-3\frac{\psi^2_i}{\psi^2}\rt]+\frac{(f_\e)_{11}}{\la_1}.
\end{aligned}
\ee
Since
\[V_i=2\sum_lu_lu_{li}=2u_iu_{ii}\,\,\mbox{and}\,\, V_{ii}=2\sum\limits_l(u^2_{li}+u_{l}u_{lii})=2u^2_{ii}+\sum_l2u_lu_{lii},\]
a straightforward calculation yields for any $\gamma\geq 2$ we have
\be\label{AB-3}
\begin{aligned}
&F^{ii}\lt[\frac{\psi_{ii}}{\psi}-\gamma\frac{\psi^2_i}{\psi^2}\rt]\\
=&\lt[\frac{\psi''}{\psi}-\gamma\lt(\frac{\psi'}{\psi}\rt)^2\rt]F^{ii}V^2_i+\frac{\psi'}{\psi}F^{ii}V_{ii}\\
=&\lt[\frac{n}{2}\lt(\frac{n}{2}+1\rt)-\gamma\frac{n^2}{4}\rt](M-V)^{-2}F^{ii}V^2_i+n(M-V)^{-1}F^{ii}(u^2_{ii}+u_lu_{lii})\\
\geq&\lt[\frac{n}{2}\lt(\frac{n}{2}+1\rt)-\gamma\frac{n^2}{4}\rt]\frac{4V}{(M-V)^2}F^{ii}u^2_{ii}+n(M-V)^{-1}F^{ii}u^2_{ii}+n(M-V)^{-1}Du\cdot Df_\e.\\
\end{aligned}
\ee
In particular, when $\gamma=3$ we get
\[F^{ii}\lt[\frac{\psi_{ii}}{\psi}-3\frac{\psi^2_i}{\psi^2}\rt]>\frac{(n-2)M}{(M-V)^2}F^{ii}u^2_{ii}+\frac{n}{M-V}Du\cdot Df_\e.\]
Plugging it into \eqref{v1change1} gives
\be\label{v1change2}
0\geq C_1\frac{f_\e}{\zeta u_0}-C_2\frac{|Du|^2}{u^2_{0}\zeta^2}\sum F^{ii}+\frac{n-2}{M}F^{n-1n-1}\frac{\la_1^2}{25n^2}-\frac{8n}{M}|Du|f_\e+\frac{(f_\e)_{11}}{\la_1},
\ee
where $C_1=C_1(n), C_2=C_2(n)>0.$
When $\e_R=\e_R(\Omega, \{0\}, n, \varphi, R)>0$ is sufficiently small, for any $0<\e<\e_R,$ using (3.2) of \cite{CW01}
we obtain \[F^{n-1n-1}>\theta_1=\theta_1(n)>0.\]
On the other hand, recall \eqref{v1change*} and \eqref{gen-step3-3} we get
\[\frac{2}{n-1}\leq \sum F^{ii}<2.\]
Now, denote $X:=u_{11}(\hx)\zeta^{\frac{n-1}{n-2}}(\hx),$ apply Proposition \ref{c0-prop} and Proposition \ref{c1-prop}, \eqref{v1change2} yields
\[
X^2r^{\frac{2}{n-1}}_0\leq \frac{C_3}{r_0^2},
\]
where $C_3=C_3(\Omega, \{0\}, n)>0$.This shows that, in this case, there exists a positive constant $C=C(\Omega, \{0\}, n)>0$ such that
\[\max\limits_{\bar\Omega_u}G\leq C.\]

\textbf{Case 2.} $u_{n-1n-1}\leq\frac{1}{5n} u_{11}$ at $\hx.$
Note that by \eqref{AB-critic} we have
\be\label{AB-1}
\begin{aligned}
\frac{u^2_{111}}{\la_1^2}&=\lt(\frac{n-1}{n-2}\frac{\zeta_1}{\zeta}+\frac{\psi_1}{\psi}\rt)^2\\
&\leq\frac{4}{3}\frac{\psi_1^2}{\psi^2}+4\lt(\frac{n-1}{n-2}\rt)^2\frac{\zeta_1^2}{\zeta^2},
\end{aligned}
\ee
and for $i\geq 2$ we have
\be\label{AB-2}
\begin{aligned}
\frac{n-1}{n-2}\frac{\zeta^2_i}{\zeta^2}&=\frac{n-2}{n-1}\lt(\frac{u_{11i}}{\la_1}+\frac{\psi_i}{\psi}\rt)^2\\
&\leq\lt[1-\frac{1}{2(n-1)}\rt]\frac{u^2_{11i}}{\la_1^2}+(2n-1)\frac{\psi^2_i}{\psi^2}.
\end{aligned}
\ee
Inserting \eqref{AB-1} and \eqref{AB-2} into \eqref{AB-test1} gives
\be\label{AB-test2}
\begin{aligned}
0&\geq\frac{n-1}{n-2}F^{ii}\frac{\zeta_{ii}}{\zeta}-5\lt(\frac{n-1}{n-2}\rt)^2F^{11}\frac{\zeta_1^2}{\zeta^2}+F^{ii}\lt(\frac{\psi_{ii}}{\psi}-2n\frac{\psi_i^2}{\psi^2}\rt)\\
&+\frac{1}{\la_1}(f_\e)_{11}-\frac{1}{\la_1}F^{pq, rs}u_{pq1}u_{rs1}+2F^{ii}\sum_{p>\mu}\frac{u^2_{1pi}}{\la_1(\la_1-\la_p)}\\
&-\lt[2-\frac{1}{2(n-1)}\rt]\sum_{i>\mu}F^{ii}\frac{u^2_{11i}}{\la_1^2}.
\end{aligned}
\ee
Here, we have used  when $\mu>1$ by \eqref{AB-case1} we have $u_{11i}=0$ for $2\leq i\leq \mu.$
Since we have proved Lemma \ref{c2-lowerdecay-lem}, below, we employ notations introduced in Section \ref{sec-step3}.
In view of the concavity of $F$ we have
\[\begin{aligned}
&-F^{pq, rs}u_{pq1}u_{rs1}+2F^{ii}\sum\limits_{p>\mu}\frac{u^2_{1pi}}{\la_1-\la_p}-\lt[2-\frac{1}{2(n-1)}\rt]\sum\limits_{p>\mu}F^{pp}\frac{u^2_{11p}}{\la_1}\\
&\geq 2\sum\limits_{p>\mu}\frac{F^{pp}-F^{11}}{\la_1-\la_p}u^2_{11p}+2F^{11}\sum\limits_{p>\mu}\frac{u^2_{11p}}{\la_1-\la_p}
-\lt[2-\frac{1}{2(n-1)}\rt]\sum\limits_{p>\mu}F^{pp}\frac{u^2_{11p}}{\la_1}\\
&\simeq \frac{2\C}{\s_{n-2}}\sum\limits_{p>\mu}\lt[\frac{\la_p+\la_1}{\la_p^2\la_1^2}+\frac{1}{\la_1^2(\la_1-\la_p)}
-\frac{1}{\la_1\la_p^2}+\frac{1}{4(n-1)}\frac{1}{\la_1\la_p^2}\rt]u^2_{11p}\\
&=\frac{2\C}{\s_{n-2}}\sum\limits_{p>\mu}\frac{1}{\la_1\la_p}\lt[\frac{1}{\la_1-\la_p}+\frac{1}{4(n-1)\la_p}\rt]u^2_{11p}
\end{aligned}\]
It is clear that when $p\leq n-1$
\[\frac{1}{\la_1\la_p}\lt[\frac{1}{\la_1-\la_p}+\frac{1}{4(n-1)\la_p}\rt]>0.\]
When $p=n$ we have
\[\begin{aligned}
&\frac{1}{\la_1\la_n}\lt[\frac{1}{\la_1-\la_n}+\frac{1}{4(n-1)\la_n}\rt]\\
=&\frac{1}{\la_1|\la_n|}\lt[\frac{1}{4(n-1)|\la_n|}-\frac{1}{\la_1+|\la_n|}\rt].
\end{aligned}\]
Since $\la_{n-1}>|\la_n|$ and in this case we have $5n\la_{n-1}\leq\la_1,$ it is easy to see that
\[\lt[\frac{1}{4(n-1)|\la_n|}-\frac{1}{\la_1+|\la_n|}\rt]>0.\]
Moreover, by letting $\gamma=2n$ in \eqref{AB-3}, we obtain
\[F^{ii}\lt[\frac{\psi_{ii}}{\psi}-2n\frac{\psi^2_i}{\psi^2}\rt]\gtrsim\frac{(n-2)M}{(M-V)^2}F^{ii}u^2_{ii}.\]
Therefore, \eqref{AB-test2} implies that
\be\label{AB-n-1}
\begin{aligned}
0&\geq\frac{n-1}{n-2}F^{ii}\frac{\xi_{ii}}{\xi}-5\lt(\frac{n-1}{n-2}\rt)^2F^{11}\frac{\xi_1^2}{\xi^2}
+F^{ii}\lt(\frac{\psi_{ii}}{\psi}-2n\frac{\psi^2_i}{\psi^2}\rt)-C_4\e\\
&\geq-5\lt(\frac{n-1}{n-2}\rt)^2F^{11}\frac{\xi_1^2}{\xi^2}
+\frac{(n-2)M}{(M-V)^2}F^{ii}u^2_{ii}-C_4\e\\
&\geq-5\lt(\frac{n-1}{n-2}\rt)^2F^{11}\frac{\xi_1^2}{\xi^2}
+\frac{(n-2)M}{(M-V)^2}F^{11}\la_1^2-C_4\e,
\end{aligned}
\ee
where $C_4=C_4(\|u\|_{C^3}, R, \Omega, \{0\},n)>0.$
When $0<\e<\e_R$ with $\e_R=\e_R(\Omega, \{0\}, n, \varphi, R)>0$ being sufficiently small, by virtue of \eqref{gen-step3-2} we obtain
\be\label{AB-n-3}
0\geq-C_5\frac{|Du|^2}{u^2_0\zeta^2}+C_6r_0^{\frac{2}{n-1}}\la^2_1,
\ee
where $C_5=C_5(n)>0,$ $C_6=C_6(\Omega, \{0\}, n)>0$ are some positive constants that are independent of $x_0, R$ and $\e.$

Same as in Case 1, we denote $X:=u_{11}(\hx)\zeta^{\frac{n-1}{n-2}}(\hx).$ Applying Proposition \ref{c0-prop} and Proposition \ref{c1-prop} we derive at $\hx$
\[
X^2r^{\frac{2}{n-1}}_0\leq \frac{C_7}{r_0^2},
\]
where we $C_7=C_7(\Omega, \{0\}, n)>0.$

Combining Case 1 and Case 2, we conclude that there exists a positive constant $C=C(\Omega, \{0\}, n)>0$ such that
\[\max\limits_{\bar\Omega_u}G\leq C.\]
In particular,  we obtain for any $x_0\in\Omega^R\setminus B_{A_0},$ there exists $C=C(\Omega, \{0\}, n)>0,$
such that $G(x_0)\leq C,$ which yields $\la_1(x_0)\leq Cr_0^{-\frac{n}{n-1}}$ for some $C=C(\Omega, \{0\}, n)>0.$

When $x_0\in B_{A_0}\setminus B_1,$ we consider the test function
\[\Laplace\hure+C(n)\hure.\]
Following the same argument in Lemma \ref{c2-global-lem}, we obtain
\[\max\limits_{\overline{B_{A_0}\setminus B_1}}\lt(\Laplace\hure+C(n)\hure\rt)
=\max\limits_{\p\lt(B_{A_0}\setminus B_1\rt)}\lt(\Laplace\hure+C(n)\hure\rt).\]

Since $A_0=A_0(\Omega, \{0\}, n)>0$ is fixed, it is easy to see that there exists $C=C(\Omega, \{0\}, n)>0$
such that for any $x_0\in \ol{ B_{A_0}\setminus B_1} $ we have, $\la_1(x_0)\leq CA_0^{-\frac{n}{n-1}}\leq Cr_0^{-\frac{n}{n-1}}.$ This completes the proof of \eqref{c2-upd-s} and \eqref{c2-upd} follows by rescaling.
\end{proof}

\section{Smooth solution to \eqref{eq-main1}}
\label{sec-smooth}
Combining the results obtained in Section \ref{solv-appr}, Section \ref{improve}, Section \ref{sec-deformation}, Section \ref{sec-step3}, and Section \ref{sec-AB}, we arrive at the following crucial theorem of this paper.

\begin{theorem}
\label{theorem-key}
 Given $\eta_0=\eta_0(\Omega)>0, \ev=\ev(\Omega, \{0\}, n)>0, \eta_0R^{-8}>\e_R=\e_R(\Omega, \{0\}, n, \varphi, R)>0$ sufficiently small, $R_0=R_0(\Omega)>0$ sufficiently large, for any $0<\e<\e_R,$ $R>R_0,$ and $\varphi\in C^{\infty}(\bar\Omega)$ that satisfies $\varphi\leq0$ and $\|\varphi\|_{C^2}\leq \ev$
let $\ure$ be the admissible solution of \eqref{eq-appr}. Then there exist positive constants $\bb, \bb_1, \bb_2, \ba, \ba_1, \ba_2>0$ that depend on $\Omega$, $\{0\},$ and $n$ such that
on $\ol{\Omega\setminus B_{1/R}},$ $\ure$ satisfies
\be\label{c0-ab}
\bb r^{\frac{n-2}{n-1}}\leq\ure\leq\ba r^{\frac{n-2}{n-1}},
\ee
and
\be\label{c1-ab}
\bb_1 r^{-\frac{1}{n-1}}\leq|D\ure|\leq\ba_1 r^{-\frac{1}{n-1}}.
\ee
Moreover, denote the eigenvalues of Hessian of $D^2\ure$ by $\la(D^2\ure)=(\la_1, \cdots, \la_n)$ and assume $\la_1\geq\la_2\geq\cdots\geq\la_n,$ then we have
\be\label{c2-ab}
\bb_2 r^{-\frac{n}{n-1}}\leq\la_{n-1}\leq\la_1\leq\ba_2 r^{-\frac{n}{n-1}}\,\,\mbox{and}\,\,-\ba_2 r^{-\frac{n}{n-1}}\leq\la_n\leq-\bb_2 r^{-\frac{n}{n-1}}.
\ee
\end{theorem}
A direct rescaling yields
\begin{corollary}
\label{corollary-key}
 Given $\eta_0=\eta_0(\Omega)>0, \ev=\ev(\Omega, \{0\}, n)>0, \eta_0R^{-8}>\e_R=\e_R(\Omega, \{0\}, n, \varphi, R)>0$ sufficiently small, $R_0=R_0(\Omega)>0$ sufficiently large, for any $0<\e<\e_R,$ $R>R_0,$ and $\varphi\in C^{\infty}(\bar\Omega)$ that satisfies $\varphi\leq0$ and $\|\varphi\|_{C^2}\leq \ev$ let $\hure$ be the admissible solution of \eqref{eq-appr-s}.
Then there exist positive constants $\bb, \bb_1, \bb_2, \ba, \ba_1, \ba_2>0$ that depend on $\Omega,$ $\{0\},$ and $n$ such that
on $\ol{\Omega\setminus B_{1/R}},$ $\hure$ satisfies
\be\label{c0-ab-s}
\bb r^{\frac{n-2}{n-1}}\leq\hure\leq\ba r^{\frac{n-2}{n-1}},
\ee
and
\be\label{c1-ab-s}
\bb_1 r^{-\frac{1}{n-1}}\leq|D\hure|\leq\ba_1 r^{-\frac{1}{n-1}}.
\ee
Moreover, denote the eigenvalues of Hessian of $D^2\hure$ by $\la(D^2\hure)=(\hat\la_1, \cdots, \hat\la_n)$ and assume
$\hat\la_1\geq\hat\la_2\geq\cdots\geq\hat\la_n,$ then we have
\be\label{c2-ab-s}
\bb_2 r^{-\frac{n}{n-1}}\leq\hat\la_{n-1}\leq\hat\la_1\leq\ba_2 r^{-\frac{n}{n-1}}
\,\,\mbox{and}\,\,-\ba_2 r^{-\frac{n}{n-1}}\leq\hat\la_n\leq-\bb_2 r^{-\frac{n}{n-1}}..
\ee
\end{corollary}
\begin{remark}
\label{rmk-key}
We emphasize that \eqref{c2-ab} yields that $F$ is uniformly elliptic with respect to $\ure$ (same is true for $\hure$). That is, there exist positive constants $\la$ and $\Lambda$ such that
 \[\la|\xi|^2\leq F^{ij}|_{\ure}\xi_i\xi_j\leq\Lambda|\xi|^2\,\,\mbox{for all $\xi\in\mathbb R^n$}.\]
More presicely, by the concavity of $F,$ the eigenvalues of $(F^{ij}|_{\ure})=:(\df^1, \df^2, \cdots, \df^n)$ satisfy $\df^1\leq\cdots\leq\df^n.$
In view of \eqref{fk}, \eqref{gen-*}, and Theorem \ref{theorem-key} we get
\[\begin{aligned}
\df^1&=\frac{\s_{n-2}(\la|1)-h\s_{n-3}(\la|1)}{\s_{n-2}(\la)}\\
&\geq-\frac{\s_{n-1}(\la|1)}{\la_1\s_{n-2}}-C\frac{\e}{r}\\
&\geq C\frac{\C}{\la_1^n}\geq \la(\Omega, \{0\}, n)>0.
\end{aligned}\]
On the other hand, it is easy to check that $\df^n<2.$ Therefore, $F$ is uniformly elliptic. Moreover, both the minimum and the maximum eigenvalues of
$(F^{ij})$ are independent of $\e$ and $R.$
\end{remark}

Theorem \ref{theorem-key} and Corollary \ref{corollary-key} imply a much stronger version of Theorem \ref{solve-appr-thm}. In particular, we have
\begin{theorem}
\label{thm-cm-est}
Given $\eta_0=\eta_0(\Omega)>0, \ev=\ev(\Omega, \{0\}, n)>0, \eta_0R^{-8}>\e_R=\e_R(\Omega, \{0\}, n, \varphi, R)>0$ sufficiently small, $R_0=R_0(\Omega)>0$ sufficiently large, for any $0<\e<\e_R,$ $R>R_0,$ and $\varphi\in C^{\infty}(\bar\Omega)$ that satisfies $\varphi\leq0$ and $\|\varphi\|_{C^2}\leq \ev$ there exists a unique admissible solution
$\ure\in C^{\infty}(\ol{\Omega\setminus B_{1/R}})$
satisfying \eqref{eq-appr}. Moreover, for any integer $m\geq 0$ we have
\be\label{cm-est}
|D^m\ure(x)|<C_m|x|^{\frac{n-2}{n-1}-m},
\ee
for some $C_m=C_m(\Omega, \{0\}, n)$ when $m=0, 1, 2,$ and $C_m=C_m(m, \Omega, \{0\}, \varphi, n)>0$ when $m\geq 3.$

Analogously, there exists a unique admissible solution
$\hure\in C^{\infty}(\ol{\Omega^R\setminus B_1})$
satisfying \eqref{eq-appr-s}. Moreover, for any integer $m\geq 0$ we have
\be\label{cm-est-s}
|D^m\hure(x)|<C_m|x|^{\frac{n-2}{n-1}-m},
\ee
for some $C_m=C_m(\Omega, \{0\}, n)$ when $m=0, 1, 2,$ and $C_m=C_m(m, \Omega, \{0\}, \varphi, n)>0$ when $m\geq 3.$
\end{theorem}
\begin{proof}
First, it is clear that when $m=0, 1, 2,$ the conclusion follows from Theorem \ref{theorem-key} and Corollary \ref{corollary-key}. We only need to consider the case when
$m\geq 3.$
For our convenience, in the neighborhood of the singularity $\{0\}$ (that is, in the neighborhood of $B_{1/R}(0)$), we will study the rescaled solution $\hure.$

\textbf{Region 1.} We denote $U_1:=B_{R/2}\setminus\bar B_{2.5}(0).$ In this proof, for an arbitrary point $x\in\bar U_1,$ we denote $|x|=:L.$ Let
$$u^L(y):=\lt(\frac{L}{2}\rt)^{-\frac{n-2}{n-1}}\hure\lt(x+\frac{L}{2}y\rt),\,\, |y|\leq 1.$$
By our assumption that $B_1\ssubset\Omega$ (this is equivalent to $B_R\ssubset\Omega^R$) we know $u^L(y)$ is well defined in $\bar B_1(0).$
Furthermore, by Corollary \ref{corollary-key} we get
\[|D^2u^L(y)|=\lt|\lt(\frac{L}{2}\rt)^{\frac{n}{n-1}}D^2\hure\lt(x+\frac{L}{2}y\rt)\rt|\leq C\,\,\mbox{in $\bar B_1(0)$}\]
for some $C=C(\Omega, \{0\}, n)>0.$
It is easy to see that
\be\label{smooth-1}
F(D^2u^L(y))=\lt(\frac{L}{2}\rt)^{\frac{n}{n-1}}\frac{\e}{\lt(\frac{n-1}{n-2}\rt)\lt|x+\frac{L}{2}y\rt|
\lt(\lt|x+\frac{L}{2}y\rt|+\e\rt)^{\frac{1}{n-1}}\lt(\frac{n}{2}\lt|x+\frac{L}{2}y\rt|+(n-1)\e\rt)}
=:f^L.
\ee
By Remark \ref{rmk-key} we know $F$ is uniformly elliptic with respect to $u^L.$ Applying the Evans-Krylov theorem (see \cite{CW98, Evans, K85, GT83}),
we have
\[\|u^L\|_{C^{2, \alpha}(B_{3/4})}\leq C\]
for some $C=C(\Omega, \{0\}, n).$ We note that in fact by Evans-Krylov theorem $$C=C(n, \lambda, \Lambda, \alpha, \|f^L\|_{C^{\alpha}(B_1)}, \|u^L\|_{C^2(B_1)}).$$
However, under our assumptions we have $\|f^L\|_{C^\alpha(B_1)}\ll 1$ and everything else depends on $\Omega,$ $\{0\},$ and $n.$ Thus,  we obtain $C=C(\Omega, \{0\}, n).$
Differentiating equation \eqref{smooth-1} and applying the Schauder estimates for linear, uniformly elliptic equations, we obtain for any integer $m\geq 3$
\[|D^mu^L(0)|\leq C_m,\]
for the same reason as before, here $C_m=C_m(m, \Omega, \{0\}, n)>0.$
It is easy to see that this yields \eqref{cm-est-s} for $x\in\bar U_1.$

\textbf{Region 2.} We denote $U_2:=B_3(0)\setminus\bar B_1(0).$ Note that by the discussion in Region 1 we know $\hure|_{\p B_3}$ is a smooth function with $\|D^m\hure\|_{C^0(\p B_3)}\leq C_m$ for any integer $m\geq 0.$ Let $h\in C^{\infty}(\bar U_2)$ be an arbitrary smooth function satisfies
\[h=\hure\,\,\mbox{on $\p B_3(0)$ and $h=\ba$ on $\p B_1(0)$}\] Then by Theorem \ref{regularity},
we have
\[\|\hure\|_{C^{2, \alpha}(\bar U_2)}\leq C\]
for some $C=C(\Omega, \{0\}, n)>0.$ Again, here
$$C=C(n, \lambda, \Lambda, \alpha, \p U_2, \|f_\e\|_{C^{\alpha}(\bar U_2)}, \|h\|_{C^{2, \alpha}(\bar U_2)}, \|\hure\|_{C^2(\bar U_2)}).$$
However, $\p U_2=\p B_3\cup\p B_1$ is fixed, $\|f_\e\|_{C^\alpha(\bar U_2)}\ll 1,$ and everything else depends on $\Omega,$ $\{0\},$ and $n,$ we obtain $C=C(\Omega, \{0\}, n).$
Same as in Region 1, differentiating equation $F=f_\e$ and applying the Schauder estimates for linear, uniformly elliptic equations, we obtain for any integer $m\geq 3,$
\[\|D^m\hure\|_{C^0(\bar U_2)}\leq C_m,\]
for some $C_m=C_m(m, \Omega, \{0\}, n)>0.$ Therefore, we prove \eqref{cm-est-s} for $x\in\bar U_2.$

\textbf{Region 3.} We denote $U_3:=\Omega^R\setminus\bar B_{R/3}(0).$ In this region, we shall consider $\ure=R^{-\frac{n-2}{n-1}}\hure(Rx)$ instead. In particular, $\ure$ satisfies
\[
\left\{\begin{aligned}
F(D^2\ure)&=f_{\frac{\e}{R}}\,\,&\mbox{in $\Omega\setminus\bar B_{1/3}$}\\
\ure&=1+\varphi\,\,&\mbox{on $\p\Omega,$}\\
\ure&=R^{-\frac{n-2}{n-1}}\hure(Rx)\,\,&\mbox{on $\p B_{1/3}.$}
\end{aligned}
\right.
\]
Note that by the discussion in Region 1 we know $\ure|_{\p B_{1/3}}=R^{-\frac{n-2}{n-1}}\hure(Rx)|_{\p B_{1/3}}$ is a smooth function with $\|D^m\ure\|_{C^0(\p B_{1/3})}\leq C_m$ for any integer $m\geq 0.$ Following the same argument as in Region 2 we obtain for any integer $m\geq 3,$
\[\|D^m\ure\|_{C^0(\ol{\Omega\setminus B_{1/3}})}\leq C_m,\]
for some $C_m=C_m(m, \Omega, \{0\}, n, \varphi)>0.$ This proves \eqref{cm-est} in $\ol{\Omega\setminus B_{1/3}},$ which implies \eqref{cm-est-s} in $\bar U_3.$

Combining estimates above, we complete the proof of \eqref{cm-est-s}, and \eqref{cm-est} then follows by rescaling.
\end{proof}

Now, let $\e_1, \e_2$ be any arbitrary numbers satisfying $0<\e_1<\e_2<\e_R,$ the standard maximum principle implies $u^{R, \e_1}\geq u^{R, \e_2}.$ That is, $\{u^{R, \e}\}$ is a sequence of decreasing functions with respect to $\e$. Applying Theorem \ref{theorem-key} and Theorem \ref{thm-cm-est}, we know as $\e\goto 0,$
\[\ure\goto u^R\,\,\mbox{in $C^{\infty}(\ol{\Omega\setminus B_{1/R}})$},\]
and $u^R$ is the unique admissible solution of the following equation
\be\label{eq-main-R}
\left\{\begin{aligned}
F(D^2u)&=0\,\,&\mbox{in $\overline{\Omega\setminus B_{1/R}}$}\\
u&=1+\varphi\,\,&\mbox{on $\p\Omega,$}\\
u&=\ba R^{-\frac{n-2}{n-1}}\,\,&\mbox{on $\p B_{1/R}.$}
\end{aligned}
\right.
\ee
We note that the uniqueness comes from the standard maximum principle: If $v$ is another admissible solution of \eqref{eq-main-R}.
Consider $w=u^R-v,$ then  $w$ satisfies $\hat F^{ij}w_{ij}= 0,$ thus $w\equiv 0.$ Here, $\hat F^{ij}=\int_0^1F^{ij}(v+sw)ds.$
Moreover, $u^R$ satisfies \eqref{c0-ab}. \eqref{c1-ab}, \eqref{c2-ab}, and \eqref{cm-est}.

Next, we denote
\[
\tilde u^R(x):=\left\{\begin{aligned}
&u^{R}(x)\,\,&\mbox{in $\ol{\Omega\setminus B_{1/R}}$}\\
&\ba R^{-\frac{n-2}{n-1}}\,\,&\mbox{in $B_{1/R}.$}
\end{aligned}
\right.
\]

\textbf{Claim:} $\td u^R\goto u(x)$ in $C^{\infty}_{loc}\lt(\bar\Omega\setminus\{0\}\rt)$ as $R\goto\infty.$

\textbf{Proof of the claim:} Let $\{R_i\}_{i=1}^\infty$ be any sequence of real number in $(R_0, +\infty),$ for $R_0>0,$
chosen in Theorem \ref{theorem-key} such that $R_i\goto\infty$ as $i\goto\infty.$ Clearly, the sequence $\{\|\td u^{R_i}\|_{C^0(\bar\Omega)}\}_{i=1}^\infty$
is bounded. For any compact set $K\subset\bar\Omega\setminus\{0\}$ and any integer $m>0,$ by Theorem \ref{thm-cm-est} we know
$\{\td u^{R_i}\}_{i\geq i_0}$ is bounded in $C^m(K)$ when $i_0>0$ is large enough such that $K\subset\overline{\Omega\setminus B_{1/R_{i_0}}}$. By virtue of Arzel\`a-–Ascoli theorem, up to extraction of a subsequence, we know
$\td u^{R_i}$ converges as $i\goto\infty$ in $C^\infty_{loc}(\bar\Omega\setminus\{0\})$ to a $C^\infty(\bar\Omega\setminus\{0\})$ admissible solution of \eqref{eq-main1}.
From the standard maximum principle we also know that the admissible solution of \eqref{eq-main1} is unique. Therefore, the limit does not depend on the sequence
$\{R_i\}_{i=1}^\infty$ and the claim follows.

We conclude
\begin{theorem}
\label{thm-existence}
 Let $\Omega$ be a bounded, strictly convex, smooth domain in $\mathbb R^n$  satisfying $B_1(0)\ssubset\Omega.$ Then there exists a constant
 $\e_{\varphi}=\e_{\varphi}(\Omega, \{0\}, n)>0,$
 such that for any given non-positive function $\varphi\in C^{\infty}(\bar\Omega)$ with $\|\varphi\|_{C^2(\bar\Omega)}\leq\e_\varphi,$ there exists a unique smooth admissible solution
 $u\in C^{\infty}(\bar\Omega\setminus\{0\})\cap C^0(\bar\Omega)$ of \eqref{eq-main1} satisfying for any $x\in \bar\Omega\setminus\{0\}$
 \be\label{c0-u}
\bb |x|^{\frac{n-2}{n-1}}\leq u(x)\leq\ba |x|^{\frac{n-2}{n-1}},
\ee
and
\be\label{c1-u}
\bb_1 |x|^{-\frac{1}{n-1}}\leq|Du(x)|\leq\ba_1 |x|^{-\frac{1}{n-1}}.
\ee
Denote the eigenvalues of Hessian of $D^2u(x)$ by $\la(D^2u(x))=(\la_1, \cdots, \la_n)$ and assume $\la_1\geq\la_2\geq\cdots\geq\la_n,$ then we have
\be\label{c2-u}
\bb_2 |x|^{-\frac{n}{n-1}}\leq\la_{n-1}\leq\la_1\leq\ba_2 |x|^{-\frac{n}{n-1}}\,\,\mbox{and}\,\,-\ba_2 |x|^{-\frac{n}{n-1}}\leq\la_n\leq-\bb_2 |x|^{-\frac{n}{n-1}}.
\ee
Moreover, for any integer $m\geq 3$ we have
\be\label{cm-u}
|D^m u(x)|<C_m|x|^{\frac{n-2}{n-1}-m}.
\ee
Here, $\bb, \bb_1, \bb_2, \ba, \ba_1, \ba_2>0$ depend on $\Omega$, $\{0\},$ and $n$;  and $C_m>0$ is some positive constant that depend on $m, \Omega, \{0\},$ $n,$ and $\varphi.$
\end{theorem}

\section{Solvability of DPPD}
\label{sec-DPPD}
In this section, we shall prove Theorem \ref{thm-main}. For our convenience, we shall always let the singular point $z$ be the origin and we shall always assume $B_1(0)\ssubset\Omega.$ In view of Remark \ref{appr-rmk-1}, we may also assume $\varphi\in C^{\infty}(\bar\Omega)$ satisfying $\varphi\leq 0$ and
$\|\varphi\|_{C^2(\bar\Omega)}\leq\e_\varphi$ for some $\ev=\e_\varphi(\Omega, \{0\}, n)>0.$ Now, by Theorem \ref{thm-existence},
we know there exists an admissible solution $u$ of \eqref{eq-main1}.
Moreover, by \eqref{c0-u} we know the density of $u$ satisfies
\[\Theta(u, \{0\})=c_1\in[\bb, \ba].\]
Therefore, $u-1$ solves the following DPPD
\[
\left\{\begin{aligned}
\frac{\s_{n-1}}{\s_{n-2}}(D^2u)&=0\,\,&\mbox{in $\bar\Omega\setminus\{0\},$}\\
u&=\varphi\,\,&\mbox{on $\p\Omega,$}\\
\Theta(u, \{0\})&=c_1,
\end{aligned}
\right.
\]

In the following, we shall show that for any $c>c_1$ there exists an admissible solution of the DPPD
\be\label{DPPD-main}
\left\{\begin{aligned}
\frac{\s_{n-1}}{\s_{n-2}}(D^2u)&=0\,\,&\mbox{in $\bar\Omega\setminus\{0\},$}\\
u&=\varphi\,\,&\mbox{on $\p\Omega,$}\\
\Theta(u, \{0\})&=c,
\end{aligned}
\right.
\ee

First, we want to point out that, in view of Theorem \ref{thm-existence}, we know that as long as $\varphi$ satisfies
$\|\varphi\|_{C^2(\bar\Omega)}\leq\e_\varphi$, equation \eqref{eq-main1} admits a solution. In other words, for any $\la\geq 1$ there exists an admissible solution $\tilde u^\la$ of the following equation
\[
\left\{\begin{aligned}
\frac{\s_{n-1}}{\s_{n-2}}(D^2u)&=0\,\,&\mbox{in $\bar\Omega\setminus\{0\},$}\\
u&=1+\frac{\varphi}{\la}\,\,&\mbox{on $\p\Omega,$}\\
u(0)&=0.
\end{aligned}
\right.
\]
It worths mention that $\tilde u^\la$ satisfies estimates \eqref{c0-u}, \eqref{c1-u}, \eqref{c2-u}, and \eqref{cm-u}. Moreover, $\Theta(\tilde u^{\la}, \{0\})=:\tilde c_{\la}\in[\bb, \ba].$
Denote $u^{\lambda}:=\lambda\tilde u^{\lambda},$ then $u^{\lambda}$ satisfies
\be\label{eq-main-la}
\left\{\begin{aligned}
\frac{\s_{n-1}}{\s_{n-2}}(D^2u)&=0\,\,&\mbox{in $\bar\Omega\setminus\{0\},$}\\
u&=\lambda+\varphi\,\,&\mbox{on $\p\Omega,$}\\
u(0)&=0,
\end{aligned}
\right.
\ee
and $\Theta(u^\la, \{0\})=\la\tilde c_{\la}.$ Therefore, in order to prove Theorem \ref{thm-main} we only need to show that for any given $c>c_1,$ there exists
$\la=\la(c)>1$ such that $\Theta(u^\la, \{0\})=c.$ Then $u^\la-\la$ solves \eqref{eq-main}. Before stating the main theorem in this section, let us state the following comparison theorem which can be derived from the standard maximum principle.
\begin{theorem}
\label{thm-compare}
Let $u, v\in C^{\infty}(\bar\Omega\setminus\{0\})\cap C^0(\bar\Omega)$ be two admissible solutions of $\frac{\s_{n-1}}{\s_{n-2}}(D^2 u)=0$ in $\Omega\setminus\{0\}$ satisfying $u, v>0$
in $\bar\Omega\setminus\{0\}$ and $u(0)=v(0)=0.$
Denote
\[\tilde\ga:=\sup\limits_{x\in\p\Omega}\frac{u}{v},\,\,\mbox{and}\,\, \hat\ga:=\inf_{x\in\p\Omega}\frac{u}{v}, \]
then we have
\[\hat\gamma v\leq u\leq\tilde\ga v\,\,\mbox{in $\bar\Omega.$}\]
\end{theorem}

We shall prove
\begin{theorem}
\label{thm-DPPD}
Let $u^1$ be the solution of \eqref{eq-main1} and denote $c_1:=\Theta(u^1, \{0\}).$ Then for any $c>c_1$ there exists $\la=\la(c)>1$
such that the solution of \eqref{eq-main-la} which is denoted by $u^\la$ satisfies
\be\label{DPPD-1}
\Theta(u^\la, \{0\})=c.
\ee
\end{theorem}
\begin{proof}
To demonstrate our idea, we shall show there exists $\la>1$ such that $\Theta(u^\la, \{0\})=2c_1.$ For general $c>c_1,$ the process of finding the corresponding
$\la$ is the same.

\textbf{Step 1.} Denote $u_1:=u^1,$ and $\ga_1:=1,$ then $u_1$ satisfies \eqref{eq-main-la} with $\la=\ga_1.$ Moreover, we know
$\Theta(u_1, \{0\})=c_1.$ We shall also denote $\tau_1:=\frac{2c_1}{c_1}=2$

\textbf{Step 2.} Let $\ga_2:=\tau_1\ga_1=2,$ $\tilde u_1=\tau_1u_1,$ then $\tilde u_1$ satisfies the boundary condition
$$\tilde u_1|_{\p\Omega}=\tau_1(\ga_1+\varphi)=\ga_2+\tau_1\varphi.$$
Moreover, $\Theta(\tilde u_1, \{0\})=2c_1.$ Now, let $u_2$ be the solution of \eqref{eq-main-la} with $\la=\ga_2,$
that is $u_2|_{\p\Omega}=2+\varphi.$ From now on, for our convenience, we shall assume $-\ev\leq\varphi\leq 0$ on $\p\Omega$ with $-\ev$ and $0$ being achieved.
Then we have
\[\inf_{\p\Omega}\frac{u_2}{\td u_1}=1\]
and
\[\sup_{\p\Omega}\frac{u_2}{\td u_1}=\frac{1-\frac{1}{2}\ev}{1-\ev}=1+\frac{(1-\tau_1^{-1})\ev}{1-\ev}:=1+a_1.\]
By Theorem \ref{thm-compare} we get
\[\td u_1\leq u_2\leq (1+a_1)\td u_1\,\,\mbox{in $\bar\Omega.$}\]
This yields
\[2c_1\leq\Theta(u_2, \{0\})\leq 2(1+a_1)c_1.\]
If $\Theta(u_2, \{0\})=2c_1,$ then we are done. Otherwise, denote $\Theta(u_2, \{0\})=:2(1+\beta_1)c_1$ for some
$\beta_1\in(0, a_1].$ By our assumption that $\ev<1/20$ we have $\beta_1<\ev.$ We shall denote $\tau_2:=\frac{2c_1}{2(1+\beta_1)c_1}=\frac{1}{1+\beta_1}.$ We note that $\ga_2-\ga_1=1.$

\textbf{Step 3.} Let $\ga_3=\tau_2\ga_2,$ $\tilde u_2=\tau_2u_2,$ then $\tilde u_2$ satisfies the boundary condition
$$\tilde u_2|_{\p\Omega}=\tau_2(\ga_2+\varphi)=\ga_3+\tau_2\varphi.$$
Moreover, $\Theta(\tilde u_2, \{0\})=2c_1.$ Now, let $u_3$ be the solution of \eqref{eq-main-la} with $\la=\ga_3.$ Then we have
\[\sup_{\p\Omega}\frac{u_3}{\tilde u_2}=1\]
and
\[\inf_{\p\Omega}\frac{u_3}{\tilde u_2}=\frac{1-\ev/\ga_3}{1-\ev/\ga_2}=1-\frac{\beta_1\ev}{\gamma_2-\ev}=:1-a_2.\]
By Theorem \ref{thm-compare} we get
\[(1-a_2)\tilde u_2\leq u_3\leq\tilde u_2\,\,\mbox{in $\bar\Omega.$}\]
This yields
\[2(1-a_2)c_1\leq\Theta(u_3, \{0\})\leq 2c_1.\]
If $\Theta(u_3, \{0\})=2c_1,$ then we are done. Otherwise, denote $\Theta(u_3, \{0\})=:2(1-\beta_2)c_1$ for some
$\beta_2\in(0, a_2].$ By our assumption that $\ev<1/20$ we have $\beta_2<\ev\beta_1<\ev^2.$

We shall denote $\tau_3=\frac{1}{1-\beta_2}.$ We note that $0<\ga_2-\ga_3=\frac{\beta_1}{1+\beta_1}\ga_2<4\ev.$ We also note that
by $\beta_1<\ev$ we get $\ga_3=\tau_2\tau_1\ga_1=\frac{2}{1+\beta_1}\ga_1>\ga_1.$

\textbf{Step 4.} Let $\ga_4=\tau_3\ga_3,$ $\tilde u_3=\tau_3u_3,$ then $\tilde u_3$ satisfies the boundary condition
$$\tilde u_3|_{\p\Omega}=\tau_3(\ga_3+\varphi)=\ga_4+\tau_3\varphi.$$
Moreover, $\Theta(\tilde u_3, \{0\})=2c_1.$ Now, let $u_4$ be the solution of \eqref{eq-main-la} with $\la=\ga_4.$ Then we have
\[\sup_{\p\Omega}\frac{u_4}{\tilde u_3}=\frac{1-\ev/\ga_4}{1-\ev/\ga_3}=1+\frac{\beta_2\ev}{\ga_3-\ev}=:1+a_3\]
and
\[\inf_{\p\Omega}\frac{u_4}{\tilde u_3}=1.\]
By Theorem \ref{thm-compare} we get
\[\tilde u_3\leq u_4\leq(1+a_3)\tilde u_3\,\,\mbox{in $\bar\Omega.$}\]
This yields
\[2c_1\leq\Theta(u_4, \{0\})\leq 2(1+a_3)c_1.\]
If $\Theta(u_4, \{0\})=2c_1,$ then we are done. Otherwise, denote $\Theta(u_4, \{0\})=:2(1+\beta_3)c_1$ for some
$\beta_3\in(0, a_3].$ By our assumption that $\ev<1/20$ we have $\beta_3<\ev\beta_2<\ev^3.$

We shall denote $\tau_4=\frac{1}{1+\beta_3}.$ We note that $0<\ga_4-\ga_3=\frac{\beta_2}{1-\beta_2}\ga_3<4\ev^2.$
We also note that by $\beta_2<\ev\beta_1$ we get $\ga_4=\tau_3\tau_2\ga_2=\frac{\ga_2}{1-\beta_2+\beta_1-\beta_2\beta_1}<\ga_2.$

\textbf{Step 5.} Let $\ga_5=\tau_4\ga_4,$ $\tilde u_4=\tau_4u_4,$ then $\tilde u_4$ satisfies the boundary condition
$$\tilde u_4|_{\p\Omega}=\tau_4(\ga_4+\varphi)=\ga_5+\tau_4\varphi.$$
Moreover, $\Theta(\tilde u_4, \{0\})=2c_1.$ Now, let $u_5$ be the solution of \eqref{eq-main-la} with $\la=\ga_5.$ Then we have
\[\sup_{\p\Omega}\frac{u_5}{\tilde u_4}=1\]
and
\[\inf_{\p\Omega}\frac{u_5}{\tilde u_4}=\frac{1-\ev/\ga_5}{1-\ev/\ga_4}=1-\frac{\beta_3\ev}{\ga_4-\ev}=:1-a_4.\]
By Theorem \ref{thm-compare} we get
\[(1-a_4)\tilde u_4\leq u_5\leq\tilde u_4\,\,\mbox{in $\bar\Omega.$}\]
This yields
\[2(1-a_4)c_1\leq\Theta(u_5, \{0\})\leq 2c_1.\]
If $\Theta(u_5, \{0\})=2c_1,$ then we are done. Otherwise, denote $\Theta(u_5, \{0\})=:2(1-\beta_4)c_1$ for some
$\beta_4\in(0, a_4].$ By our assumption that $\ev<1/20$ we have $\beta_4<\ev\beta_3<\ev^4.$

We shall denote $\tau_5=\frac{1}{1-\beta_4}.$ We note that $0<\ga_4-\ga_5=\frac{\beta_3}{1+\beta_3}\ga_4<4\ev^3.$
We also note that by $\beta_3<\ev\beta_2$ we get $\ga_5=\tau_4\tau_3\ga_3=\frac{\ga_3}{1-\beta_2+\beta_3-\beta_2\beta_3}>\ga_3.$

\textbf{Step 6.} Let $\ga_6=\tau_5\ga_5,$ $\tilde u_5=\tau_5u_5,$ then $\tilde u_5$ satisfies the boundary condition
$$\tilde u_5|_{\p\Omega}=\tau_5(\ga_5+\varphi)=\ga_6+\tau_5\varphi.$$
Moreover, $\Theta(\tilde u_5, \{0\})=2c_1.$ Now, let $u_6$ be the solution of \eqref{eq-main-la} with $\la=\ga_6.$ Then we have
\[\sup_{\p\Omega}\frac{u_6}{\tilde u_5}=\frac{1-\ev/\ga_6}{1-\ev/\ga_5}=1+\frac{\beta_4\ev}{\ga_5-\ev}=:1+a_5\]
and
\[\inf_{\p\Omega}\frac{u_6}{\tilde u_5}=1.\]
By Theorem \ref{thm-compare} we get
\[\tilde u_5\leq u_6\leq(1+a_5)\tilde u_5\,\,\mbox{in $\bar\Omega.$}\]
This yields
\[2c_1\leq\Theta(u_6, \{0\})\leq 2(1+a_5)c_1.\]
If $\Theta(u_6, \{0\})=2c_1,$ then we are done. Otherwise, denote $\Theta(u_6, \{0\})=:2(1+\beta_5)c_1$ for some
$\beta_5\in(0, a_5].$ By our assumption that $\ev<1/20$ we have $\beta_5<\ev\beta_4<\ev^5.$

We shall denote $\tau_6=\frac{1}{1+\beta_5}.$ We note that $0<\ga_6-\ga_5=\frac{\beta_4}{1-\beta_4}\ga_5<4\ev^4.$
We also note that by $\beta_4<\ev\beta_3$ we get $\ga_6=\tau_5\tau_4\ga_4=\frac{\ga_4}{1-\beta_4+\beta_3-\beta_4\beta_3}<\ga_4.$

$\cdots\cdots$

By repeating the above process, we can see that we either find a $\ga_N$ in finite steps such that the solution of \eqref{eq-main-la} with $\la=\ga_N,$
denoted by $u_N,$ satisfies $\Theta(u_N, \{0\})=2c_1,$ or we obtain a sequence $\{\ga_i\}_{i=1}^\infty.$ From the discussion above, it is clear that
$|\ga_i-\ga_{i+1}|<4\ev^{i-1}$ for all $i\in\mathbb N,$ thus $\{\ga_i\}_{i=1}^\infty$ is a cauchy sequence and we denote $\ga_\infty:=\lim\limits_{i\goto\infty}\ga_i.$
Let $u_\infty$ be the solution of \eqref{eq-main-la} with $\la=\ga_\infty.$

\textbf{Claim:} $\Theta(u_\infty, \{0\})=2c_1.$

\textbf{Proof of the claim:} Note that for any $n\in\mathbb N$ we have
\[\Theta(u_{2n+1}, \{0\})=2(1-\beta_{2n})c_1\,\,\mbox{for some $0<\beta_{2n}<\ev^{2n}$}\]
and
\[\Theta(u_{2n}, \{0\})=2(1+\beta_{2n-1})c_1\,\,\mbox{for some $0<\beta_{2n-1}<\ev^{2n-1}.$}\]
This gives for $n\geq 2,$
\[|\Theta(u_n, \{0\})-2c_1|<2c_1\ev^{n-1}.\]
Therefore, $\lim\limits_{n\goto\infty}\Theta(u_n, \{0\})=2c_1.$

On the other hand, we know that $\{\ga_{2n+1}\}_{n=1}^\infty$ is an increasing sequence with $\lim\limits_{n\goto\infty}\ga_{2n+1}=\ga_{\infty}$
and $\{\ga_{2n}\}_{n=1}^\infty$ is a decreasing sequence with $\lim\limits_{n\goto\infty}\ga_{2n}=\ga_{\infty}.$
By the standard maximum principle we know for all $n\geq 1,$ $u_{2n+1}\leq u_\infty\leq u_{2n}$ in $\bar\Omega.$
This implies
\[\Theta(u_{2n+1}, \{0\})\leq\Theta(u_\infty, \{0\})\leq\Theta(u_{2n}, \{0\})\,\,\mbox{for all $n\geq 1.$}\]
Letting $n\goto\infty$ we conclude $\Theta(u_{\infty}, \{0\})=2c_1.$ This finishes the proof of the claim. Consequently, we prove this theorem.
\end{proof}

\appendix
\section {Regularized distance function}
\label{app}
It has been well understood that distance function is a good tool to construct barriers for solving Dirichlet problems (see for example, \cite{CNS3, GT83}).
However, for most of domains $\Omega,$ the distance function to the boundary of $\Omega$ (i.e., $\partial\Omega$) is only well defined in a small neighborhood of $\partial\Omega.$
Therefore, sometimes it is more convenient to use regularized distance function to construct barriers. Inspired by \cite{Lie85}, we have following results.

\begin{definition}
\label{def-reg-dist}
Let $\Omega\subset\R^n$ be
an open set that contains the origin $\{0\}$ and has non-empty boundary $\partial\Omega$. A function $\vartheta$ is called a {\it\textbf{regularized distance}} for $\Omega$
if $\vartheta\in C^2(\mathbb R^n\setminus\{0\})\cup C^{0, 1}(\mathbb R^n)$
and if the ratios $\vartheta(x)/d(x)$ and $d(x)/\vartheta(x)$ are positive and uniformly bounded for all $x\in\R^n\setminus\partial\Omega.$
Here, $d(x)$ is the signed distance to $\p\Omega.$
\end{definition}
We note that although our definition of the regularized distance differs from the one in \cite{Lie85}, they share an important common property: the differentiability of the regularized distance does not depend on the cut locus of $\p\Omega.$  In general, regularized distance functions cannot be used to construct barriers.
However, in \cite{Xiao22} the author discovered that when $\Omega$ is a star-shaped domain, one can find a regularized distance that has very nice properties, and thus can be used to construct barriers.

For readers' convenience, we include the Section 3, which contains the construction of a regularized distance function for star-shaped domains, from \cite{Xiao22} here.

Recall that when $\Omega\subset\R^n$ is a star-shaped domain, denote $\Gamma:=\partial\Omega,$ then $\Gamma$ can be parameterized as a graph of the radial function
$\rho(\theta): \mathbb S^{n-1}\goto\mathbb R,$ i.e.,
\[\Gamma=\{\rho(\theta)\theta: \theta\in\mathbb S^{n-1}\}.\]

We denote $\Phi:=\log\rho,$ it is clear that the second fundamental form of $\Gamma$ can be expressed as follows
\[h_{ij}=\frac{\rho}{w}\lt(\delta_{ij}+\Phi_i\Phi_j-\Phi_{i,j}\rt),\]
where $w=\sqrt{1+|\nabla\Phi|^2}=\sqrt{1+\frac{|\nabla\rho|^2}{\rho^2}},$ $\Phi_{i, j}=\nabla_{ij}\Phi,$
and $\nabla$ denotes the Levi-Civita connection on $\mathbb S^{n-1}.$
 By a direct calculation, we obtain
\be\label{add1}
g_{ij}=\rho^2(\delta_{ij}+\Phi_i\Phi_j),\,\,
g^{ij}=\frac{1}{\rho^2}\lt(\delta_{ij}-\frac{\Phi_i\Phi_j}{w^2}\rt),\,\,\text{and}\,\,\gamma^{ij}=\frac{1}{\rho}\lt(\delta_{ij}-\frac{\Phi_i\Phi_j}{w(1+w)}\rt).
\ee
Here, $(g_{ij})$ is the metric on $\Gamma,$ $(g^{ij})$ is the inverse of $(g_{ij}),$ and $\gamma^{ij}$ is the square root of $g^{ij},$
i.e., $\sum\limits_k\gamma^{ik}\gamma^{kj}=g^{ij}.$
Let $a_{ij}=\gamma^{ik}h_{kl}\gamma^{lj},$ then the eigenvalues of $(a_{ij})_{1\leq i, j\leq n-1},$ denoted by $\kappa[a_{ij}]=(\kappa_1, \cdots, \kappa_{n-1})$ are the principal curvatures of $\Gamma.$

The following calculation can be found in Section 3 of \cite{Xiao22}, for readers convenience, we include it here.
\subsubsection{Hessian in spherical coordinates}
\label{sub-A.1}
Let $f: \R^n\goto \R$ be a scalar function, then $f$ can also be expressed as a function of $(\theta, r)\in\mathbb{S}^{n-1}\times\R.$
Note that the Euclidean metric is $g_E=r^2dz^2+dr^2,$ where $dz^2$ is the standard metric on $\mathbb{S}^{n-1}.$ In the following, we will denote the standard connection in $\mathbb R^n$ by $D$.
Now, we choose a local orthonormal frame $\{e_1, \cdots, e_{n-1}\}$ on the unit sphere $\mathbb{S}^{n-1}$. Let $\tau_{a}=\dfrac{e_{a}}{r}$, $1\leq a\leq n-1,$ which is the orthonormal frame on the sphere with radius $r,$ and we also let $\tau_r=\frac{\partial}{\partial r}$.
Then a direct calculation yields the Hessian of $f$ in spherical coordinates is
\be\label{hess1.1}D^2_{ab}f=D^2f(\tau_{a},\tau_{b})=\frac{1}{r^2}f_{ab}+\frac{1}{r}f_r\delta_{ab},\ee
\be\label{hess1.2}D^2_{a r}f=D^2f(\tau_{a},\tau_r)=\frac{1}{r}f_{a r}-\frac{1}{r^2}f_{a},\ee
and
\be\label{hess1.3}D^2_{rr}f=D^2f(\tau_r,\tau_r)=f_{rr}.\ee
Here $1\leq a, b\leq n-1,$ $f_{ab}=e_{b}e_{a}f, f_{a r}=\tau_re_{a} f,$ and $f_{rr}=\tau_r\tau_rf$.

\subsubsection{Hessian of the regularized distance}
\label{sub-A.2}
Now, we fix an arbitrary point $p\in \mathbb{S}^{n-1},$ let $\{e_1, \cdots, e_{n-1}\}$ be the normal coordinates at $p,$ then the Christoffel symbols vanish at
$p.$ This implies at this point we get $\nabla_{ij}\rho=e_ie_j\rho,$ that is, $\rho_{i,j}= \rho_{ij}$ at $p.$ Moreover, we may rotate the coordinates such that $|\nabla\rho(p)|=\rho_1$ and $\rho_{\al\beta}(p)=\rho_{\al\al}\delta_{\al\beta}$
for $2\leq\al,\beta\leq n-1.$ Then at the point $\hat{p}=\rho(p)p\in\Gamma,$ in view of \eqref{add1} we have
\[
\left\{
\begin{aligned}
\gamma^{11}&=\frac{1}{\rho}\lt(1-\frac{w^2-1}{w(1+w)}\rt)=\frac{1}{\rho w},\\
\gamma^{1\al}&=0,\,\,&2\leq\al\leq n-1,\\
\gamma^{\al\beta}&=\frac{1}{\rho}\delta_{\al\beta},\,\,&2\leq\al, \beta\leq n-1,
\end{aligned}
\right.
\]
and
\be\label{second-fundamental-form}
\left\{
\begin{aligned}
a_{11}&=\gamma^{1k}h_{kl}\gamma^{l1}=\gamma^{11}h_{11}\gamma^{11}=\frac{h_{11}}{\rho^2 w^2},\\
a_{1\al}&=\gamma^{1k}h_{kl}\gamma^{l\alpha}=\frac{h_{1\al}}{\rho^2 w},\,\,&2\leq\al\leq n-1,\\
a_{\al\beta}&=\gamma^{\al\al}h_{\al\beta}\gamma^{\beta\beta}=\frac{1}{\rho^2}h_{\al\beta},\,\,&2\leq\al, \beta\leq n-1.
\end{aligned}
\right.
\ee

Now, let us consider the function  $$\mathfrak b=\frac{r}{\rho(\theta)}.$$ Note that $1-\fb$ is a regularized distance for the domain $\Omega.$
we will compute the Hessian of $\fb$ at point $(p, r)\in \mathbb{S}^{n-1}\times\R_+$ under the coordinates chosen above.

A straightforward calculation yields
\[\fb_1=e_1\fb=-r\rho^{-2}\rho_1,\,\,\fb_{11}=e_1e_1 \fb=-r\lt(-2\rho^{-3}\rho_1\rho_1+\rho^{-2}\rho_{11}\rt),\]
\[\fb_{1\al}=-r\rho^{-2}\rho_{1\al}\,\,\mbox{for $2\leq\al\leq n-1,$}\]
\[\fb_\al=e_{\al}\fb=0,\,\,\fb_{\al\beta}=e_{\beta}e_{\al}\fb=-r\rho^{-2}\rho_{\al\beta}\,\,\mbox{for $2\leq\al, \beta\leq n-1,$}\]
and
\[\fb_n:=\frac{\p \fb}{\p r}=\rho^{-1},\,\, \fb_{nn}:=\frac{\p^2 \fb}{\p r^2}=0,\,\, \fb_{1n}:=\frac{\p \fb_1}{\p r}=-\rho^{-2}\rho_1,\,\, \fb_{\al n}:=\frac{\p \fb_{\al}}{\p r}=0.\]
Following the notation in Subsection \ref{sub-A.1}, suppose $\tau_{a}=\frac{e_{a}}{r}, 1\leq a\leq n-1$ and $\tau_n:=\tau_r=\frac{\p}{\p r},$ then $\{\tau_1, \cdots, \tau_{n-1}, \tau_n\}$ forms an orthonormal frame at $(p,r)$. Combining \eqref{hess1.1}, \eqref{hess1.2}, and \eqref{hess1.3} with \eqref{second-fundamental-form} we get, at the point $(p, r)$
\begin{eqnarray}
D^2_{11}\fb&=&D^2\fb(\tau_1,\tau_1)=\frac{1}{r^2}\fb_{11}+\frac{1}{r\rho}=\frac{w^3}{r}a_{11}\label{cs2.1},\\
D^2_{1\al}\fb&=&D^2\fb(\tau_1,\tau_{\al})=\frac{1}{r^2}\fb_{1\al}=\frac{w^2}{r}a_{1\al},\,\,2\leq\al\leq n-1,\label{cs2.2}\\
D^2_{\al\beta}\fb&=&D^2\fb(\tau_{\al},\tau_{\beta})=\frac{1}{r^2}\fb_{\al\beta}+\frac{\delta_{\al\beta}}{r\rho}=\frac{w}{r}a_{\al\beta},\,\, 2\leq\al, \beta\leq n-1,\label{cs2.3}\\
D^2_{ij}\fb&=&D^2\fb(\tau_i,\tau_j)=0,\,\, \mbox{for all other cases.}\nonumber
\end{eqnarray}
We also notice that at $p$ we have
\[h_{\al\beta}=\frac{\rho}{w}\lt(\delta_{\al\beta}-\frac{\rho_{\al\beta}}{\rho}\rt)=h_{\al\al}\delta_{\al\beta},\,\,2\leq\al, \beta\leq n-1,\]
this implies $a_{\al\beta}=a_{\al\al}\delta_{\al\beta}$ is diagonalized.
Therefore, at $(p, r)$ we obtain
\be\label{hessian-b}
\begin{aligned}
\text{Hessian}(\fb)&=\left[\ju{ccccc}{\frac{w^3}{r}a_{11}&\frac{w^2}{r}a_{12}&\cdots&\frac{w^2}{r}a_{1n-1}&0\\
\frac{w^2}{r}a_{12}&\frac{w}{r}a_{22}&\cdots&0&0\\
\vdots&\vdots&\ddots&\vdots\\
\frac{w^2}{r}a_{1n-1}&0&\cdots&\frac{w}{r}a_{n-1n-1}&0\\
0&0&\cdots&0&0}\right].
\end{aligned}
\ee
We want to emphasize that $\la(a_{ij})=(\la_1, \cdots, \la_{n-1})$ are the principal curvatures of $\Gamma$ at $\hat{p}=\rho(p)p.$

In general, consider the function $\phi=\phi(\mathfrak b),$ where $\phi$ is a function defined on $\R.$
We will compute the Hessian of $\phi$ at $(p, r).$ Denote $\phi'|_{(p, r)}=\frac{d\phi}{d\fb}|_{(p, r)}=M,$ $\phi''|_{(p, r)}=\frac{d^2\phi}{d\fb^2}|_{(p, r)}=B,$ we get for
$1\leq i,j\leq n$,
$$D^2_{ij}\phi=MD^2_{ij}\fb+B(\tau_i\fb)(\tau_j\fb).$$
From \eqref{hessian-b}, it is easy to obtain that at $(p, r)$ under the specific coordinates we chose earlier, the hessian of $\phi$ is
 \be\label{hessian-phi}
\begin{aligned}
\text{Hessian}(\phi)&=\left[\ju{ccccc}{\frac{Mw^3}{r}a_{11}+B\rho^{-4}\rho_1^2&\frac{Mw^2}{r}a_{12}&\cdots&\frac{Mw^2}{r}a_{1n-1}&-B\rho^{-3}\rho_1\\
\frac{Mw^2}{r}a_{12}&\frac{Mw}{r}a_{22}&\cdots&0&0\\
\vdots&\vdots&\ddots&\vdots\\
\frac{Mw^2}{r}a_{1n-1}&0&\cdots&\frac{Mw}{r}a_{n-1n-1}&0\\
-B\rho^{-3}\rho_1&0&\cdots&0&B\rho^{-2}}\right].
\end{aligned}
\ee

\end{document}